\documentclass[3p]{article}
\usepackage{graphicx} 
\usepackage{amssymb}
\usepackage{amsthm}
\usepackage{amsmath}
\usepackage{pgfplots}
\usepackage{CJKutf8}
\usepackage{setspace}
\usepackage{titlesec}
\usepackage[
    bookmarks=true,
    bookmarksopen=true,
    bookmarksnumbered=true,
    bookmarksdepth=2
]{hyperref}
\numberwithin{equation}{section}
\allowdisplaybreaks[4]
\date{}

\theoremstyle{plain}

\newtheorem{thm}{Theorem}[section]
\newtheorem{lem}[thm]{Lemma}
\newtheorem{cor}[thm]{Corollary}
\newtheorem{prop}[thm]{Proposition}

\newtheorem*{rem}{Remark}

\newtheorem{defi}[thm]{Definition}

\begin{document}

\title{Sparse domination for $q$-variation associated with functional calculus}
\author{Yanhan Chen \footnote{Department of Mathematics, Graduate School of Science, Kyoto University, Kyoto 606-8502, Japan. \newline e-mail: \texttt{chen.yanhan.67s@st.kyoto-u.ac.jp}}}
\maketitle
\renewcommand{\abstractname}{Abstract}
\begin{abstract}
 We establish sparse domination for the $q$-variation of functional calculus families $\{\varphi(tL)\}_{t>0}$ on spaces of homogeneous type. Under generalized Gaussian estimates, we consider holomorphic multipliers of injective accretive operators, and Sobolev-type spectral multipliers of non-negative self-adjoint operators. For the Sobolev-type multipliers, we show that our regularity assumption is sharp in general except the endpoint. As corollaries, we obtain weighted $q$-variational estimates for Bochner-Riesz means and regularized Schr\"{o}dinger groups. Finally, we give several examples of operators to which our results apply.
\end{abstract}

\tableofcontents
\section{Introduction}
\subsection{Background}
Let $2\leq q<\infty$ and $a:(0,\infty)\rightarrow \mathbb{C}$. Define
$$\|a\|_{{V}^q}
:=
\sup_{t_0<t_1<\cdots}\left\{
\left(
|a(t_0)|^q
+
\sum_{k=1}^{\infty}|a(t_k)-a(t_{k-1})|^q
\right)^{{1}/{q}}\right\},$$
where the supremum is taken over all increasing sequences in $\mathbb{R}_+$. We also define the homogeneous $q$-variation by
$$\|a\|_{\widetilde{V}^q}
:=
\sup_{t_1>t_2>\cdots}\left\{
\left(
\sum_{k=1}^{\infty}|a(t_k)-a(t_{k+1})|^q
\right)^{{1}/{q}}\right\}.$$ These two norms are related by $\|a\|_{{V}^q}\sim\|a\|_{\widetilde{V}^q}+\sup_{t>0}|a(t)|$. For a family of operators $\{T_t\}_{t>0}$, we write the $q$-variation operator
$$V^q(T_t:t>0)f(x):=\|\{T_tf(x)\}_{t>0}\|_{V^q},$$
and defines $\widetilde{V}^q$ similarly. The $q$-variation controls the corresponding maximal operator. Finite $q$-variation implies the existence of limits as $t\to0$ and $t\to\infty$. Thus, variational bounds give pointwise convergence directly, while maximal estimates are usually combined with convergence on a dense subclass. Moreover, variation estimate provide quantitative information about convergence, for example through jump estimates (see \cite{DK24,LMX12} for instance). Variation inequalities have been studied extensively in probability, ergodic theory and harmonic analysis. L\'{e}pingle \cite{L76} obtained $q$-variational inequalities for martingales with $q>2$. Bourgain \cite{B89} developed variational methods in the study of pointwise ergodic theorems. In harmonic analysis, Jones, Seeger and Wright \cite{JSW08} obtained variational and jump inequalities for several classes of operators, including singular integrals and spherical means. Further results for Euclidean averaging operators were obtained by Bourgain, Mirek, Stein and Wr\'{o}bel \cite{BMSW18}. For analytic semigroups, Le Merdy and Xu \cite{LMX12} established $q$-variational estimates under contractive regularity assumptions.

The main goal of this paper is to study the $q$-variation $V^q(\varphi(tL):t>0)$ of functional calculus families. For the classical Laplacian $L=-\Delta$ on $\mathbb{R}^n$, the spectral multiplier is given by Fourier multiplier. In this case, an important approach to maximal and variational multiplier estimates is based on square functions. Stein \cite{S76} introduced square functions associated with Bochner-Riesz means and spherical means to study maximal estimates and pointwise convergence. Later, Lee, Rogers and Seeger \cite{LRS12} improved the $L^p$ estimates for Stein's square functions and obtained corresponding maximal Bochner-Riesz estimates. They also observed that their result imply local $2$-variational estimates for Bochner-Riesz means. In their subsequent work \cite{LRS14} they extended such estimates to more general radial Fourier multipliers.

For more general $L$, we suppose $-L$ generates an analytic semigroup $\{e^{-tL}\}_{t>0}$ on $L^2(X)$. A standard assumption in spectral multiplier theory is that the semigroup admits a kernel $p_{t}(x,y)$ satisfying the Gaussian estimate
\begin{equation}
|p_t(x,y)|
\lesssim
\frac{1}{|B(x,t^{1/m})|}
\exp\left(
-c\left(\frac{d(x,y)}{t^{1/m}}\right)^{m/(m-1)}
\right)\tag{GE$_m$}
\end{equation}
for some $m\geq2$ (see for e.g. \cite{AM06,GSC05,O05}). However, pointwise Gaussian estimate \eqref{GE} may fail for some operators (see for e.g. \cite{D97,LSV02,SV94}). An important extension is provided by generalized Gaussian estimates
\begin{equation}
    \|\chi_{B_2} e^{-tL} \chi_{B_1}\|_{p_0\rightarrow q_0}\leq C|B_{1}|^{-1/p_0}|B_2|^{1/q_0}\,\textnormal{exp}\left(-c\left(\frac{d(B_1,B_2)}{t^{1/m}}\right)^{{m}/{(m-1)}}\right).\tag{GGE$_{p_0,q_0,m}$}
\end{equation}
for every $t>0$, and balls $B_1,B_2$ of radius $t^{1/m}$ (see Section 2.3 for the precise definition). Spectral multiplier theories under \eqref{GGE} have been studied by Blunck \cite{B03} and Sikora, Yan and Yao \cite{SYY14}, among others. Deleaval and Kriegler \cite{DK23} obtained maximal estimates for a Sobolev-type multiplier families based on H\"{o}rmander functional calculus. They later extended the result to $q$-variational estimates \cite{DK24}. More recently, Chen, Lin, Wu and Yan \cite{CLWY26} studied maximal functions associated with H\"{o}rmander-type spectral multipliers under \eqref{GE} and other assumptions involving a Doob transform. We look for a stronger form of control through sparse domination, which provides quantitative weighted estimates for $q$-variation operators.

Sparse domination has played an important role in weighted harmonic analysis. Following Hyt\"{o}nen's resolution of the $A_2$ conjecture \cite{H12}, Lerner \cite{L13} developed a proof based on the local oscillation decomposition. This provides control of Calder\'on-Zygmund operators by positive dyadic operator with the form
$$|Tf(x)|\lesssim \sum_{Q\in\mathcal{S}}\langle f\rangle_{1,Q}\chi_{Q}(x),$$
where $\mathcal{S}$ is a sparse family. Such sparse structure provides sharp quantitative weighted estimates for the Calder\'on-Zygmund operators that
$$\|T\|_{L^p(w)\to L^p(w)}
\lesssim
[w]_{A_p}^{\max\{1,1/(p-1)\}}
$$
for $1<p<\infty$. The sparse domination method has also been extended beyond classical Calder\'{o}n-Zygmund theory. For operators without suitable kernel
regularity estimates, pointwise sparse
domination may be difficult to establish. A natural generalization is to consider the bilinear sparse bounds 
$$|\langle Tf,g\rangle|
\lesssim
\sum_{Q\in\mathcal S}
|Q|\langle f\rangle_{p_0,Q}
\langle g\rangle_{q_0',Q}.$$
We write $T\in S(p_0,q_0)$ if $T$ admits a bilinear sparse bound of this form. The precise definitions of sparse families and sparse bounds are given in Section 2.4. The advantage of this bilinear version of sparse domination is that it applies to operators whose $L^p-$bounedness only holds for $p_0<p<q_0$. Such sparse bounds also holds for various operators, such as Bochner-Riesz multipliers \cite{BBL17} and spherical maximal operators \cite{L17}. Bernicot, Frey and Petermichl \cite{BFP16} established such sparse bounds for non-integral operators under suitable off-diagonal estimates. They also obtained quantitative weighted estimates with $w\in A_{p/p_0}\cap RH_{(q_0/p)^{\prime}}$ for operators admitting such a bilinear sparse bound. Recently, Beltran, Roos and Seeger \cite{BRS24} developed a multi-scale sparse domination framework based on bilinear forms. They applied this framework to Fourier multipliers, maximal operators, square functions and variation operators. Their results also provide further examples of sparse domination beyond classical Calder\'{o}n-Zygmund theory.
\subsection{Main results}
We now turn to the $q$-variation associated with functional calculus. We refer to Section 2.4 for the precise definitions of sparse domination and the class $S(p_0,q_0)$. Our first result concerns the holomorphic functional calculus for injective accretive operators.
\begin{thm}\label{thm1}
    Let $L$ be an injective $\omega$-accretive operator on $L^2(X)$ with $0\leq\omega<\pi/2$. Assume that the semigroup $\{e^{-tL}\}_{t>0}$ satisfies \eqref{DGGE} for some $m\geq 2$ and $1\leq p_0<2<q_0\leq\infty$. Suppose that for some $\omega<\mu<\pi$, 
    $$\varphi(z):=\varphi(0)e^{-z}+\varphi_0(z),\quad z\in \Sigma_{\mu},$$
    where $\varphi_0\in H^{\infty}_0(\Sigma_{\mu}).$
    Then for $q>2$, the $q$-variation $V^q(\varphi(tL):{t>0})$ belongs to $S(p_0,q_0)$.
\end{thm}
A $0$-accretive operator is precisely a non-negative self-adjoint operator. In this case, the spectral theorem allows us to consider more general multiplier classes. Following \cite{DK23,DK24}, we introduce a Sobolev-type spectral multiplier class 
$$\Lambda^{\beta}:=\{\varphi:(0,\infty)\to\mathbb{C}:
\varphi\circ\exp\in W^{\beta,2}(\mathbb{R})\}.$$ 
See Section 2.2 for details. Our second main result is as follows.
\begin{thm}\label{thm2}
    Let $L$ be a non-negative self-adjoint operator on $L^2(X)$, and the semigroup $\{e^{-tL}\}_{t>0}$ satisfies \eqref{DGGE} for some $m\geq 2$ and $1\leq p_0<2<q_0\leq\infty$. Let
    $$\varphi(\lambda):=\varphi(0)e^{-\lambda}+\varphi_{0}(\lambda),\quad \lambda>0,$$
    where $\varphi_0\in \Lambda^{\beta}$ for some $\beta>\nu(1/p_0-1/q_0)+1/2$.  Then for $q>2$, the $q$-variation operator $V^q(\varphi(tL):{t>0})$ belongs to $S(p_0,q_0)$. In particular, for every pair of bounded, compactly supported functions $f,g$, there exist sparse families $\mathcal{S}_b\subset\mathcal{D}^b$, $b=1,...,K_D$, such that
    $$|\langle V^q(\varphi(tL):t>0)f,g\rangle|\lesssim(|\varphi(0)|+\|\varphi_0\|_{\Lambda^{\beta}})\sum_{b=1}^{K_D}\Lambda_{\mathcal{S}_b;p_0,q_0}(f,g).$$
    If $\varphi(0)=0$, the assumption
\textnormal{(DGGE$_{p_0,q_0,m}$)} can be replaced by
\textnormal{(GGE$_{p_0,q_0,m}$)}.
\end{thm}
The regularity condition $\beta>\nu(1/p_0-1/q_0)+1/2$ is sharp in general except the endpoint, even without variation. We will give a counterexample on $\mathbb{R}$ that satisfies \textnormal{(DGGE$_{p_0,q_0,m}$)} in section 5. The corresponding $(p_0,q_0)$ sparse bound can fail for a single spectral multiplier $\varphi(L)$ for every $0\leq\beta<1/p_0-1/q_0+1/2$.

In both theorems, $\varphi_0$ vanishes at zero. The term $\varphi(0)e^{-\lambda}$ allows us to consider multipliers with a non-zero limit at zero. This includes regularized Sch\"{o}dinger group and Bochner-Riesz means (which will be discusses below). Our proofs are based on general sparse domination argument under off-diagonal estimates. For holomorphic multipliers in Theorem \ref{thm1}, we apply the sparse domination criterion of Bernicot, Frey and Petermichl \cite{BFP16}. For Sobolev-type multipliers in Theorem \ref{thm2}, we establish a quantitative sparse domination criterion for Hilbert space-valued operators. The Hilbert space-valued multiplier families $\{\varphi(tL)\}_{t>0}$ commute with the semigroup. This allows us to prove the sparse bound for the $\varphi_0$ part without additional domination assumption. Then, We use a dyadic Fourier decomposition $\varphi=\sum_{j\geq0}\varphi_j$ and apply this criterion to each $\varphi_j$. Combining the resulting sparse bounds with the embedding $\Lambda^\eta\hookrightarrow V^q$ for $\eta>1/2$ (see Lemma \ref{lambdaemb}), we obtain the regularity condition in Theorem \ref{thm2}.

As an immediate consequence of Theorem \ref{thm1}, \ref{thm2} and Lemma \ref{wei-lem}, we obtain the following weighted estimates.

\begin{cor}\label{wei-cor}
    Under the assumptions of either Theorem \ref{thm1} or \ref{thm2}.

    \textnormal{(i)} For $p_0<p<q_0$ and $w\in A_{p/p_0}\cap RH_{(q_0/p)^{\prime}}$, we have that
    $$\|V^q(\varphi(tL):t>0)\|_{L^{p}(w)\rightarrow L^p(w)}\lesssim\left([w]_{A_{p/p_0}}[w]_{RH_{(q_0/p)^{\prime}}}\right)^{\alpha},$$
    where $\alpha=\text{max}\{1/(p-p_0),(q_0-1)/(q_0-p)\}.$
    
     \textnormal{(ii)} Further suppose that $(X,d)$ is a connected metric space. For $w\in A_1\cap RH_{(q_0/p_0)^{\prime}}$, there is a function $\psi(w):=\psi([w]_{A_1},[w]_{A_{\infty}},[w]_{RH_{(q_0/p_0)^{\prime}}})$ \textnormal{(}see \textnormal{\cite{FN19}} for a precise representation\textnormal{)} so that
    $$\|V^q(\varphi(tL):t>0)\|_{L^{p_0}(w)\rightarrow L^{p_0,\infty}(w)}\lesssim\psi(w).$$
\end{cor}
Under the assumptions of Corollary \ref{wei-cor} \textnormal{(i)}, a.e. finite $q$-variation implies that $\varphi(tL)f(x)$ has limits as $t\rightarrow0$ and $t\rightarrow\infty$ for a.e. $x\in X$. More precisely, for every $f\in L^p(w)$,
\begin{equation}\label{1.1}
\lim_{t\rightarrow0}\varphi(tL)f(x)=\varphi(0)f(x),\quad\lim_{t\rightarrow\infty}\varphi(tL)f(x)=\varphi(0)Pf(x)\quad\text{a.e. } x\in X.
\end{equation}
Here $P$ denotes the bounded extension to $L^p(w)$ of the orthogonal projection from $L^2(X)$ onto $\ker L$. This extension follows from 
$$\|Pf\|_{L^p(w)}\leq\left\|\sup_{t>0}|e^{-tL}f|\right\|_{L^p(w)}\leq\|V^q(e^{-tL}:t>0)f\|_{L^p(w)}\lesssim\|f\|_{L^p(w)}$$
for every $f\in L^2\cap L^p(w)$. The first inequality holds since $e^{-tL}f\rightarrow Pf$ in $L^2(X)$ as $t\rightarrow\infty$. In particular, $P=0$ in Theorem \ref{thm1} since $L$ is injective. For $f\in L^2(X)$, the funtional calculus gives $\varphi(tL)f\rightarrow \varphi(0)f$ as $t\rightarrow0$ and $\varphi(tL)f\rightarrow\varphi(0)Pf$ as $t\rightarrow\infty$ in $L^2(X)$. The existence of the pointwise limitation is given by finite variation. For $f\in L^p(w)$, choose $f_n\in L^2\cap L^p(w)$ such that $f_n\rightarrow f$ in $L^p(w)$ with $\sum_n\|f_{n+1}-f_n\|_{L^p(w)}<\infty$. The variation estimates implies that $\varphi(tL)f_n(x)$ converges uniformly in $t>0$ for a.e. $x\in X$.

The $q$-variation also provides quantitative information about the convergence \eqref{1.1} through jump estimates, we cite \cite{DK24} for this argument. For $a:\mathbb{R}^+\rightarrow\mathbb{C}$ a function and $\lambda>0$, we define the jump quantity $N_\lambda(a)$ to be the supremum of all integers $N\geq 0$ such that there is an increasing sequence 
$$0<s_1<t_1\leq s_2<t_2\leq...\leq s_N<t_N$$
so that $|a(t_k)-a(s_k)|>\lambda$ for $k=1,...,N$. It is clear from the definition that $\lambda^q N_\lambda(a)\leq\|a\|_{V^q}^q$. Thus, we have the following jump estimate.
\begin{cor}
    Under the assumptions of Corollary \ref{wei-cor}, for $2<q<\infty$ and any spectral multiplier $\varphi$ as there, we have
    $$\|N_{\lambda}(\varphi(tL)f:t>0)^{1/q}\|_{L^p(w)}\lesssim\left([w]_{A_{p/p_0}}[w]_{RH_{(q_0/p)^{\prime}}}\right)^{\alpha}\frac{\|f\|_{L^p(w)}}{\lambda},$$
    where $\alpha$ is as in Corollary \ref{wei-cor}.
\end{cor}
We next comment on the range $p_0<p<q_0$. For analytic semigroups, Le Merdy and Xu \cite{LMX12} established unweighted $q$-variational estimates for all $1<p<\infty$. Their argument assumes contractive regularity but did not require Gaussian estimates. Our sparse domination uses generalized Gaussian estimates and give weighted bounds in the range $p_0<p<q_0$. This range is natural for operators whose
semigroups do not satisfy pointwise Gaussian
bounds; see, for example, \cite{AM06,B03,SYY14}.

When we consider $L^p$-boundedness for a fixed $p_0<p<q_0$, the regularity condition in Theorem \ref{thm2} can be improved. Indeed, by H\"{o}lder's inequality we have
    $${(\textnormal{GGE}_{p_0,q_0,m})}\Longrightarrow{(\textnormal{GGE}_{p_1,q_1,m})}\quad\text{for every}\quad1\leq p_0\leq p_1<2<q_1\leq q_0.$$ 
By choosing $p_1$ sufficiently close to $\min\{p,2\}$ and $q_1$ sufficiently close to $\max\{p,2\}$, we obtain the following observation.
\begin{rem}
   Under the assumptions of Theorem \ref{thm2}, for every fixed $p_0<p<q_0$, the
$q$-variation operator
$V^q(\varphi(tL):t>0)$ is bounded on $L^p(X)$ when $\beta>\nu|1/2-1/p|+1/2$. The corresponding weighted estimates are
established in Section 6.
\end{rem}

We next consider two applications of Theorem \ref{thm2}: Bochner-Riesz means and the regularized Schr\"{o}dinger group. For the Euclidean Laplacian, the Bocher-Riesz means are defined by 
$$\mathcal{F}(B_{R}^{\delta}f)(\xi)=\left(1-\frac{|\xi|^2}{R^2}\right)^{\delta}_+\,\mathcal{F}f(\xi),\quad R>0.$$
The classical Bochner-Riesz conjecture predicts that for $1<p<\infty$ and $\delta>\max\{n|1/p-1/2|-1/2,0\}$, the operator $\|B_{R}^{\delta}\|_{L^p\rightarrow L^p}<\infty$ uniformly on $R$. The corresponding maximal and variational estimates have also been widely studied; see \cite{L04,T02} and the references therein. Lee, Rogers and Seeger \cite{LRS12} proved square-function estimates and noted that this gives the local $2$-variational estimates for Bochner-Riesz means. They established that for compact intervals $I\subset (0,\infty)$, $\|V^2(B_{R}^{\delta}:R\in I)\|_{L^p\rightarrow L^p}<\infty$ when $2(n+2)/n\leq p<\infty$ and $\delta>n(1/2-1/p)-1/2$. For general non-negative self-adjoint operators, Chen, Ouhabaz, Sikora and Yan \cite{COSY16} studied the $L^p$-boundedness for spectral multipliers, including Bochner-Riesz estimates, based on restriction estimates. Later, Chen, Lee, Sikora and Yan \cite{CLSY20} obtained corresponding maximal estimates. 

We next turn to regularized Schr\"{o}dinger groups. For the classial Laplacian, Miyachi \cite{M80,M81} established the sharp smoothing condition $s=n|1/p-1/2|$ for fixed-time $L^p$ estimates of $(I-\Delta)^{-s}e^{-it\Delta}$. Chen, Duong, Li and Yan \cite{CDLY20} extended these estimates to general self-adjoint operators under generalized Gaussian estimates. Under \textnormal{(GGE$_{r,r,m}$)}, they proved that
$$\|(I+L)^{-s}e^{itL}f\|_{L^p}\lesssim(1+|t|)^s\|f\|_{L^p},\quad t\in\mathbb{R},$$
for $r<p<r^{\prime}$ and $s\geq \nu|1/p-1/2|$. The smoothing condition is sharp in their settings. Maximal estimates based on H\"{o}rmander functional calculus were studied by Deleaval and Kriegler \cite{DK23}. In their subsequent work, Deleaval and Kriegler \cite{DK24} also established the corresponding $q$-variational estimates $V^q$ ($q>2$) for some Sobolev-type spectral multipliers.

As a corollary of Theorem \ref{thm2}, we have the following weighted $q$-variational estimates for Bochner-Riesz means and regularized Schr\"{o}dinger groups.

\begin{cor}\label{application}
   Let $L$ be a non-negative self-adjoint a operator on $L^2(X)$. Suppose the semigroup $\{e^{-tL}\}_{t>0}$ satisfies \eqref{DGGE} for some $m\geq 2$ and $1\leq p_0<2<q_0\leq\infty$. Fix $p_0<p<q_0$ and $q>2$. Assume that $\delta>\nu|1/p-1/2|$ and $s>\nu|1/p-1/2|+1/2$. Then there exists $\varepsilon_0>0$, such that for every $w\in A_{\max\{1,p/2\}+\varepsilon_0}\cap RH_{1/((1-p/2)_++\varepsilon_0)}$, we have
    $$\|V^q((I-tL)_+^{\delta}:t>0)f\|_{L^p(w)}\leq C_{\delta}\left([w]_{A_{\max\{1,p/2\}+\varepsilon_0}}[w]_{RH_{1/((1-p/2)_++\varepsilon_0)}}\right)^{\gamma_{p,\varepsilon_0}}\|f\|_{L^p(w)}$$
    and
    $$\|V^q((I+tL)^{-s}e^{itL}:t>0)f\|_{L^p(w)}\leq C_{s}\left([w]_{A_{\max\{1,p/2\}+\varepsilon_0}}[w]_{RH_{1/((1-p/2)_++\varepsilon_0)}}\right)^{\gamma_{p,\varepsilon_0}}\|f\|_{L^p(w)},$$
    where  $\gamma_{p,\varepsilon_0}=(\max\{1,p-1\}+\varepsilon_0)/(p\varepsilon_0)$.
\end{cor}
The preceding approximation argument \eqref{1.1} gives that for $f\in L^p(w)$ here,
\begin{equation}\label{1.2}
\lim_{t\rightarrow0}(I-tL)_+^\delta f(x)=\lim_{t\rightarrow0}(I+tL)^{-s}e^{itL}f(x)=f(x)
\end{equation}
and 
\begin{equation}\label{1.3}
\lim_{t\rightarrow\infty}(I-tL)_+^\delta f(x)=\lim_{t\rightarrow\infty}(I+tL)^{-s}e^{itL}f(x)=Pf(x)
\end{equation}
for a.e. $x\in X$. The corresponding jump estimates also hold with the same weight factors as above.

The paper is organized as follows. In Section 2, we recall some preliminary results. We prove Theorem \ref{thm1} in Section 3 and Theorem \ref{thm2} in Section 4. In Section 5, we construct a counterexample to show that the regularity condition in Theorem \ref{thm2} cannot be lowered in general. Finally in Section 6, we establish weighted estimates, prove Corollary \ref{application}, and give examples of $L$ satisfying our assumptions.

We conclude the introduction with some notation. We use $a\lesssim b$ to say that there exists a constant $C$, which is independent of the important parameters, such that $a\leq Cb$. Moreover, we write $a\sim b$ if $a\lesssim b$ and $b\lesssim a$. For a measurable set $E$, we denote the characteristic function by $\chi_E$. We write $C_{0}(\mathbb{R})$ for the class of continuous function on $\mathbb{R}$ that vanish at $\pm\infty$, and $C_c^{\infty}(\mathbb{R})$ for the class of smooth function with compact support.

\section{Preliminary}
\subsection{Spaces of homogeneous type}
Let $(X,d)$ be a metric space, and $\mu$ be a non-negative Borel measure on $X$. For a measurable set $E\subset X$, we write $|E|=\mu(E)$. We assume that $\mu$ satisfies the doubling property: there exists a constant $C>0$ such that
\begin{equation*} 
    |B(x,2r)|\leq C\ |B(x,r)|,\quad \forall x\in X,\ r>0,
\end{equation*}
where $B(x,r):=\left\{y\in X:\ d(x,y)<r \right\}$. It follows that there exists $\nu>0$ such that
\begin{equation}\label{doub}
    |B(x,r)|\leq C\left(\frac{r}{s}\right)^{\nu} |B(x,s)|,\quad \forall x\in X,\ r\geq s>0,
\end{equation}
which implies that there also exists $C^{\prime}>0$ that
\begin{equation}
    |B(x,r)|\leq C^{\prime} \left(\frac{d(x,y)+r}{s}\right)^{\nu}|B(y,s)|,\quad \forall x,y\in X,\ r\geq s>0.
\end{equation}
Two intersecting balls with comparable radius have comparable measures. For more details on spaces of homogeneous type, we refer to \cite{CW71}.

We recall the definition and construction of dyadic systems on spaces of homogeneous type (see \cite{C90,HK12} for more details).
\begin{defi}
    Let $0<c_{ 0 }\leq C_{0}<\infty$ and $0<\delta<1$. Assume that for every $k\in\mathbb{Z}$, there is a pairwise disjoint collection $\mathcal{D}_{k} = (Q_{j}^{k})_{j\in J_{k}}$ in $X$ and points $(z_{j}^{k})_{j\in J_{k}}\in X$. Then we will call $(\mathcal{D}_{k})_{k\in\mathbb{Z}}$ a dyadic system in $X$, if it satisfies the following properties:\\
    \textnormal{(i)} For all $k\in\mathbb{Z}$, we have
    $$X = \bigcup_{j\in J_{k}}Q_{j}^{k};$$
\textnormal{(ii)} Let $l\geq k$, if $Q\in{\mathcal{D}}_{l}$ and $Q^{\prime}\in\mathcal {D}_{k}$, we have either $Q\cap Q^{\prime} =\varnothing\ \text{or}\ Q\subset Q^{\prime};$\\
\textnormal{(iii)} For each $k\in\mathbb{Z}$ and $j\in J_{k}$, we have
$$B(z_{j}^{k},c_{0}\delta^{k})\subset Q_{j}^{k}\subset B(z_{j}^{k},C_{0}\delta^{k});$$
\textnormal{(iv)} For $l\geq k$, if $Q_{j^{\prime}}^{l}\subset Q_{j}^{k}$, then $B(z_ {j^{\prime}}^{l},C_{0}\delta^k)\subset B(z_j^{k},C_{0}\delta^{k})$.
\end{defi}
These elements of a dyadic system are called `cubes'. For a detailed discussion on the construction of dyadic systems and for the following theorem, we refer the reader to \cite{HK12} and references therein.
\begin{thm} \label{Dyadic}
    There exist constants $c_{0},C_{0},\delta>0$, $K_D=K_D(c_{0}, C_{0},\delta)\in\mathbb{N}_+$ and $\rho=\rho(c_{0},C_{0},\delta )$ as well as a finite collection of families $\mathcal{D}^{b}$ for $b=1,...,K_D$, where each $\mathcal {D}^{b}$ is a dyadic system with the property that for every ball $B=B(x,r)\subset X$, there exists $b\in\{1,...,K_D\}$ and ${Q}\in\mathcal{D}^{b}$ with $B\subset Q$ and $\textnormal{diam}(Q)\leq\rho r$.
\end{thm}
\noindent We denote
$$\mathcal{D}:=\bigcup_{b=1}^{K_D}\mathcal{D}^{b},$$
and call a cube dyadic cube whenever $Q\in\mathcal{D}$. For every dyadic cube $Q\in\mathcal{D}$, we denote $r_Q=\delta^k$ where the integer $k$ is determined by $\delta^{k+1}\leq\text{diam}(Q)<\delta^k$. 

\subsection{Holomorphic functional calculus and $\Lambda^{\beta}_{2,2}(\mathbb{R_+})$}
\paragraph{Accretive operators and holomorphic functional calculus.}
We first recall some basic facts concerning the holomorphic functional calculus for accretive operators. We refer to \cite{ADM96,H06,M86} and the references therein for further details.

For $0<\sigma<\pi$, let
$$\Sigma_{\sigma}:=\{z\in\mathbb{C}: z\neq0,|\arg z|<\sigma \}$$
and set $\Sigma_0=(0,\infty)$. We write $\overline{\Sigma_{\sigma}}$ for the closure of $\Sigma_{\sigma}$ in $\mathbb{C}$. Let $H(\Sigma_{\sigma})$ denote the space of holomorphic functions on $\Sigma_{\sigma}$. Set
$$H^{\infty}(\Sigma_{\sigma}):=\{\varphi\in H(\Sigma_{\sigma}): \|\varphi\|_{\infty,\sigma}:=\text{sup}_{z\in\Sigma_{\sigma}}|\varphi(z)|<\infty\}.$$
For $\alpha,\beta>0$, define
$$\Psi_{\alpha,\beta}(\Sigma_{\sigma}):=\{\varphi\in H(\Sigma_{\sigma}): \exists\ C:|\varphi(z)|\leq C|z|^{\alpha} (1+|z|^{\alpha+\beta})^{-1},\ z\in\Sigma_{\sigma}\}.$$
We denote the infimum of such constant $C$ as $\|\varphi\|_{\Psi_{\alpha,\beta}(\Sigma_{\sigma})}$. We define
$${H^{\infty}_0(\Sigma_{\sigma}):=\bigcup_{\alpha,\beta>0}\Psi_{\alpha,\beta} (\Sigma_{\sigma})}.$$
Then we have the following elementary consequence of Cauchy's integral formula for $H^{\infty}_0$.
\begin{lem}\label{derivative}
Let $0<\nu<\mu<\pi$, let $\alpha,\beta>0$, and let $k\in\mathbb N$. If $\varphi\in H^{\infty}_0(\Sigma_{\mu})$, then $z^k \varphi^{(k)}(z)$ belongs to $H^{\infty}_0(\Sigma_{\nu})$. More precisely, for $\varphi\in\Psi_{\alpha,\beta}(\Sigma_\mu)$, \[ \|z^k \varphi^{(k)}(z)\|_{\Psi_{\alpha,\beta}(\Sigma_\nu)} \lesssim\|\varphi\|_{\Psi_{\alpha,\beta}(\Sigma_\mu)}. \]
\end{lem}

\begin{proof} Choose $\delta\in(0,1/2)$, such that $ B(z,\delta|z|)\subset\Sigma_\mu$ for all $z\in\Sigma_\nu.$  For $|w-z|=\delta|z|$, one has $ |w|\sim|z|$, thus 
$$ |w|^{\alpha}(1+|w|^{\alpha+\beta})^{-1}\sim |z|^{\alpha}(1+|z|^{\alpha+\beta})^{-1}.$$
By Cauchy integral formula, 
$$ \varphi^{(k)}(z) = \frac{k!}{2\pi i}\int_{|w-z|=\delta|z|}\frac{\varphi(w)}{(w-z)^{k+1}}\,dw. $$
It follows that 
$$|z^k \varphi^{(k)}(z)| \lesssim \frac{|z|^{\alpha}}{1+|z|^{\alpha+\beta}}\|\varphi\|_{\Psi_{\alpha,\beta}(\Sigma_\mu)}.$$
This completes the proof. 
\end{proof}

Denote $L^2(X)$ by $\mathcal{H}$, and we let $0\leq\omega<\pi$. A closed operator $L$ on $\mathcal{H}$ is called $\omega$-sectorial, if the spectrum $\sigma(L)$ is contained in $\overline{\Sigma_{\omega}}$, and, for each $\theta>\omega$, there exists a $C_{\theta}>0$ such that 
$$\|(zI-L)^{-1}\|_{\mathcal{B}(\mathcal{H})}\leq C_{\theta}|z|^{-1}\quad \text{for } z\notin \overline{\Sigma_{\theta}}.$$
Further for $0\leq \omega<\pi/2$, a closed operator $L$ on $\mathcal{H}$ is called $\omega$-accretive if $\sigma(L)\subset \overline{\Sigma_{\omega}}$ and $\langle Lu,u\rangle_{\mathcal{H}}\in \overline{\Sigma_{\omega}}$ for all $u\in D(L)$. It is easy to see $\|(zI-L)^{-1}\|_{\mathcal{B}(H)}\leq(\text{dist}(z,\overline{\Sigma_{w}}))^{-1}$, so that $L$ is $w$-sectorial. We remark that $w$-accretive operators necessarily have dense domain, and are self-adjoint if and only if $\omega=0$. Let $\omega<\mu<\pi$ and $\varphi\in H^{\infty}_{0}(\Sigma_{\mu})$, the holomorphic functional calculus $\varphi(L)$ is defined by 
\begin{equation}\label{Hi}
\varphi(L):=\frac{1}{2\pi i}\int_{\gamma}\varphi(\lambda)(\lambda I-L)^{-1}\,d\lambda,
\end{equation}
where $\gamma:=\{\rho e^{i\nu}:\ \rho:\infty\rightarrow0\}\ \cup\ \{\rho e^{-i\nu}:\ \rho:0\rightarrow\infty\}$, $\omega<\nu<\mu$. The integral converges absolutely in $\mathcal{B}(\mathcal{H})$, and its value is independent of the choice of $\nu$ by Cauchy's theorem. We say that $L$ admits a bounded $H^\infty(\Sigma_\mu)$ functional calculus on $\mathcal{H}$ if there exists a constant $C_\mu>0$ such that
$$\|\varphi(L)\|_{\mathcal B(\mathcal{H})} \leq C_\mu\|\varphi\|_{\infty,\mu}$$ 
for all $\varphi\in H^{\infty}_0(\Sigma_\mu)$. If $L$ is injective, then its range is dense in $\mathcal{H}$ and the calculus extends uniquely to a bounded unital algebra homomorphism from $H^\infty(\Sigma_\mu)$ into $\mathcal B(\mathcal{H})$ (see e.g. \cite{M86}). Notice that (see e.g. \cite[p. 103]{ADM96}) an injective $\omega$-accretive operator has bounded $H^{\infty}(\Sigma_{\mu})$ functional calculus for every $\mu>\omega$. The bounded $H^{\infty}$ functional calculus gives the following standard quadratic estimate. 
\begin{lem}\label{quad}
\textnormal{(\cite[Theorem 7.3.1]{H06})} Let $L$ be an injective $\omega$-sectorial operator on the Hilbert space $H$, which admits a bounded $H^{\infty}(\Sigma_{\theta})$ functional calculus for a $\omega<\theta<\pi$. Then for every $\omega<\mu<\pi$ and $0\neq \varphi\in H^{\infty}_0(\Sigma_{\mu})$, it holds that
\begin{equation}\label{qua}
    \int_{0}^{\infty}\|\varphi(tL)x\|_{H}^2\frac{dt}{t}\sim\|x\|_{H}^2\notag
\end{equation}
for all $x\in H$, where the implicit constant depend on $\varphi$.
\end{lem}
Finally, since $\omega<\pi/2$, it follows from \cite[Section 7]{H06} that for an $\omega$-accretive operator $L$ (or generally, $\omega$-sectorial operator), $-L$ generates an analytic semigroup $\{e^{-zL}\}_{z\in \Sigma_{\pi/2-\omega}}$ on $\Sigma_{\pi/2-\omega}$. We also have the following representation of $\varphi(L)$.
\begin{equation}\label{H2}
    \varphi(L)=\int_{\Gamma_{+}} e^{-zL}\eta_{+}(z)\,dz+\int_{\Gamma_{-}}e^{-zL}\eta_{-}(z)\,dz,
\end{equation}
where
\begin{equation}\label{H3}
    \eta_{\pm}(z):=\frac{1}{2\pi i}\int_{\gamma_{\pm}}e^{z\lambda}\varphi(\lambda)\,d\lambda,\quad z\in\Gamma_{\pm},
\end{equation}
$\gamma_{\pm}$ denotes the curve $\gamma$ defined above, corresponding respectively to the angle $\nu$ and $-\nu$, and $\Gamma_{\pm}:=\{re^{\pm i\phi}:\ r:0\rightarrow\infty\}$ for any $\phi\in(\pi/2-\nu,\pi/2-\omega)$. 
\paragraph{The space $\Lambda^{\beta}_{2,2}(\mathbb{R}_+)$.} We next specialize to the case where $L$ is a non-negative self-adjoint operator on $\mathcal{H}$. For every Borel function $\varphi:[0,\infty)\rightarrow\mathbb{C}$, one defines
\begin{equation}\label{spectralthm}
    \varphi(L):=\int_{0}^{\infty}\varphi(\lambda)\,dE_{L}(\lambda),
\end{equation}
where $E_{L}$ denotes the associated spectral measure.
\begin{defi}
    We define
    $$\Lambda^{\beta}:=\Lambda_{2,2}^{\beta}(\mathbb{R}_+):=\{\varphi:(0,\infty)\rightarrow\mathbb{C}:\ \varphi\circ\text{exp}\in W^{\beta,2}(\mathbb{R})\},$$
    where $W^{\beta,2}$ denotes the usual Sobolev space. We write $\|\varphi\|_{\Lambda^{\beta}}:=\|\varphi\circ\text{exp}\|_{W^{\beta,2}}$. The space $\Lambda^{\beta}$ is a separable Hilbert space, and embeds into $C_0(\mathbb{R}_+)$ for $\beta>1/2$. This follows from the Sobolev embedding $W^{\beta,2}(\mathbb{R})\hookrightarrow C_{0}(\mathbb{R})$ for $\beta>1/2$. We define $\varphi(0)=0$ and use this in \eqref{spectralthm}. Moreover, by Lemma \ref{derivative}, it follows that $H^{\infty}_0(\Sigma_{\mu})|_{\mathbb{R}_+}\hookrightarrow \Lambda^{\beta}$ for every $0<\mu<\pi$ and $\beta\geq0$. 
\end{defi}
\begin{rem}
    The $\Lambda^{\beta}$ norm is invariant under dilations that $\|\varphi(t\cdot)\|_{\Lambda^{\beta}}=\|\varphi\|_{\Lambda^{\beta}}$ for $t>0$. Suppose $\varphi:\mathbb{R}_+\rightarrow\mathbb{C}$ is measurable and has compact support in $\mathbb{R}_+$, then $\|\varphi\|_{\Lambda^{\beta}}\sim\|\varphi\|_{W^{\beta,2}(\mathbb{R})}$, where the equivalence constants depend on the compact support.
    
    For comparison with dyadic Sobolev conditions in \textnormal{\cite{CLWY26,DK23,DK24}}, let $\phi_0\in C_{c}^{\infty}(\mathbb{R}_+)$ satisfy supp$(\phi_0)\subset[1/2,2]$ and $\sum_{k\in\mathbb{Z}}\phi_0(2^{-k}t)=1$ for $t>0$. Then 
    \begin{equation}
    \|\varphi\|_{\Lambda^{\beta}}\sim\left(\sum_{k\in\mathbb{Z}}\|\phi_0\varphi(2^{k}\cdot)\|_{W^{\beta,2}}^2\right)^{1/2}.
    \end{equation}
   This follows from Sovolev localization applied to $\varphi\circ\textnormal{exp}$, \textnormal{see e.g. \cite[Theorem 1.3]{M88}}.
\end{rem}
 The following simple embedding will be used later.
\begin{lem}\label{lambdaemb}
    \textnormal{(\cite[Lemma 3.2]{DK24})} Let $2\leq q<\infty$ and $\beta>1/2$, then $\Lambda^{\beta}\hookrightarrow{V}^q$. Here we use the unique continuous representative of each $a\in\Lambda^\beta$.
\end{lem}
\subsection{Generalized Gaussian estimates}
\begin{defi}
Let $L$ be an $\omega$-accretive operator on $L^2(X)$ for $0\leq\omega<\pi/2$, generating the semigroup $\{e^{-tL}\}_{t>0}$. We say that $\{e^{-tL}\}_{t>0}$ satisfies the generalized Gaussian \eqref{GGE} for $m\ge 2$ and $1\leq p_0<2<q_0\leq \infty$, if there exist constant $C,c>0$ such that
\begin{equation}\label{GGE}
    \|\chi_{B_2} e^{-tL} \chi_{B_1}\|_{p_0\rightarrow q_0}\leq C|B_{1}|^{-1/p_0}|B_2|^{1/q_0}\,\textnormal{exp}\left(-c\left(\frac{d(B_1,B_2)}{t^{1/m}}\right)^{{m}/{(m-1)}}\right).\tag{GGE$_{p_0,q_0,m}$}
\end{equation}
for every $t>0$, and all balls $B_1,B_2$ of radius $t^{1/m}$. Similarly, a family $\{R_t\}_{t>0}\subset \mathcal{B}(L^2(X))$ is said to satisfy \textnormal{(GGE$_{p_0,q_0,m}$)} if the same estimate holds.
\end{defi}
In the special case $p_0=1$, $q_0=\infty$, the condition (GGE$_{p_0,q_0,m}$) is equivalent to $m$-th order Gaussian estimates (see e.g. \cite{BK03}). This means that the semigroup $e^{-tL}$ has integral kernels $p_t(x,y)$ satisfying the Gaussian upper bound
\begin{equation}\label{GE}
    |p_{t}(x,y)|\leq \frac{C}{|B(x,t^{1/m})|}\,\textnormal{exp}\left(-c\left(\frac{d(B_1,B_2)}{t^{1/m}}\right)^{{m}/{(m-1)}}\right).\tag{GE$_{m}$}
\end{equation}

We also introduce a domination condition on the semigroup. Recall that a linear operator $R$ on $L^2(X)$ is called positive if $Rf\geq0$ a.e. whenever $f\geq 0$ a.e.
\begin{defi}
Let $\{e^{-tL}\}_{t>0}$ be a semigroup acting on $L^2(X)$, we say that $\{e^{-tL}\}_{t>0}$ satisfies \eqref{DGGE} if there exists a family of positive linear operators $\{R_t\}_{t>0}$ on $L^2(x)$ such that for every $t>0$,
\begin{equation}\label{DGGE}
    \|R_t\|_{\mathcal{B}(L^2(X))}\leq 1,\quad |e^{-tL}f|\leq R_t(|f|)\ \text{a.e. on }X,\tag{DGGE$_{p_0,q_0,m}$}
\end{equation}
and $\{R_t\}_{t>0}$ satisfies \eqref{GGE}. Notice that \textnormal{(DGGE$_{p_0,q_0,m}$)} implies both the contractive regularity of $\{e^{-tL}\}_{t>0}$ \textnormal{(see e.g. \cite{LMX12})} and the usual \textnormal{(GGE$_{p_0,q_0,m}$)}.
\end{defi}

We note that the generalized Gaussian estimates can be extended from positive real $t$ to complex values.

\begin{lem}\label{ext}
For $m\geq 2$ and $1\leq p_0<2<q_0\leq \infty$, we denote $\delta_{p_0}:=1/p_{0}-1/2$, $\delta_{q_0}:=1/2-1/q_0$. For $r>0$, we write $v_r(x):=|B(x,r)|$.

\textnormal{(i)} Let $L$ be an $\omega$-sectorial operator on $L^2(X)$ with $0\leq\omega<\pi/2$. Suppose $\{e^{-tL}\}_{t>0}$ satisfies \textnormal{(GGE$_{p_0,q_0,m}$)}. Then for every $0<\theta<\pi/2-\omega$, there exist constant $C_{\theta},c_{\theta}>0$, such that for any balls $B_1,B_2$ and $z\in \Sigma_{\theta}$, we have that
\begin{equation}
\|\chi_{B_2}v_{\textnormal{Re}(z)^{1/m}}^{1/p_0-1/q_0}e^{-zL}\chi_{B_1}\|_{p_0\rightarrow q_{0}}\leq C_{\theta}\,\textnormal{exp}\left(-c_\theta \left(\frac{d(B_1,B_2)}{|z|^{1/m}}\right)^{{m}/{(m-1)}}\right).\notag
\end{equation}

\textnormal{(ii)} Let $L$ be a non-negative self-adjoint operator on $L^2(X)$. Suppose $\{e^{-tL}\}_{t>0}$ satisfies \textnormal{(GGE$_{p_0,q_0,m}$)}. Then there exists constant $C^{\prime},c^{\prime}>0$, such that for any measurable set $E,F$ and $z\in \mathbb{C}_+$, we have that
\begin{equation}
    \|\chi_{F}v_{\textnormal{Re}(z)^{1/m}}^{\delta_{q_0}}e^{-zL}v_{\textnormal{Re}(z)^{1/m}}^{\delta_{p_0}}\chi_{E}\|_{p_0\rightarrow q_0}\leq C^{\prime}\,\textnormal{exp}\left(-c_\theta \left(\frac{d(E,F)}{|z|\textnormal{Re}(z)^{(1/m)-1}}\right)^{{m}/{(m-1)}}\right).\notag
\end{equation}
\end{lem}

\begin{proof}
    The proof of \textnormal{(i)} follows directly from the strategy in \cite[Theorem 2.1]{B07}; see also the remarks in Section 2 therein. We therefore only prove \textnormal{(ii)} here. The proof is based on the Phragm\'{e}n-Lindel\"{o}f theorem. Fix measurable sets $E,F$, $z\in \mathbb{C}_+$ and $s>0$. We consider the analytic function $G:\mathbb{C}_+\rightarrow \mathcal{B}(L^{p_0},L^{q_0})$ defined by
    \begin{equation}
        G(\omega):=e^{-\omega(\textnormal{Re}(z))^{-1}}\left(\frac{\omega}{s^m}\right)^{(\nu/m)(1/p_0-1/q_0)}\chi_{F}v_s^{\delta_{q_0}}e^{-\omega L}v_s^{\delta_{p_0}}\chi_E.
    \end{equation}
    First recall the basic volume estimate 
    \begin{equation}
        v_s(x)\lesssim\left(1+\frac{s}{r}\right)^{\nu}v_r(x)\quad\text{for all } x\in X,\ r,s>0.
    \end{equation}
    It follows from \cite[Proposition 3.1]{B07} that
   $$\|v_{t^{1/m}}^{\delta_{q_0}}e^{-tL}\|_{2\rightarrow q_0},\ \|e^{-tL}v_{t^{1/m}}^{\delta_{p_0}}\|_{p_0\rightarrow 2}\leq C\quad\text{for all }t>0.$$
   Hence, for all $\omega\in\mathbb{C}_+$ and $s>0$, we write $e^{-\omega L}=e^{-\textnormal{Re}(w)L/2}e^{-i\textnormal{Im}(\omega)L}e^{-\textnormal{Re}(w)L/2}$, it holds that
   \begin{align*}
       \|v_s^{\delta_{q_0}}e^{-\omega L}v_{s}^{\delta_{p_0}}\|_{p_0\rightarrow q_0}&\leq \|v_s^{\delta_{q_0}}e^{-\textnormal{Re}(\omega)L/2}\|_{2\rightarrow q_0}\|e^{-i\textnormal{Im}(\omega)L}\|_{2\rightarrow2}\|e^{-\textnormal{Re}(\omega)L/2}v_{s}^{\delta_{p_0}}\|_{p_0\rightarrow2}\\
       &\lesssim\left(1+\frac{s}{\textnormal{Re}(\omega)^{1/m}}\right)^{\nu\delta_{q_0}}\|v_{t^{1/m}}^{\delta_{q_0}}e^{-tL}\|_{2\rightarrow q_0}\left(1+\frac{s}{\textnormal{Re}(\omega)^{1/m}}\right)^{\nu\delta_{p_0}}\|e^{-tL}v_{t^{1/m}}^{\delta_{p_0}}\|_{p_0\rightarrow 2}\\
       &\lesssim\left(1+\frac{s}{\textnormal{Re}(\omega)^{1/m}}\right)^{\nu(1/p_0-1/q_0)}.
   \end{align*}
 Consequently,
 \begin{align}
     \|G(\omega)\|_{p_0\rightarrow q_0}&\lesssim e^{-\textnormal{Re}(\omega)(\textnormal{Re}(z))^{-1}}\left(\frac{|\omega|^{1/m}}{s}\right)^{\nu(1/p_0-1/q_0)}\left(1+\left(\frac{s^m}{\textnormal{Re}(\omega)}\right)^{(\nu/m)(1/p_0-1/q_0)}\right)\notag\\
     &=e^{-\textnormal{Re}(\omega)(\textnormal{Re}(z))^{-1}}\left(\frac{|\omega|}{\textnormal{Re}(\omega)}\right)^{(\nu/m)(1/p_0-1/q_0)}\left(1+\left(\frac{\textnormal{Re}(\omega)}{s^m}\right)^{(\nu/m)(1/p_0-1/q_0)}\right)\notag\\
     &\lesssim \left(\frac{|\omega|}{\textnormal{Re}(\omega)}\right)^{(\nu/m)(1/p_0-1/q_0)}\left(1+\frac{\textnormal{Re}(z)}{s^m}\right)^{(\nu/m)(1/p_0-1/q_0)}\notag\\
     &=\cos(\textnormal{arg}\,\omega)^{-(\nu/m)(1/p_0-1/q_0)}\left(1+\frac{\textnormal{Re}(z)}{s^m}\right)^{(\nu/m)(1/p_0-1/q_0)}.\label{I}
 \end{align}
 The implicit constant is independent of $z$ and $s$. Here in the second-to-last inequality, we use the basic estimate $e^{-\lambda}(1+\lambda^{\gamma}A)\leq C_{\gamma}(1+A)$ for $\lambda,\gamma,A>0$. On the other hand, it follows from \cite[Theorem 1.2]{BK05} that
 $$\|\chi_F v_{t^{1/m}}^{\delta_{q_0}} e^{-tL}v_{t^{1/m}}^{\delta_{p_0}}\chi_{E}\|_{p_0\rightarrow q_0}\lesssim\textnormal{exp}\left(-c\left(\frac{d(E,F)}{t^{1/m}}\right)^{m/(m-1)}\right).$$
 By a similar calculation as above, for $t>0$, one have
 \begin{align}
     \|G(t)\|_{p_0\rightarrow q_0}&\lesssim e^{-t(\textnormal{Re}(z))^{-1}}\left(\frac{t}{s^m}\left(1+\frac{s^m}{t}\right)\right)^{(\nu/m)(1/p_0-1/q_0)}\textnormal{exp}\left(-c\left(\frac{d(E,F)}{t^{1/m}}\right)^{m/(m-1)}\right)\notag\\
     &\lesssim\left(1+\frac{\textnormal{Re}(z)}{s^m}\right)^{(\nu/m)(1/p_0-1/q_0)}\textnormal{exp}\left(-c\left(\frac{d(E,F)}{t^{1/m}}\right)^{m/(m-1)}\right).\label{II}
 \end{align}
 Now by applying the Phragm\'{e}n-Lindel\"{o}f theorem \cite[Lemma 3.3]{B07} to $G(\omega)$, we obtain that there exist $c^{\prime}>0$ such that for every $\omega\in \mathbb{C}_+$,
 \begin{equation}
     \|G(\omega)\|_{p_0\rightarrow q_0}\lesssim\left(\frac{|\omega|}{\textnormal{Re}(\omega)}\left(1+\frac{\textnormal{Re}(z)}{s^m}\right)\right)^{(\nu/m)(1/p_0-1/q_0)}\textnormal{exp}\left(-c^{\prime}\left(\frac{d(E,F)}{|\omega|\textnormal{Re}(\omega)^{(1/m)-1}}\right)^{m/(m-1)}\right).\notag
 \end{equation}
 By letting $\omega=z$, we obtain that
 \begin{equation}
     \|\chi_F v_s^{\delta_{q_0}}e^{-zL}v_s^{\delta_{p_0}}\chi_E\|_{p_0\rightarrow q_0}\lesssim\left(1+\frac{s^m}{\textnormal{Re}(z)}\right)^{(\nu/m)(1/p_0-1/q_0)}\textnormal{exp}\left(-c^{\prime}\left(\frac{d(E,F)}{|z|\textnormal{Re}(z)^{(1/m)-1}}\right)^{m/(m-1)}\right).\notag
 \end{equation}
 Finally, by taking $s=\textnormal{Re}(z)^{1/m}$, we obtain the desired result.
\end{proof}

The generalized Gaussian estimates for the semigroup also imply the same estimates for its derivatives $(tL)^N e^{-tL}$.
\begin{lem}\label{deri}
    Let $L$ be an $\omega$-sectorial operator on $L^2(X)$ with $0\leq\omega<\pi/2$. Suppose $\{e^{-tL}\}_{t>0}$ satisfies \textnormal{(GGE$_{p_0,q_0,m}$)}. Then for every $N\in \mathbb{N}$, $(tL)^N e^{-tL}$ also satisfies \textnormal{(GGE$_{p_0,q_0,m}$)}.
\end{lem}
\begin{proof}
    Fix $t>0$ and $0<\theta<\pi/2-\omega$. Choose $\varepsilon>0$ sufficiently small that the circle $\gamma_{t}:=\{z\in\mathbb{C}_+:|z-t|=\varepsilon t\}$ is contained in $\Sigma_{\theta}$. By Cauchy's integral formula, we have 
    \begin{equation}\label{Cauchy}
        (tL)^Ne^{-tL}=\frac{(-1)^N N!}{2\pi i}\int_{\gamma_t}\frac{t^N e^{-z L}}{(z-t)^{N+1}}\,dz.
    \end{equation}
    Let $B_1,B_2$ be balls of radius $t^{1/m}$. For any $x\in B_2$, we have $B_2\subset B(x,2t^{1/m})$. Thus $|B_2|\leq|B(x,2t^{1/m})|\lesssim |B(x,t^{1/m})|$, and $v_{t^{1/m}}(x)\gtrsim|B_2|$. Since $\textnormal{Re}(z)\sim|z|\sim t$ for every $z\in \gamma_t$, it follows from Lemma \ref{ext} \textnormal{(i)} that
    \begin{align}
        \|\chi_{B_2}e^{-zL}\chi_{B_1}\|_{p_0\rightarrow q_{0}}&\lesssim\|\chi_{B_2}v_{\textnormal{Re}(z)^{1/m}}^{1/p_0-1/q_0}e^{-zL}\chi_{B_1}\|_{p_0\rightarrow q_{0}}\|\chi_{B_2}v_{t^{1/m}}^{1/q_0-1/p_0}\|_{\infty}\notag\\
        &\lesssim |B_2|^{1/q_0-1/p_0}\,\textnormal{exp}\left(-c_\theta^{\prime} \left(\frac{d(B_1,B_2)}{t^{1/m}}\right)^{{m}/{(m-1)}}\right).\label{2.13}
    \end{align}
    Hence, by Minkowski's inequality we have that
    \begin{align*}
        \|\chi_{B_2}(tL)^N e^{-tL}\chi_{B_1}\|_{p_0\rightarrow q_0}&\lesssim\int_{\gamma_t}\frac{t^N}{|z-t|^{N+1}}\|\chi_{B_2}e^{-zL}\chi_{B_1}\|_{p_0\rightarrow q_0}\,|dz|\\
        &\lesssim\frac{t^N}{t^{N+1}}2\pi t|B_2|^{1/q_0-1/p_0}\,\textnormal{exp}\left(-c_\theta^{\prime} \left(\frac{d(B_1,B_2)}{t^{1/m}}\right)^{{m}/{(m-1)}}\right)\\
        &\sim|B_{1}|^{-1/p_0}|B_2|^{1/q_0}\,\textnormal{exp}\left(-c^{\prime\prime}\left(\frac{d(B_1,B_2)}{t^{1/m}}\right)^{{m}/{(m-1)}}\right).
    \end{align*}
    Here in the last step, we use the estimate $|B_2|/|B_1|\lesssim(1+d(B_1,B_2)/t^{1/m})^\nu$, and absorb the polynomial factor into the exponential.
\end{proof}
\subsection{Sparse domination}
We first recall some basic notations. For a measurable set $E\subset X$ with $0<|E|<\infty$ and $1\leq p<\infty$, we denote $\langle f\rangle_{p,E}:=(|E|^{-1}\int_{E}|f|^p\,d\mu)^{1/p}$, and we set $\langle f\rangle_{\infty,E}:=\operatorname*{ess\,sup}_{x\in E}|f(x)|$. We denote by $M$ the Hardy-Littlewood maximal operator. For $1\leq p<\infty$, we denote $M_p$ the operator defined by $M_p(f):=[M(|f|^p)]^{1/p}$.

From the subsection 2.1 we know that we have several dyadic grids $\mathcal{D}^b$ for $b\in\{1,...,K_D\}$ and $\mathcal{D}:=\bigcup_{b=1}^{K_D} \mathcal{D}^b$ is the class of dyadic cubes. Let $0<\eta<1$, a collection of dyadic sets $\mathcal{S}:=(P)_{P\in\mathcal{S}}\subset \mathcal{D}$ is said to be $\eta$-sparse if there exist pairwise disjoint sets $E_Q\subset Q$, $Q\in\mathcal{S}$ such that $|E_{Q}|\geq\eta|Q|$. We simply call a family sparse if it is $\eta$-sparse for some $0<\eta<1$. For a dyadic cube $Q\in \mathcal{D}$, we denote by $5Q$ its neighborhood
$$5Q:=\{x\in X:\ d(x,Q)\leq 4r_Q\},$$
where the 'side length' $r_Q$ is determined in Section 2.1. We recall the sparse domination principle by Bernicot, Frey and Petermichl \cite{BFP16}.
\begin{lem}\label{spDo}
    Let $L$ be an injective $\omega$-accretive operator on $L^2(X)$ with $0\leq\omega<\pi/2$, and suppose that the semigroup $\{e^{-tL}\}_{t>0}$ satisfies \textnormal{(GGE$_{p_0,q_0,m}$)} for some $m\geq 2$ and $1\leq p_0<2<q_0\leq\infty$. Let $T$ be a \textnormal{(sub)linear} operator bounded on $L^2(X)$. Suppose further that the following conditions are satisfied:

\textnormal{(i)} There exists $N_0\in\mathbb{N}$ such that, for every $N\geq N_0$, $t>0$ and two balls $B_1,B_2$ of radius $t^{1/m}$,
\begin{equation}\label{BFP1}
\|\chi_{B_2} T(tL)^Ne^{-tL}\chi_{B_1}\|_{{p_0}\rightarrow {q_0}}
\lesssim
|B_1|^{-1/p_0}|B_2|^{1/q_0}
\left(
1+\frac{d(B_1,B_2)}{t^{1/m}}
\right)^{-(\nu+1)}.
\end{equation}

\textnormal{(ii)} There exists $p_0\leq p_1<2$ such that for every $x\in X$ and $r>0$,
\begin{equation}\label{BFP2}
\left(
\frac{1}{|B(x,r)|}
\int_{B(x,r)}
|Te^{-r^mL}f|^{q_0}\,d\mu
\right)^{1/q_0}
\lesssim
\inf_{y\in B(x,r)}M_{p_1}(Tf)(y)
+
\inf_{y\in B(x,r)}M_{p_1}f(y).
\end{equation}
Then, for every $p\in(p_0,q_0)$, every dyadic cube $Q_0\in\mathcal{D}$ and
$f\in L^p(X)$, $g\in L^{p'}(X)$ both supported in $5Q_0$, there exists a sparse
collection $\mathcal{S}\subset\mathcal{D}$, depending on $f$ and $g$, such that
\begin{equation}\label{sparseDomi}
\left|
\int_{Q_0}Tf\cdot g\,d\mu
\right|
\lesssim
\sum_{Q\in\mathcal{S}}
|Q|\,
\langle f\rangle_{p_0,5Q}
\langle g\rangle_{q_0',5Q}.
\end{equation}
\end{lem}
Conditions (i) and (ii) may be viewed as substitutes for the kernel regularity assumptions in the classical Calder\'{o}n-Zygmund theory (see e.g. \cite{A07}). The off-diagonal estimate (i) gives the cancellation of $T$ relative to the semigroup, and (ii) is a Cotlar-type inequality controlled by maximal operators.
\begin{proof}
    The case $m=2$ is proved in \cite[Theorem 5.7]{BFP16}. The same proof applies here with $t^{1/2}\rightarrow t^{1/m}$ and $r^2\rightarrow r^m$.
\end{proof}
\begin{rem}
The estimate \eqref{sparseDomi} can be formulated in terms of averages over dyadic cubes rather than their dilations. More precisely, \eqref{sparseDomi} yields the usual bilinear sparse domination
\begin{equation}\label{bsp}
    |\langle Tf,g\rangle|
\lesssim
\sum_{j=1}^{K_D}
\sum_{Q\in\mathcal{S}_j}
|Q|\,
\langle f\rangle_{p_0,Q}
\langle g\rangle_{q_0^{\prime},Q}=:\sum_{j=1}^{K_D}
\Lambda_{\mathcal{S}_j;p_0,q_0}(f,g),
\end{equation}
for compactly supported bounded functions $f,g$, where each
$\mathcal{S}_j\subset\mathcal{D}^{j}$ is sparse \textnormal{(}see \textnormal{\cite[Section 6]{BFP16}} for a detailed argument\textnormal{)}. 
\end{rem}
\begin{defi}
We say an operator $T$ admits a $(p_0,q_0)$ sparse domination if there exist sparse families $\mathcal{S}_j$ such that $|\langle Tf,g\rangle|$ is bounded by a finite sum of sparse forms as in \eqref{bsp} for every pair of compactly supported bounded function $f,g$. By the existence of a universal sparse family \textnormal{(}see for e.g. \textnormal{\cite{LM16}}\textnormal{)}, this is equivalent to \eqref{bsp}. We write $T\in S(p_0,q_0)$.
\end{defi}
Sparse domination naturally leads to weighted $L^p$ estimates. We next recall the relevant classes of weights. Let $w:X\rightarrow(0,\infty)$ be a locally integrable function. For
$1<p<\infty$, we say that $w$ is an $A_p$ weight if
\[
[w]_{A_p}
:=
\sup_B
\left(
\frac{1}{|B|}\int_Bw\,d\mu
\right)
\left(
\frac{1}{|B|}\int_Bw^{1-p'}\,d\mu
\right)^{p-1}
<\infty
\]
with $p^{\prime}$ the conjugate exponent $p^{\prime}=p/(p-1)$. For $p=1$, we define
$$
[w]_{A_1}
:=
\sup_B
\left(
\frac{1}{|B|}\int_B w\,d\mu
\right)
\left(
\operatorname*{ess\,inf}_{x\in B}w(x)
\right)^{-1}.
$$
For $1<q<\infty$, we say that $w\in RH_q$ if
$$
[w]_{RH_q}
:=
\sup_B
\left(
\frac{1}{|B|}\int_Bw^q\,d\mu
\right)^{1/q}
\left(
\frac{1}{|B|}\int_Bw\,d\mu
\right)^{-1}
<\infty.
$$
For $q=\infty$, we define
$$[w]_{RH_{\infty}}:=\sup_B \left(\operatorname*{ess\,sup}_{x\in B}w(x)\right)\left(\frac{1}{|B|}\int_{B}w\,d\mu\right)^{-1}.$$
By the definition of dyadic system in $X$, the balls in the definition above can be replaced by the dyadic cubes $Q\in\mathcal{D}$, and the corresponding $A_p$ and $RH_{S}$ constants are comparable to those defined with balls.

\begin{lem}\label{wei-lem}
    Let $1\leq p_0<q_0\leq\infty$ and $T\in S(p_0,q_0)$. Then the following assertions hold.

    \textnormal{(i)} For $p_0<p<q_0$ and $w\in A_{p/p_0}\cap RH_{(q_0/p)^{\prime}}$,
    $$\|T\|_{L^{p}(w)\rightarrow L^p(w)}\lesssim\left([w]_{A_{p/p_0}}[w]_{RH_{(q_0/p)^{\prime}}}\right)^{\alpha},$$
    where $\alpha=\text{max}\{1/(p-p_0),(q_0-1)/(q_0-p)\}.$

    \textnormal{(ii)} Further suppose that $(X,d)$ is a connected metric space. For $w\in A_1\cap RH_{(q_0/p_0)^{\prime}}$, there is a function $\psi(w):=\psi([w]_{A_1},[w]_{A_{\infty}},[w]_{RH_{(q_0/p_0)^{\prime}}})$ so that
    $$\|T\|_{L^{p_0}(w)\rightarrow L^{p_0,\infty}(w)}\lesssim\psi(w).$$
\end{lem}
Part (i) follows from \cite[Proposition 6.4]{BFP16} and (ii) follows from \cite{FN19}. 

We conclude this subsection by recalling the following result on universal sparse families. It was proved in \cite{LM16} on $\mathbb{R}^n$. The same proof applies to spaces of homogeneous type.
\begin{lem}\label{universal}
    \textnormal{(\cite[Lemma 4.7]{LM16})} Let $r,s\in[1,\infty]$ and $0<\eta,\eta_0<1$. Given $f,g$, there exist $\eta$-sparse families $\mathcal{S}_{b}\subset \mathcal{D}^b$, $b=1,2,...K_D$, such that for any $\eta_0$-sparse family $\mathcal{S}_0$, it holds that
    $$\Lambda_{\mathcal{S}_0;r,s}(f,g)\lesssim(1-\eta)^{-(1/r+1/s^{\prime})}\eta_0^{-1}\sum_{b=1}^{K_D}\Lambda_{\mathcal{S}_b;r,s}(f,g).$$
\end{lem}

\section{Sparse domination for holomorphic multipliers}

In this section, we prove Theorem \ref{thm1} by considering separately the semigroup part and the holomorphic functional calculus part. In both cases, we use Lemma \ref{spDo} and establish the required estimates.

\subsection{The case $\varphi(\lambda)=e^{-\lambda}$}
We first prove the sparse domination for the semigroup part $V^q(e^{-tL}:t>0)$. In this part, we require $q>2$, and injectivity of $L$ is not needed.

\medskip
\noindent\textbf{Choice of representatives and reductions.} We begin by fixing a representative of $e^{-tL}f$ for $f\in L^2(X)$. For each $t>0$, $e^{-tL}f$ is defined as an element of $L^2(X)$, and hence only up to a set of measure zero. Thus, without fixing representatives for all $t>0$, the $q$-variation $V^q(e^{-tL}f:t>0)$ is not well defined, since the null set may depend on $t$. Indeed, it follows from \cite[Lemma p. 72]{S70} that for any $f\in L^2(X)$ and $m\in \mathbb{N}$, the function $t\mapsto t^m L^me^{-tL}f(x)$ is real-analytic for a.e. $x\in X$. We select such representatives for both $e^{-tL}$ and $Le^{-tL}f$. Now since the semigroup $\{e^{-tL}\}_{t>0}$ is analytic, for every $a,b>0$ and $f\in L^2(X)$, the identity
\begin{equation}
    e^{-aL}f-e^{-bL}f=-\int_{b}^a Le^{-sL}fds\notag
\end{equation}
holds at $L^2(X)$ level. Thus, for the chosen representatives, there exists null set $E\subset X$ such that, for $x\notin E$ and $a,b\in \mathbb{Q}_+$,
\begin{equation}\label{4.1}
    e^{-aL}f(x)-e^{-bL}f(x)=-\int_{b}^a Le^{-sL}f(x)ds.
\end{equation}
Moreover, by our choice of representatives, both sides of \eqref{4.1} are continuous in $a$ and $b$ outside another null set. Therefore,\eqref{4.1} holds for every $a,b>0$ and for a.e. $x\in X$ (the null set does not depend on $a,b$).

\medskip
We next reduce the sparse domination of $V^q$ to the homogeneous $q$-variation operators $\widetilde{V}^q$. Let 
$$\mathcal{M}(\varphi(tL):t>0)f:=\sup_{t>0}|\varphi(tL)f|.$$
Recall that 
\begin{equation}\label{4.2}
V^q(\varphi(tL):t>0)f\sim\mathcal{M}(\varphi(tL):t>0)f+\widetilde{V}^q(\varphi(tL):t>0)f.
\end{equation}
For $f\in L^2(X)$, it follows from the strong continuity ot the semigroup that $e^{-tL}f\rightarrow f$ in $L^2(X)$ when $t\rightarrow0$. Hence, there exists a sequence $t_n\rightarrow0$ such that $e^{-t_n L}f(x)\rightarrow f(x)$ a.e. $x\in X$. Therefore, for every $t>0$ and a.e. $x\in X$,
$$|e^{-tL}f(x)|\leq|e^{-t_n L}f(x)|+|e^{-tL}f(x)-e^{-t_n L}f(x)|\leq|e^{-t_n L}f(x)|+\widetilde{V}^q(e^{-tL}:t>0)f(x).$$
Letting $n\rightarrow\infty$ and then taking the supremum over $t>0$, we obtain that
\begin{equation}\label{4.3}
    \mathcal{M}(e^{-tL}:t>0)f\leq|f|+\widetilde{V}^q(e^{-tL}:t>0)f\quad \text{a.e. on }X.
\end{equation}
For a bounded function $f$ with compact support, the lebesgue differentiation theorem gives $|f(x)|\leq Mf(x)$ a.e. $x\in X$. It is well known that the classical maximal operator $M$ admits a pointwise sparse domination (see e.g. \cite[Section 2]{LM13} in the Euclidean setting, the same argument applies to spaces of homogeneous type). Hence, $Ab(f):=|f|$ also admits a $S(p_0,q_0)$ sparse bound for every $1\leq p_0<q_0\leq\infty$. Thus, it remains to prove the sparse domination of $\widetilde{V}^q$. To apply Lemma \ref{spDo}, we first establish the required off-diagonal estimate. For convenience, we set $Q_t ^N:=c_N^{-1}(tL)^Ne^{-tL}$ for $N\in \mathbb{N}_+$, where $c_N:=\int_0^{\infty}s^{N-1}e^{-s}ds$ is the normalization constant.
\begin{lem}\label{ODE1}
    Under the assumption of Theorem \ref{thm1}, there exists $N_0\in \mathbb{N}$ such that, for every $N\geq N_0$, $t>0$ and two balls $B_1,B_2$ of radius $t^{1/m}$,
\begin{equation}
\|\chi_{B_2} \widetilde{V}^q(e^{-tL}:t>0)(tL)^Ne^{-tL}\chi_{B_1}\|_{{p_0}\rightarrow {q_0}}
\lesssim
|B_1|^{-1/p_0}|B_2|^{1/q_0}
\left(
1+\frac{d(B_1,B_2)}{t^{1/m}}
\right)^{-(\nu+1)}.\notag
\end{equation}
\end{lem}
\begin{proof}
   For every $N\in\mathbb{N}$, it follows from \eqref{4.1} and $q>2$ that for every decreasing $(s_j)_{j\in\mathbb{N}}\subset \mathbb{R_+}$,
    $$\left(\sum_{j\in\mathbb{N}}|e^{-s_j L}Q_t^N f-e^{-s_{j+1}L}Q_{t}^Nf|^q\right)^{1/q}\leq\sum_{j\in\mathbb{N}}\left|\int_{s_{j+1}}^{s_j}Le^{-sL}Q_t ^N f\,ds\right|\leq\int_{0}^\infty \left|Le^{-sL}Q_t^Nf\right|ds$$
    a.e. on $X$, and the null set does not depend on $(s_j)_{j\in\mathbb{N}}$. Since $Le^{-sL}Q_t^N=c_N^{-1}t^N L^{N+1} e^{-(t+s)L}$, by taking supremum over $(s_j)_{j\in\mathbb{N}}$, we obtain that for a.e. $x\in X$,
    \begin{equation}\label{4.4}
        \widetilde{V}^q(e^{-tL}:t>0)\,Q_t^Nf(x)\lesssim\int_{0}^{\infty} \frac{t^N}{(t+s)^{N+1}} |Q_{t+s}^{N+1}f(x)|ds=\int_{t}^\infty |Q_t^{N+1}f(x)|\left(\frac{t}{u}\right)^N\frac{du}{u}.
    \end{equation}
    For $B_i=B(x_i,t^{1/m})$, let $\widetilde{B}_i=B(x_i,u^{1/m})$ ($i=1,2$). By Lemma \ref{deri},
    $$\|\chi_{\widetilde{B}_2}Q_u^{N+1}\chi_{\widetilde{B}_1}\|_{p_0\rightarrow q_0}\lesssim|\widetilde{B}_1|^{-1/p_0}|\widetilde{B_2}|^{1/q_0}\textnormal{exp}\left(-c\left(\frac{d(\widetilde{B_1},\widetilde{B}_2)}{u^{1/m}}\right)^{m/(m-1)}\right).$$
    For $u\geq t$, we have $|\widetilde{B_1}|\geq|B_1|$, and by the doubling property,  $|\widetilde{B}_2|\lesssim(u/t)^{\nu/m}|B_2|$. If $d(x_1,x_2)>4u^{1/m}$, then $d(\widetilde{B}_1,\widetilde{B}_2)\sim d(B_1,B_2)\sim d(x_1,x_2)$. On the other hand, if $d(x_1,x_2)\leq 4u^{1/m}$, then $\textnormal{exp}(-c(d(B_1,B_2)/u^{1/m})^{m/(m-1)})\sim 1$. Consequently, we have that
    \begin{equation}\label{4.5}
        \|\chi_{B_2}Q_u^{N+1}\chi_{B_1}\|_{p_0\rightarrow q_0}\lesssim|{B}_1|^{-1/p_0}|{B_2}|^{1/q_0}\left(\frac{u}{t}\right)^{\nu/mq_0}\textnormal{exp}\left(-c^{\prime}\left(\frac{d({B_1},{B}_2)}{u^{1/m}}\right)^{m/(m-1)}\right).
    \end{equation}
    Thus, Minkowski's inequality gives that
    \begin{align}
        \|\chi_{B_2}\widetilde{V}^q(e^{-tL}:t>0)Q_t^{N+1}(f\chi_{B_1})\|_{L^{q_0}}&\leq\int_{t}^{\infty}\|Q_u^{N+1}(f\chi_{B_1})\|_{L^{q_0}}\left(\frac{t}{u}\right)^{N}\frac{du}{u}\notag\\
        &\lesssim|B_1|^{-1/p_0}|B_2|^{1/q_0}\|f\|_{L^{p_0}(B_1)}\int_t^{\infty}\left(\frac{t}{u}\right)^{N-(\nu/mq_0)}\textnormal{exp}\left(-c^{\prime}Du^{-1/(m-1)}\right)\frac{du}{u},\label{4.6}
    \end{align}
    where we denote $D=d(B_1,B_2)^{m/(m-1)}$. If $u\geq D^{m-1}$, then
    \begin{align} \int_t^{\infty}\left(\frac{t}{u}\right)^{N-(\nu/mq_0)}\textnormal{exp}\left(-c^{\prime}Du^{-1/(m-1)}\right)\frac{du}{u}\lesssim\int_{1}^{\infty}v^{-N-1+(m/q_0)}dv\sim1\notag
    \end{align}
    for $N>\nu/mq_0$. Otherwise for the case $u<D^{m-1}$, by the change of variable $v=Du^{-1/(m-1)}$, the integral in \eqref{4.6} can be written as
    $$\left(\frac{D}{t^{1/(m-1)}}\right)^{-N+(\nu/mq_0)}\int_{0}^{Dt^{-1/(m-1)}}v^{(m-1)(N-(\nu/mq_0))-1}e^{-c^{\prime}v}dv\lesssim\left(\frac{d(B_1,B_2)}{t^{1/m}}\right)^{m(-N+(\nu/mq_0))/(m-1)}$$
    for $N>\nu/mq_0$. Hence, if $N>((\nu+1)(m-1)/m)+(\nu/mq_0)$, the desired off diagonal estimate follows.
\end{proof}
We also have the following Cotlar-type inequality.
\begin{lem}\label{CI1}
    Under the assumption of Theorem \ref{thm1}, for every $x\in X$ and $r>0$, it holds that
    $$\left(
\frac{1}{|B(x,r)|}
\int_{B(x,r)}
|\widetilde{V}^q(e^{-tL}:t>0)e^{-r^mL}f|^{q_0}\,d\mu
\right)^{1/q_0}
\lesssim
\inf_{y\in B(x,r)}M_{p_0}(\widetilde{V}^q(e^{-tL}:t>0)f)(y).$$
\end{lem}
\begin{proof}
    The proof is based on the positive linear domination of the semigroup, which provides a substitute for commutativity between the $q$-variation operators and semigroup. We first note that for $0\leq\{f_j\}_{j=0}^M\subset L^2(X)$, a positive linear operator $R$ on $L^2(X)$ and $q>1$, it holds for a.e. $x\in X$ that
    \begin{equation}
        \left(\sum_{j=0}^{M}|Rf_j(x)|^q\right)^{1/q}=\sup_{\substack{\sum_{j=0}^M b_j^{q^{\prime}}\leq1\\b_j\in\mathbb{Q}_+}}\sum_{j=0}^Mb_j Rf_j(x)=\sup_{\substack{\sum_{j=0}^Mb_j^{q^{\prime}}\leq1\\b_j\in \mathbb{Q}_+}} R\left(\sum_{j=0}^Mb_jf_j\right)(x)\leq R\left[\left(\sum_{j=0}^Mf_j^q\right)^{1/q}\right](x).
    \end{equation}
    We take $b_j\in \mathbb{Q}_+$, such sequences are dense in the unit ball of $\ell^{q^{\prime}}$, and there are only countably many such sequences, so the same null set can be chosen for all $(b_j)_{j\in\mathbb{N}}$. Hence, for every finite decreasing sequence $(t_j)_{j=0}^M\in\mathbb{R}_+$ and $f\in L^2(X)$, we have, for a.e. $x$ that
    \begin{align}
        \left(\sum_{j=0}^{M-1}|e^{-t_j L}e^{-r^m L}f(x)-e^{-t_{j+1}L}e^{-r^m L}f(x)|^q\right)^{1/q}&=\left(\sum_{j=0}^{M-1}|e^{-r^m L}(e^{-t_jL}f-e^{-t_{j+1}L}f)(x)|^q\right)^{1/q}\notag\\
        &\leq\left(\sum_{j=0}^{M-1}R_{r^{m}}(|e^{-t_jL}f-e^{-t_{j+1}L}f|)^q(x)\right)^{1/q}\notag\\
        &\leq R_{r^{m}}\left[\left(\sum_{j=0}^{M-1}|e^{-t_jL}f-e^{-t_{j+1}L}f|^q\right)^{1/q}\right](x).\notag
    \end{align}
    By the continuity of the representatives fixed at the beginning of this subsection, it suffices to take the supremum over finite decreasing sequence of positive rational numbers. Since there are only countably many such sequences, we can choose the same null set for all of them. Consequently, we have that
    \begin{equation}\label{4.8}
        \widetilde{V}^q(e^{-tL}:t>0)e^{-r^m L}f(x)\leq R_{r^{m}}(\widetilde{V}^q(e^{-tL}:t>0)f)(x)\quad \text{a.e. } x\in X.
    \end{equation}
    For convenience, we denote $g:=\widetilde {V}^q(e^{-tL}:t>0)f$ and $B:=B(x,r)$. For $k\in \mathbb{N}$, define $A_k=2^{k+1}B\setminus 2^kB$. First we cover $A_k$ with balls of radius $r$. Choose a maximal set $\{x_l\}_{l=1}^{N_k}\in A_k$ such that $d(x_i,x_j)>r$. By maximality, $A_k\subset \bigcup_l B(x_l,r)$. Moreover, the balls $\{B(x_l,r/2)\}_{l=1}^{N_k}$ are pairwise disjoint and are contained in $2^{k+2}B$. Since $x_l\in 2^{k+1}B$, the doubling property gives $|2^{k+2}B|\lesssim(d(x_B,x_l)+2^{k+2}r)^{\nu}/r^{\nu}|B(x_l,r/2)|\lesssim2^{k\nu}|B(x_l,r/2)|$. Hence,
    $$2^{-k\nu}\sum_{l=1}^{N_k}|2^{k+2}B|\lesssim\sum_{l=1}^{N_k}|B(x_l,r/2)|\leq|2^{k+2}B|,$$
    and $N_k\lesssim2^{k\nu}$. On the otherhand, suppose that $y\in B(x_k,r)$ for $k=1,...,M$, then pairwise disjoint balls $B(x_{l_k},r/2)$ are contained in $B(y,3/2r)$, and $\sum_{k=1}^{M}|B(x_{l_k},r/2)|\leq|B(y,3/2r)|$. Since the doubling property gives $|B(y,3/2r)|\lesssim|B(x_{l_k},r/2)|$, we have $M\lesssim1$. Therefore, there exist constant $C_0>0$, independent of $k$, such that $\sum_{l=1}^{N_k}\chi_{B(x_l,r)}(x)\leq C_0\quad\text{for all }x\in X$. We write $B_{k,l}:=B(x_l,r)$, $l=1,...,N_k$. Using the \textnormal{(GGE$_{p_0,q_0,m}$)} assumption for $R_t$, We obtain, for every $\gamma>0$ that
    \begin{align}
        |B|^{-1/q_0}\|\chi_{B} R_{r^m}(g\chi_{A_k})\|_{L^{q_0}}&\leq|B|^{-1/q_0}\sum_{l=1}^{N_k}\|\chi_{B} R_{r^m}(g\chi_{B_{k,l}})\|_{L^{q_0}}\notag\\
        &\lesssim(1+2^k)^{-\gamma}\sum_{l=1}^{N_k}|B_{k,l}|^{-1/p_0}\|g\|_{L^{p_0}(B_{k,l})}\notag\\
        &\lesssim(1+2^k)^{-\gamma}2^{k\nu/p_0}|2^{k+2}B|^{-1/p_0}N_k^{1/p_0^{\prime}}\left(\sum_{l=1}^{N_k} \|g\|_{L^{p_0}(B_{k,l})}^{p_0}\right)^{1/p_0}\notag\\
        &\lesssim2^{-k(\gamma-\nu)}\langle g\rangle_{p_0,2^{k+2}B}\leq2^{-k(\gamma-\nu)}\inf_{y\in B}M_{p_0}g(y).\label{4.9}
    \end{align}
    Hence, combining \eqref{4.8} with \eqref{4.9} and taking $\gamma>\nu$, we obtain that
    \begin{align}
        \langle\,\widetilde{V}^q(e^{-tL}:t>0)e^{-r^m L}f\,\rangle_{q_0,B}&\leq|B|^{-1/q_0}\sum_{k=0}^{\infty}\|\chi_{B}R_{r^m}(g\chi_{A_k})\|_{L^{q_0}}+|B|^{-1/q_0}\|\chi_{B}R_{r^{m}}(g\chi_{B})\|_{L^{q_0}}\notag\\
        &\lesssim\sum_{k=0}^{\infty}2^{-k(\gamma-\nu)}\inf_{y\in B}M_{p_0}g(y)+\langle g\rangle_{p_0,B}\lesssim\inf_{y\in B} M_{p_0}g(y).
    \end{align}
\end{proof}
\begin{proof}[Proof of Theorem \ref{thm1}: $\varphi=e^{-(\cdot)}$]
Le Merdy and Xu \cite{LMX12} proved the unweighted $L^2$ bounds $\|\widetilde{V}^q(e^{-tL}:t>0)\|_{2\rightarrow2}<\infty$ (actually even $V^q$) for contractive regular analytic semigroup $\{e^{-tL}\}_{t>0}$. Combining this with Lemma \ref{ODE1}, \ref{CI1} and \ref{spDo}, we obtain $\widetilde{V}^q(e^{-tL}:t>0)\in S(p_0,q_0)$. Therefore, by the reduction argument given at the beginning of this subsection, $V^q(e^{-tL}:t>0)\in S(p_0,q_0)$.
\end{proof}
\subsection{The case $\varphi\in H^{\infty}_0(\Sigma_{\mu})$}

Let $\varphi\in\Psi_{a,b}(\Sigma_{\mu})$ for some $a,b>0$. We next prove the sparse domination for the $H^{\infty}_0$ part. In this case, we only require $q\geq 2$.

\medskip
\noindent\textbf{Choice of representatives and reductions.} We fix a representative of $\varphi(tL)f$ for $f\in L^2(X)$. Fix an interval $I=(a,b)\subset \mathbb{R_+}$. The argument relies on properties of the vector valued Sobolev space $W^{1,2}(I;L^2(X))$. For further details, we refer to \cite[Section 2]{HNVW16}. Set $F(t):=\varphi(tL)f\in L^2(X)$. By Lemma \ref{derivative}, $\lambda\varphi^{\prime}(t\lambda)\in H^{\infty}_0(\Sigma_{\theta})$ for each $\theta<\mu$. Hence, \eqref{Hi} and the dominated convergence theorem gives that for every $t_0\in I$,
\begin{align}
\lim_{t\rightarrow0}\frac{F(t_0+t)-F(t_0)}{t}&=\lim_{t\rightarrow0}\frac{1}{2\pi i}\int_{\gamma}\frac{\varphi((t_0+t)\lambda)-\varphi(t_0\lambda)}{t}(\lambda I-L)^{-1}f\,d\lambda\notag\\
&=\frac{1}{2\pi i}\int_{\gamma}\lambda \varphi^{\prime}(t_0\lambda)(\lambda I-L)^{-1}f\,d\lambda=L\varphi^{\prime}(t_0L)f\notag
\end{align}
in $L^2(X)$, that is, $F^{\prime}(t)=L\varphi^{\prime}(tL)f\in L^2(X)$. Moreover, Lemma \ref{derivative} and \eqref{Hi} give that $\|F(t)\|_{L^2(X)}+\|tF^{\prime}(t)\|_{L^2(X)}\lesssim\|f\|_{L^2(X)}$. Thus, both $F$ and $F^{\prime}$ belong to $L^{2}(I; L^2(X))$. The identity
\begin{equation}\label{wd}
\int_{\mathbb{R}_+}F(t)\phi^{\prime}(t)dt=-\int_{\mathbb{R_+}}F^{\prime}(t)\phi(t)dt
\end{equation}
holds for every $\phi\in C_{c}^{\infty}(I)$ in $L^2(X)$ sense. Fix measurable representative of $F,F^{\prime}$ as $F(\cdot,x)$ and $F^{\prime}(\cdot,x)$. By Fubini's theorem, both $F(\cdot,x),F^{\prime}(\cdot,x)\in L^2(I)$. By testing \eqref{wd}, we see that $F^{\prime}(\cdot,x)$ is the weak derivative of $F(\cdot,x)$ a.e. $x\in X$. The null set obtained this way may depend on the test function $\phi$. Therefore, we first choose a countable dense subset of $C_c^{\infty}(I)$, establish the weak derivative identity for all test functions in this subset outside a single null set, and then extend the identity to every $\phi\in C_{c}^{\infty}(I)$ by approximation. Thus, $F(\cdot,x)\in W^{1,2}(I)$ a.e. $x\in X$. By the standard properties of Sobolev space on $I$, for a.e. $x\in X$, $F(\cdot,x)$ admits a locally absolutely continuous representative $\widetilde{F}(\cdot,x)$, such that
\begin{equation}\label{4.12}
    \widetilde{F}(t,x)-\widetilde{F}(s,x)=\int_{s}^t F^{\prime}(r,x)\,dr,\quad s,t\in I.
\end{equation}
Finally, we verify that $\widetilde{F}(t,x)$ is a representative of $F(t)$ for $t\in I$. By our choice of $\widetilde{F}$, we have $\widetilde{F}(t,\cdot)=F(t)$ in $L^2(X)$ a.e. $t\in I$ in $L^2(X)$. Moreover, by \eqref{4.12} and H\"{o}lder's inequality,
\begin{equation}
    \|\widetilde{F}(t_2,\cdot)-\widetilde{F}(t_1,\cdot)\|_{L^{2}(X)}\leq|t_2-t_1|^{1/2}\|F^{\prime}\|_{L^{2}(X;L^{2}(I))}\leq|t_1-t_2|^{1/2}\|F\|_{W^{1,2}(I;L^{2}(X))}.
\end{equation}
Hence, $\widetilde{F}(\cdot,x)$ is continuous in $L^{2}(X)$. On the other hand, $F(\cdot)$ is continuous (in fact, differentiable) in $L^2(X)$ as shown above. Therefore, $\widetilde{F}(t,\cdot)=F(t)$ for every $t\in I$, and provides an absolutely continuous representative. 

We apply this argument to $I_k=(1/k,k)$ for $k\geq2$. We have shown that the representatives are continuous and represent the same function outside a joint null set in $X$. Thus, the representative agree on the same intervals. For this representative, \eqref{4.12} holds for all $s,t>0$ and $\widetilde{F}(t,\cdot)=F(t)$ in $L^2(X)$ for every $t>0$.

\medskip
We then reduce the sparse domination for $V^q$ to that for $\widetilde{V}^q$, following the same argument as in the previous subsection. It suffices to show that for every $f\in L^2(X)$, $\lim_{t\rightarrow0} \varphi(tL)f=0$ in $L^{2}(X)$. Recall that $\varphi\in\Psi_{a,b}(\Sigma_{\nu})$, fix a $0<\delta<\min\{a,1\}$ and define $\psi(z):=z^{-\delta}\varphi(z)$. Then it is easy to see $\psi\in \Psi_{a-\delta,b}(\Sigma_{\mu})$. Moreover, since $L$ is injective, the fractional power $L^{\delta}$ is well-defined by the extended sectorial functional calculus (see e.g. \cite[Section 1.2, 3.1]{H06}) and corresponds to $z^{\delta}$. Hence, by the product rule for the extended functional calculus (see e.g. \cite[Theorem 1.3.2 (c)]{H06}), we have that for $f\in D(L^{\delta})$, $\varphi(tL)=t^{\delta}\psi(tL)L^{\delta}f$. It follows from \eqref{Hi} that
\begin{align}
    \|\psi(tL)\|_{2\rightarrow2}&\lesssim\int_{0}^{\infty}|\psi(t\rho e^{\pm i\theta})|\frac{d\rho}{\rho}\notag=\int_{0}^{\infty}|\psi(t\rho e^{\pm i\theta})|\frac{d\rho}{\rho}\notag\notag\\
    &\lesssim\int_{0}^{\infty}\frac{\rho^{a-\delta}}{1+\rho^{a+b}}\frac{d\rho}{\rho}<\infty.
\end{align}
Therefore, $\|\varphi(tL)f\|_{L^2}\lesssim t^{\delta}\|L^{\delta}f\|_{L^2}\rightarrow0$ as $t\rightarrow0$. Since $0<\delta<1$ and $D(L)$ is dense in $X$, it follows from \cite[Proposition 3.1.1]{H06} that $D(L^{\delta})$ is also dense in $X$. Hence, an approximation argument yields $\lim_{t\rightarrow0} \varphi(tL)f=0$ in $L^{2}(X)$ for every $f\in L^2(X)$. By a similar argument as in \eqref{4.3}, we obtain that for every $f\in L^2(X)$,
$$\mathcal{M}(\varphi(tL):t>0)f\leq\widetilde{V}^q(\varphi(tL):t>0)f\quad\text{a.e. on }X.$$
Thus, it remains to establish the sparse domination for $\widetilde{V}^q$. 

\medskip
The unweighted $L^2$ bounds follows from Lemma \ref{quad} and \ref{lambdaemb}. By the choice of representative, for a.e. $x\in X$, the function $\varphi(\cdot L)f(x)$ has weak derivative $L\varphi^{\prime}(\cdot L)f(x)$. Therefore,
\begin{align}
    \|\widetilde{V}^q(\varphi(tL):t>0)f\|_{L^2}^2&\lesssim\|\varphi(\cdot L)f\|_{L^2(X;\Lambda^{1})}^2=\|\varphi(e^{(\cdot)} L)f\|_{L^2(X;W^{1,2}(\mathbb{R}))}^2\notag\\
    &=\|\varphi(e^{(\cdot)} L)f\|_{L^2(X;\mathbb{R})}^2+\|e^{(\cdot)} L\varphi^{\prime}(e^{(\cdot)} L)f\|^2_{L^2(X;\mathbb{R})}\notag\\
    &=\int_{0}^{\infty}\|\varphi(sL)f\|_{L^2}^2\frac{ds}{s}+\int_{0}^{\infty}\|sL\varphi^{\prime}(sL)f\|_{L^2}^2\frac{ds}{s}\sim\|f\|_{L^2}^2.\label{4.15}
\end{align}
The last identity follows from a change of variables and Fubini's theorem. The last estimate follows from Lemma \ref{derivative} and \ref{quad}.

Next We establish the required off-diagonal estimate.
\begin{lem}\label{ODE2}
    Under the assumption of Theorem \ref{thm2}, there exists $N_0\in\mathbb{N}$ such that, for every $N\geq N_{0}$, $t>0$ and two balls $B_1,B_2$ of radius $t^{1/m}$,
    $$\|\chi_{B_2} \widetilde{V}^q(\varphi(tL):t>0)(tL)^Ne^{-tL}\chi_{B_1}\|_{{p_0}\rightarrow {q_0}}
\lesssim
|B_1|^{-1/p_0}|B_2|^{1/q_0}
\left(
1+\frac{d(B_1,B_2)}{t^{1/m}}
\right)^{-(\nu+1)}.\notag$$
\end{lem}
\begin{proof}
Let $\varphi\in \Psi_{a,b}(\Sigma_\mu)$ for some $a,b>0$. We first estimate $\|\chi_{B_2}\varphi(sL)(tL)^Ne^{-tL} \chi_{B_1}\|_{p_0-q_0}$ for $s,t>0$. For $N\in\mathbb{N}$, define $\Phi_{s,t}(\lambda):=\varphi(s\lambda)(t\lambda)^Ne^{-(t/2) \lambda}$. Then $\Phi_{s,t}\in H^{\infty}_{0}(\Sigma_\mu)$. Hence, by \eqref{H2} and \eqref{H3},
\begin{equation}\label{Psi1}
    \Phi_{s,t}(L)=\int_{\Gamma_{+}} e^{-zL}K^{+}_{s,t}(z)\,dz+\int_{\Gamma_{-}}e^{-zL}K^{-}_{s,t}(z)\,dz=:\Phi_{s,t}^{+}(L)+\Phi_{s,t}^{-}(L),
\end{equation}
where
\begin{equation}\label{Psi2}
    K^{\pm}_{s,t}(z):=\frac{1}{2\pi i}\int_{\gamma_{\pm}}e^{z\lambda}\Phi_{s,t}(\lambda)\,d\lambda,\quad z\in\Gamma_{\pm}.
\end{equation}
Here $\gamma:=\{\rho e^{i\zeta}:\ \rho:\infty\rightarrow0\}\ \cup\ \{\rho e^{-i\zeta}:\ \rho:0\rightarrow\infty\}$ and $\Gamma_{\pm}:=\{re^{\pm i\phi}:\ r:0\rightarrow\infty\}$, for fixed $\omega<\zeta<\mu$ and $\phi\in(\pi/2-\zeta,\pi/2-\omega)$. We estimate $\Phi^{+}_{s,t}$, the same argument applies to $\Phi_{s,t}^{-}$. Let $\lambda=\rho e^{i\zeta}$ and $z=re^{i\phi}$. Since $\textnormal{Re}(e^{z\lambda})=r\rho\cos(\phi+\zeta)$ and $\textnormal{Re}(\lambda=\rho \cos\zeta)$, there exist $c>0$ that for $N>b-1$,
\begin{align}
    |K^{+}_{s,t}(z)|&\lesssim\int_{0}^{\infty}|m(s\rho e^{i\zeta})|(t\rho)^Ne^{-(t\rho\cos\zeta)/2}e^{r\rho\cos(\phi+\zeta)}\,d\rho\notag\\
    &\lesssim t^N\int_{0}^{\infty}{(s\rho)^{a}}{(1+s\rho)^{-a-b}}\rho^N e^{-c\rho(t+r)}\,d\rho\notag\\
    &\leq t^N\int_{0}^{\infty}\min\{(s\rho)^a,(s\rho)^{-b}\}\rho^N e^{-c\rho(t+r)}\,d\rho\notag\\
    &\lesssim\frac{t^N}{(t+r)^{N+1}}\min\left\{\left(\frac{s}{t+r}\right)^a,\left(\frac{s}{t+r}\right)^{-b}\right\}.\label{4.18}
\end{align}
It holds that
\begin{equation}\label{4.19}
   \Phi_{s,t}^{+}(L)e^{-tL/2}=\int_{\Gamma_{+}}K^{+}_{s,t}(z)e^{-(z+(t/2))L}\,dz=\int_{0}^{\infty}K^{+}_{s,t}(re^{i\phi})e^{-(re^{i\phi}+(t/2))L}\,dr.
\end{equation}
For $\omega=re^{i\phi}+(t/2)$, we have $\textnormal{Re}(\omega)=(t/2)+r\cos\phi\sim t+r$, and $|\omega|=[(r\cos\phi+(t/2))^2+(r\sin\phi)^2]^{1/2}\sim t+r$. Therefore, Lemma \ref{ext} \textnormal{(i)} gives that
$$\|\chi_{B_2}v_{(t+r)^{1/m}}^{1/p_0-1/q_0}e^{-\omega L}\chi_{B_1}\|_{p_0\rightarrow q_0}\lesssim\textnormal{exp}\left(-c^{\prime} {d(B_1,B_2)^{m/(m-1)}}{(t+r)^{-1/(m-1)}}\right).$$
Denote $D:=d(B_1,B_2)^{m/(m-1)}$, then by an arguing as in \eqref{2.13}, we obtain that
\begin{equation}
    |B_2|^{1/p_0-1/q_0}\|\chi_{B_2}e^{-\omega L}\chi_{B_1}\|_{p_0\rightarrow q_0}\lesssim e^{-c^{\prime}D(t+r)^{-1/(m-1)}}.\notag
\end{equation}
Hence, by Minkowski's inequality and \eqref{4.19}, we have that
\begin{equation}\label{*}
    |B_2|^{1/p_0-1/q_0}\|\chi_{B_2}\Phi_{s,t}^{+}(L)e^{-tL/2}\chi_{B_1}\|_{p_0\rightarrow q_0}\lesssim\int_{0}^{\infty}\frac{t^N}{(t+r)^{N+1}}\min\left\{\left(\frac{s}{t+r}\right)^a,\left(\frac{s}{t+r}\right)^{-b}\right\}e^{-c^{\prime}D(t+r)^{-1/(m-1)}}dr.\tag{$\ast$}
\end{equation}
To estimate \eqref{*}, we use the following inequality. Let $c,K,R,\alpha,D>0$, then there exists a constant $C(c,K,\alpha)>0$ such that
\begin{equation}\label{A}
    \int_{R}^{\infty}u^{-K-1}e^{-cDu^{-\alpha}}du\leq C(c,K,\alpha)\left(R+D^{1/\alpha}\right)^{-K}.\tag{A}
\end{equation}
The proof is simple. We first note that $\textnormal{LHS}\leq \int_{R}^{\infty} u^{-K-1}\,du\sim R^{-K}$. By the change of variables $v=D^{1/\alpha}u^{-1}$, we also have $\textnormal{LHS}\sim D^{-K/\alpha}\int_{0}^{D^{1/\alpha}R^{-\alpha}}v^{K-1}e^{-cv^{\alpha}}\,dv\lesssim D^{-K/\alpha}$, this proves \eqref{A}. We now consider two cases for \eqref{*}. If $s\leq t< t+r$, then by \eqref{A},
\begin{equation}\label{*1}
    \eqref{*}\leq t^Ns^a\int_t^{\infty}u^{-N-a-1}e^{-c^\prime Du^{-1/(m-1)}}\lesssim\left(\frac{s}{t}\right)^{a}\left(1+\frac{d(B_1,B_2)}{t^{1/m}}\right)^{-m(N+a)}
\end{equation}
If $s>t$, then for $N>b$, it follows from \eqref{A} that
\begin{align}
    \eqref{*}&\leq t^N s^{-b}\int_{t}^su^{-N+b-1}e^{-c^{\prime}Du^{-1/(m-1)}}du+t^N s^a\int_{s}^{\infty}u^{-N-a-1}e^{-c^\prime Du^{-1/(m-1)}}du\notag\\
    &\lesssim\left(\frac{t}{s}\right)^b\left(1+\frac{d(B_1,B_2)}{t^{1/m}}\right)^{-m(N-b)}+\left(\frac{t}{s}\right)^N \left(1+\frac{d(B_1,B_2)}{t^{1/m}}\right)^{-m(N+a)}\notag\\
    &\leq\left(\frac{t}{s}\right)^b\left(1+\frac{d(B_1,B_2)}{t^{1/m}}\right)^{-m(N-b)}.\label{*2}
\end{align}
In the last inequality, we use $t/s<1$ and $N>b$. Since the same estimates as in \eqref{4.18}-\eqref{*2} also hold for $\Phi_{s,t}^{-}(L)$, and $\varphi(sL)=\Phi_{s,t}(L)e^{-tL/2}$ (this follows from the product rule for the holomorphic functional calculus; see e.g. \cite[Theorem 1.3.2]{H06}), we deduce that
\begin{equation}\label{4.22}
    \|\chi_{B_2}\varphi(sL)(tL)^Ne^{-tL}\chi_{B_1}\|_{p_0\rightarrow q_0}\lesssim|B_2|^{-1/p_0+1/q_0}\begin{cases}
        \left(\frac{s}{t}\right)^a\left(1+\frac{d(B_1,B_2)}{t^{1/m}}\right)^{-m(N+a)},\quad\quad0<s\leq t;\\\left(\frac{t}{s}\right)^{b}\left(1+\frac{d(B_1,B_2)}{t^{1/m}}\right)^{-m(N-b)},\quad\quad s>t.
    \end{cases}
\end{equation}

We now consider the $q$-variation. It follows from \eqref{4.12} that for any decreasing sequence $(s_j)_{j\in\mathbb{N}}\subset \mathbb{R}_+$,

\begin{align}
    \left(\sum_{j\in\mathbb{N}}|\varphi(s_j L)(tL)^N e^{-tL}f(x)-\varphi(s_{j+1}L)(tL)^N e^{-tL}f(x)|^q\right)^{1/q}&\leq\sum_{j\in\mathbb{N}}\left|\int_{s_{j+1}}^{s_j}L\varphi^{\prime}(sL)(tL)^Ne^{-tL}f(x)\,ds\right|\notag\\
    &\leq\int_{0}^{\infty}|L\varphi^{\prime}(sL)(tL)^N e^{-tL}f(x)|\,ds\notag
\end{align}
for a.e. $x\in X$. We then take the supremum over $(s_j)_{j\in\mathbb{N}}$ and obtain that
\begin{equation}\label{4.23}
    \widetilde{V}^q(\varphi(tL):t>0)(tL)^Ne^{-tL}f(x)\leq\int_{0}^{\infty}|L\varphi^{\prime}(sL)(tL)^N e^{-tL}f(x)|\,ds.
\end{equation}
Hence, it follows from \eqref{4.22} and Minkowski's inequality that
\begin{align}
    \|\chi_{B_2}\widetilde{V}^q(\varphi(tL):t>0)(tL)^Ne^{-tL}\|_{p_0\rightarrow q_0}&\leq\int_{0}^{\infty}\|\chi_{B_2}L\varphi^{\prime}(sL)(tL)^N e^{-tL}\chi_{B_1}\|_{p_0\rightarrow q_0}\,ds\notag\\
    &\lesssim|B_2|^{-1/p_0+1/q_0}\Bigg[\int_{0}^{t}\left(\frac{s}{t}\right)^a\left(1+\frac{d(B_1,B_2)}{t^{1/m}}\right)^{-m(N+a)}\frac{ds}{s}\notag\\
    &\quad+\int_{t}^{\infty}\left(\frac{t}{s}\right)^{b}\left(1+\frac{d(B_1,B_2)}{t^{1/m}}\right)^{-m(N-b)}\frac{ds}{s}\Bigg].\notag\\
    &\lesssim|B_2|^{-1/p_0+1/q_0}\left(1+\frac{d(B_1,B_2)}{t^{1/m}}\right)^{-m(N-b)}\notag\\
    &\lesssim|B_1|^{-1/p_0}|B_2|^{q_0}\left(1+\frac{d(B_1,B_2)}{t^{1/m}}\right)^{-m(N-b)+\nu/p_0}.
\end{align}
In the last inequality, we use that $|B_1|/|B_2|\lesssim(1+d(B_1,B_2)/t^{1/m})^{\nu}$. This proves the desired inequality for $N>((\nu+(\nu/p_0)+1)/m)+b$.
\end{proof}
We also have the following Cotlar-type inequality. The proof is the same as that of Lemma \ref{CI1}, using the commutativity of $\varphi(tL)$ and $e^{-sL}$.
\begin{lem}\label{CI2}
     Under the assumption of Theorem \ref{thm1}, for every $x\in X$ and $r>0$, it holds that
    $$\left(
\frac{1}{|B(x,r)|}
\int_{B(x,r)}
|\widetilde{V}^q(\varphi(tL):t>0)e^{-r^mL}f|^{q_0}\,d\mu
\right)^{1/q_0}
\lesssim
\inf_{y\in B(x,r)}M_{p_0}(\widetilde{V}^q(\varphi(tL):t>0)f)(y).$$
\end{lem}
\begin{proof}
It follows from the product rule for the holomorphic functional calculus (see e.g. \cite[Theorem 1.3.2]{H06}) that for every $f\in L^2(X)$, $\varphi(sL)e^{-r^mL}f=(\varphi(sL)e^{-r^m L})f=e^{-r^m L}\varphi(sL)f$ in $L^2(X)$. Thus, by the same argument as in the proof of Lemma \ref{CI1}, we obtain that
    \begin{equation}\label{4.25}
         \widetilde{V}^q(\varphi(tL):t>0)e^{-r^m L}f(x)\leq R_{r^{m}}(\widetilde{V}^q(\varphi(tL):t>0)f)(x)\quad \text{a.e. } x\in X.
    \end{equation}
Let $g=\widetilde{V}^q(\varphi(tL):t>0)g$, then the same calculation as in the proof of Lemma \ref{CI1} gives the desired inequality.
\end{proof}
\begin{proof}[Proof of Theorem \ref{thm1}: $\varphi\in H^{\infty}_0$] Combining \eqref{4.15}, Lemma \ref{ODE2} and Lemma \ref{CI2}, we may apply Lemma \ref{spDo} to obtain $\widetilde{V}^q(\varphi(tL):t>0)\in S(p_0,q_0)$. The reduction at the beginning of this subsection then gives $V^q(\varphi(tL):t>0)\in S(p_0,q_0)$.
\end{proof}
\section{Sparse domination under Sobolev-type regularity}
In this section, we prove Theorem \ref{thm2}. The sparse domination of $V^q(e^{-tL}:t>0)$ follows directly from Section 4.1, since a non-negative self-adjoint operator is $0$-accretive and injectivity is not used in the argument in Sec 4.1. Hence, it remains to consider the $\Lambda^\beta$ part. We first apply a dyadic decomposition to $\varphi=\sum \varphi_j$. To sum the sparse bounds for the corresponding $q$-variation operator $V_j^q$, we need to keep track of the dependence of these bounds on the scale parameter. This dependence is not explicit in Lemma \ref{spDo}. We therefore follow the argument by Bernicot, Frey and Petermichl \cite{BFP16} and establish a quantitative sparse domination criterion.

We also note that, in this section, we may allow $q=2$, and the assumption \textnormal{(DGGE$_{p_q,q_0,m}$)} on $\{e^{-tL}\}_{t>0}$ may be weaken to \textnormal{(GGE$_{p_0,q_0,m}$)}. Indeed, the stronger assumption $q>2$ and \textnormal{(DGGE$_{p_q,q_0,m}$)} are only used in the semigroup part.

We next introduce some notation. Let $E_{L}(
\lambda)$ be the spectral measure associated with $L$, and $\Pi_0:=E_{L}(\{0\})$ be the orthogonal projection onto $\ker L$. Then 
\begin{equation}\label{5.1}
    L^2(X)=\ker L\oplus (\ker L)^{\perp}=\ker L\oplus\overline{\textnormal{Ran}\,L}=\Pi_0\, L^2(X)\oplus (I-\Pi_0) L^2(X).
\end{equation}
We also need two classes of elementary operators. For $N\in\mathbb{N}_+$  and $t>0$, recall that $Q_t:=c_N^{-1}(tL)^{N}e^{-tL}$ with $c_N=\int_0^{\infty}s^{N-1}e^{-s}\,ds$. We define $P_t^N:=\int_{t}^{\infty}Q_s^N\,ds/s$. Then, it follows by \eqref{spectralthm} that there exists a polynomial $p$ of degree $N-1$ with $p(0)=1$, such that $P_t^N=p(tL)e^{-tL}-\Pi_0$ (since the spectral multiplier $(t\lambda)^Ne^{-t\lambda}$ vanishes iff $\lambda=0$). A similar calculation also gives 
\begin{equation}\label{5.2}
    I-P_t^N=\int_0^t Q_s\frac{ds}{s}+\Pi_0,\quad P_t^N\rightarrow I-\Pi_0\ (t\rightarrow0)\quad\text{and}\quad P_t^N\rightarrow0\ (t\rightarrow\infty),
\end{equation}
where the convergence is in $L^2(X)$ for every $f\in L^2(X)$.
\subsection{A quantitative sparse domination criterion}
We first recall some basic notation from Section 2.1. For each $Q\in \mathcal{D}^b$ that $B(x_Q,c_0\delta^{k})\subset Q\subset B(x_Q,C_0\delta^k)$, we denote $B_Q:=B(x_Q, r_Q)$, then $r_Q\sim \textnormal{diam}\,Q\sim\delta^k$. Since $r_Q>\textnormal{diam}\,Q$ and $x_{Q}\in Q$, we have $Q\subset B_{Q}$. Fix an integer $l$ such that $2^l>\delta^{-1}$ We define $Q^*:=\{x\in X:d(x,Q)<2^l r_Q\}$. Denote the parent of $Q$ in the same dyadic system by $Q^a$. Then for each $x\in Q^a$, $d(x,x_Q)\leq\textnormal{diam}(Q^a)=\delta^{-1}r_{Q}$. For each $x\in 2^l B_Q$, $d(x,Q)\leq d(x,x_Q)\leq 2^l r_{Q}$. Hence, we have
\begin{equation}\label{Q*}
    Q^a\subset 2^lB_Q\subset Q^{*},\quad Q\in\mathcal{D}.
\end{equation}
Denote $\delta_{p_0}:=1/p_0-1/2$ and $\delta_{q_0}:=1/2-1/q_0$. We have the following quantitative sparse domination criterion.
\begin{thm}\label{sdc}
    Let $1\leq p_0<2<q_0\leq\infty$ and $m\geq 2$. Suppose $L$ is a non-negative self-adjoint operator with the semigroup satisfies \textnormal{(GGE)$_{p_0,q_0,m}$}. Let $\sigma>\nu\max\{1/p_0,1/q_0^{\prime}\}$. Let $H$ be a separable Hilbert space and let $T: L^2(X)\rightarrow L^2(X;H)$ be a linear operator such that $A:=\|T\|_{L^2(X)\rightarrow L^2(X;H)}<\infty$. Assume also that $T\Pi_0=0$ and $T{P}_{t}^N=(({P}_t^N+\Pi_0)\otimes I_H)T$ for $t>0$, $N\in\mathbb{N}$. Let $R\geq 1$. Suppose that there exist $D>0$ and $N_0\in\mathbb{N}_+$, such that for every $N>N_0$, there exists a constant $C_N>0$ for which 
    \begin{equation}\label{OD}
    \|\chi_Fv_{Rt^{1/m}}^{\delta_{q_0}} TQ_t^Nv_{Rt^{1/m}}^{\delta_{p_0}}\chi_E\|_{p_0\rightarrow q_0(H)}\leq C_ND\left(1+\frac{d(E,F)}{Rt^{1/m}}\right)^{-\sigma}\tag{OD}
    \end{equation}
    for all $t>0$ and measurable set $E,F\subset X$, each of which is a ball or an annulus. Then, for every $\varepsilon>0$, $Q_0\in\mathcal{D}$ and a pair of bounded functions $f,g$ supported in $Q_0^{*}$, there exist sparse families $\mathcal{S}_b\subset \mathcal{D}^b$, $b=1,...,K_{D}$, such that
    $$\int_{Q_0}\|Tf\|_H|g|\,d\mu\leq C_\varepsilon (D+AR^{\nu(1/p_0-1/q_0)+\varepsilon})\sum_{b=1}^{K_{D}}\Lambda_{\mathcal{S}_b;p_0,q_0}(f,g).$$
\end{thm}
\begin{rem}
    \textnormal{(i)} Under the assumption $T\Pi_0=0$, the commutation condition in Theorem follows from $Te^{-tL}=(e^{-tL}\otimes I_H)T$. By the spectral theorem, for every $N\in\mathbb{N}$ and $t>0$,
$$\lim_{h\rightarrow0}\left\|\frac{L^N e^{-(t+h)L}-L^N e^{-tL}}{h}+L^{N+1}e^{-tL}\right\|_{2\rightarrow2}\leq\lim_{h\rightarrow0}\sup_{\lambda\ge 0}\left|\frac{\lambda^Ne^{-(t+h)\lambda}-\lambda^Ne^{-t\lambda}}{h}+\lambda^{N+1}e^{-t\lambda}\right|=0,$$
the convergence follows by applying mean value theorem. Since $T$ is bounded, we may differentiate the identity $Te^{-tL}=(e^{-tL}\otimes I_H) T$ $k$ times to obtain $TL^{k}e^{-tL}=(L^k e^{-tL}\otimes I_H)T$. Since $P_{t}^N+\Pi_0=p(tL)e^{-tL}$ for some polynomial $p$, we obtain $TP_{t}^N=((P_t^N+\Pi_{0})\otimes I_H)T$.

We use this commutation condition in place of the Cotlar inequality \textnormal{(ii)} in Lemma \ref{spDo}. In that setting, commutation with the semigroup and the Gaussian estimate imply the Cotlar inequality. In Lemma \ref{CI1}, we use \textnormal{(DGGE$_{p_0,q_0,m}$)} to obtain the domination estimate \eqref{4.8} for the $q$-variation operator. The scalar operators obtained by taking the $V^q$ or $\Lambda^{\eta}$ norm do not in general commute with the semigroup. While, the linear operator $\{\varphi(sL)\}_{s>0}$ does commute with it. This is why we use a Hilbert space-valued argument here. Theorem \ref{sdc} allows us to prove the sparse bound for the $\Lambda^{\beta}$ part in Theorem \ref{thm2} under \textnormal{(GGE$_{p_0,q_0,m}$)}, without other domination assumption.
    
\textnormal{(ii)}Compared with Lemma \ref{spDo}, we state the criterion using off-diagonal estimate \eqref{OD} between balls and annulus, since for our $q$-variation, \eqref{OD} holds for arbitrary measurable sets $E,F$. This allows us to replace the decay condition $\sigma>\nu+1$ by $\sigma>\nu\max\{1/p_0,1/q_0^{\prime}\}$. 

    \textnormal{(iii)}Here in \eqref{OD}, if both sets $E,F$ are contained in some balls that $E\subset B_1=(x_1,r_1)$ and $F\subset B_2=B(x_2,r_2)$. Then, for any $x\in E$, we have $B_1\subset B(x,2r_1)$. By the doubling property, $|B_1|\leq|B(x,2r_1)|\lesssim|B(x,t^{1/m})|(1+(2r_1/t^{1/m}))^{\nu}$. Hence, $\|v_{t^{1/m}}^{-\delta_{p_0}}\chi_{E}\|_{\infty}\lesssim |B_1|^{-\delta_{p_0}}(1+(r_1/t^{1/m}))^{\nu\delta_{p_0}}$. Similarly, $\|v_{t^{1/m}}^{-\delta_{q_0}}\chi_F\|_{\infty}\lesssim|B_2|^{-\delta_{q_0}}(1+(r_2/t^{1/m}))^{\nu\delta_{q_0}}$. Combining these estimate with \eqref{OD}, we obtain that for every $t>0$,
    \begin{equation}\label{GGE5}
        \|\chi_F T\widetilde{Q}_t^N \chi_E\|_{L^{p_0}\rightarrow L^{q_0}(X;H)}\lesssim D|B_1|^{-\delta_{p_0}}|B_2|^{-\delta_{q_0}}\left(1+\frac{r_1}{t^{1/m}}\right)^{\nu\delta_{p_0}}\left(1+\frac{r_2}{t^{1/m}}\right)^{\nu\delta_{q_0}}\left(1+\frac{d(E,F)}{t^{1/m}}\right)^{-\sigma}.
    \end{equation}
\end{rem}
The proof of Theorem \ref{sdc} follows the main steps of Bernicot, Frey and Petermichl \cite{BFP16}, with several changes to obtain the stated dependence on $A$, $D$ and $R$ and to admit the vector valued operators. Fix a dyadic system $\mathcal{D}^{b_0}$ containing $Q_0$, and write $\mathcal{D}=\mathcal{D}^{b_0}$ for simplicity. Let $\widetilde{P}_{t}^N:=P_{t/R^m}^N$. For $f\in L^2(X)$, define the sharp maximal operator
$$T^{\#}_Nf(x):=\sup_{\substack{Q\in\mathcal D\\x\in Q}}\left\langle\|T\widetilde{P}_{r_{Q}^m}^Nf\|_H\right\rangle_{q_0,Q}.$$
\noindent{\textbf{Step 1: weak $(p_0,p_0)$ of $T^{\#}_N$.}} 

In this step, We show that $T^{\#}_N$ satisfies the following properties.
\begin{prop}\label{Property}

 \textnormal{(i)} For every $f\in L^2(X)$, we have almost everywhere $\|Tf(\cdot,x)\|_{H}\leq T^{\#}_Nf(x)$.\\
 \textnormal{(ii)} For each $p_0\leq s<2$, set $\delta_s=1/s-1/2$. There exists a constant $C_s>0$, such that $$\|T^{\#}_N\|_{2\rightarrow2}\leq C_{s} A R^{\nu(\delta_{q_0}+\delta_s)}.$$\\
 \textnormal{(iii)} For large $N$, the operator $T^{\#}_{N}f$ is of weak type $(p_0,p_0)$. More precisely, for $N>N_0$ and every $p_0\leq s<2$, there exists a constant $C_{s}^{\prime}$>0, such that
 $$\|T^{\#}_N\|_{{p_0}\rightarrow {p_0,\infty}}\leq C_{s}^{\prime}(D+AR^{\nu(1/p_0-1/q_0+\delta_s)}).$$
\end{prop}
 Let $M^{\mathcal{D}}$ denote the dyadic Hardy-Littlewood maximal operator associated with $\mathcal{D}$. It is bounded on $L^2(X)$.
\begin{proof}[Proof of Proposition 5.2 \textnormal{(i)}] Fix $f\in L^2(X)$. Since $T\Pi_{0}=0$, it follows from \eqref{5.2} and the boundedness of $T$ that, $\|T\widetilde{P}_{\delta^{mk}}^Nf-Tf\|_{L^2(X;H)}\rightarrow0$ as $k\rightarrow\infty$. The $L^2$-boundedness of $M^{\mathcal{D}}$ gives
$$\|M^{\mathcal{D}}(\|T\widetilde{P}_{\delta^{mk}}^Nf-Tf\|_{H})\|_{L^2(X)}\lesssim\|T\widetilde{P}_{\delta^{mk}}^Nf-Tf\|_{L^2(X;H)}\rightarrow0.$$
Hence, there exists a subsequence $\{k_j\}_{j\in\mathbb{N}}$ such that $M^{\mathcal{D}}(\|T\widetilde{P}_{\delta^{mk_{j}}}^Nf-Tf\|_H)(x)\rightarrow0$ for a.e. $x\in X$. For such $x$, let $Q_k\in\mathcal{D}$ be the cube of generation k in $\mathcal{D}$ containing $x$, then $r_{Q_k}=\delta^k$. By the H\"{o}lder's inequality,
\begin{align*}
    \langle\|Tf\|_H\rangle_{1,Q_{k_j}}&\leq\langle\|T\widetilde{P}_{\delta^{mk_j}}^Nf\|_H\rangle_{1,Q_{k_j}}+\langle\|T\widetilde{P}_{\delta^{mk_{j}}}^Nf-Tf\|_H\rangle_{1,Q_{k_j}}\\
    &\leq\langle\|T\widetilde{P}_{\delta^{mk_j}}^Nf\|_H\rangle_{q_0,Q_{k_j}}+M^{\mathcal{D}}(\|T\widetilde{P}_{\delta^{mk_{j}}}^Nf-Tf\|_H)(x).
\end{align*}
    Letting $j\rightarrow\infty$ and applying the differentiation theorem, we obtain \textnormal{(i)}.
\end{proof}
For the proof of \textnormal{(ii)}, we first recall an extension theorem for Hilbert space. Let $R: L^{p_0}(X)\rightarrow L^q(X)$ be a bounded linear operator, then for every $F\in L^{p}(X;H)$, 
\begin{equation}\label{he}
\|(R\otimes I_H)F\|_{L^{q_0}(X;H)}\leq \|R\|_{p\rightarrow q}\|F\|_{L^{p_0}(X;H)}.
\end{equation}
In our case $1\leq p_0<2<q_0\leq \infty$, this follows simply from a density argument. Let $(e_k)$ be an orthonormal basis of $H$. We first prove \eqref{he} for finite sum $F=\sum_{k=1}^n e_k f_k$ using Minkowski's inequality. The estimate then extends to every $F\in L^{p_0}(X;H)$ by density. See \cite[Section 2]{HNVW16} for more general results. We need the following estimate for the proof of \textnormal{(ii)}.
\begin{lem}\label{ave1}
    Let $B=B(x_0,r)$ be a ball and $R\geq 1$. Then for every $F\in L^2(X;H)$ and a constant $c_1\leq a\leq c_2r$, where $c_1,c_2>0$, we have that
    $$\langle\|((\widetilde{P}_{a^m}^N+\Pi_0)\otimes I_H)F\|_H\rangle_{q_0,B}\lesssim R^{v(\delta_q+\delta_s)}\inf_{y\in B} M_{s}(\|F\|_H)(y),$$
    where $M_s$ denote the maximal operator and the implicit constant is independent of $B,R$ and $F$.
\end{lem}
\begin{proof}[Proof of Lemma \ref{ave1}]
    Let $A_1(B):=4B$ and $A_j(B):=2^{j+1}B\,\backslash \,2^j B$ for $j\geq 2$. First we note that for $p_0\leq s<2$, H\"{o}lder's inequality gives \textnormal{(GGE)$_{p_0,q_0,m}$}$\Rightarrow$\textnormal{(GGE)$_{s,q_0,m}$}. Thus by \cite[Theorem 1.2]{BK05}, there exist $c>0$, such that for every $t>0$,
    \begin{equation}
    \|\chi_{B}v_{t^{1/m}}^{\delta_{q_0}}e^{-tL}v_{t^{1/m}}^{\delta_{s}}\chi_{A_j(B)}\|_{s\rightarrow q_0}\lesssim\textnormal{exp}\left(-c\left(\frac{d(B,A_j(B))}{t^{1/m}}\right)^{m/(m-1)}\right).\notag
    \end{equation}
    Using the same argument as in the proof of Lemma \ref{deri}, we can extend this estimate to $(tL)^Ne^{-tL}$ for all $N\in\mathbb{N}$. Since there exists a polynomial $p$ such that $P_t+\Pi_0=p(tL)e^{-tL}$, we obtain that for $j\geq 2$,
    \begin{equation}
        \|\chi_{B}v_{a/R}^{\delta_{q_0}}(\widetilde{P}_{a^m}^N+\Pi_0)v_{a/R}^{\delta_{s}}\chi_{A_j(B)}\|_{s\rightarrow q_0}\lesssim e^{-c^{\prime} (2^jR)^{m/(m-1)}}.\notag
    \end{equation}
    For $j=1$, the same estimate holds with the exponential factor replaced by 1. For every $y\in B$, we have $B\subset B(y,2r)$. By doubling and $a\sim r$, $|B|\leq|B(y,2r)|\lesssim|B(x,a/R)|R^{\nu}$, which gives $\|v_{a/R}^{\delta_{q_0}}\chi_{B}\|_{\infty}\gtrsim|B|^{\delta_{q_0}}R^{-\nu\delta_{q_0}}$. A similar calculation gives $\|v_{a/R}^{
    \delta_{s}
    }\chi_{A_j(B)}\|_{\infty}\gtrsim|2^{j+1}B|^{\delta_{s}}2^{-j\nu\delta_{s}}R^{-\nu\delta_{s}}$. Hence, for $j\geq 2$,
    \begin{equation}\label{5.6}
         \|\chi_{B}(\widetilde{P}_{a^m}^N+\Pi_0)\chi_{A_j(B)}\|_{s\rightarrow q_0}\lesssim R^{\nu(\delta_s+\delta_{q_0})} 2^{j\nu\delta_{s}}|B|^{-\delta_{q_0}}|2^{j+1}B|^{-\delta_s}e^{-c^{\prime} (2^jR)^{m/(m-1)}}.
    \end{equation}
    Consequently, for $j\geq 2$, it follows from \eqref{he} that
    \begin{align}
        \langle\|((\widetilde{P}_{a^m}^N+\Pi_0)\otimes I_H)(F\chi_{A_j(B)})\|_{H}\rangle_{q_0,B}&=\langle\|((\chi_{B}(\widetilde{P}_{a^m}^N+\Pi_0)\chi_{A_j(B)})\otimes I_H)(F\chi_{A_j(B)})\|_{H}\rangle_{q_0,B}\notag\\
        &\lesssim R^{\nu(\delta_{s}+\delta_{q_0})}2^{j\nu\delta_{s}}e^{-c^{\prime}(2^jR)^{m/(m-1)}}|B|^{-1/2}|2^{j+1}B|^{-\delta_s}\|F\chi_{A_j(B)}\|_{L^s(X;H)}\notag\\
        &\lesssim R^{\nu(\delta_{s}+\delta_{q_0})}2^{j\nu/s}e^{-c^{\prime}(2^jR)^{m/(m-1)}}\langle \|F\|_{H}\rangle_{s,2^{j+1}B}\notag\\
        &\leq R^{\nu(\delta_{s}+\delta_{q_0})}2^{j\nu/s}e^{-c^{\prime}(2^jR)^{m/(m-1)}}\inf_{y\in B} M_s(\|F\|_H)(y).\label{5.7}
    \end{align}
    For $j=1$, the same argument gives the estimate with out exponential factor $e^{-c^{\prime}(2^jR)^{m/(m-1)}}$. 
    Finally, summing over $j\geq 1$ and using $R\geq1$, we obtain the desired inequality.
\end{proof}
Proposition 5.2 \textnormal{(ii)} follows from Lemma \ref{ave1}.
\begin{proof}[Proof of Proposition 5.2 \textnormal{(ii)}]
    Fix $x\in X$ and $f\in L^2(X)$. For any dyadic cube $Q$ containing $x$, choose a ball $B(x_Q,r)$ such that $Q\subset B$, $r\sim r_Q$ and $|B|\sim|Q|$. Using the commutation relation above, we obtain that
    \begin{align}
        \langle\|T\widetilde{P}_{r_{Q}^m}^Nf\|\rangle_{q_0,Q}&\lesssim\langle\|((\widetilde{P}_{r_{Q}^m}^N+\Pi_0)\otimes  I_H)\,Tf\|\rangle_{q_0,B}\notag\\
        &\lesssim R^{\nu(\delta_s+\delta_{q_0})}M_s(\|Tf\|_H)(x),\label{5.8}
    \end{align}
    where the last inequality follows from Lemma \ref{ave1}. Hence, $T^{\#}_Nf(x)\lesssim R^{\nu(\delta_s+\delta_{q_0})}M_{s}(\|Tf\|_H)(x)$. Since $s<2$, the $L^2$-boundedness of $M_s$ gives that
    $$\|T^{\#}_Nf\|_{L^2(X)}\lesssim R^{\nu(\delta_s+\delta_{q_0})}\|Tf\|_{L^{2}(X;H)}\leq AR^{\nu(\delta_{s}+\delta_{q_0})}\|f\|_{L^2(X)}.$$
    This proves \textnormal{(ii)}.
\end{proof}
To prove \textnormal{(iii)}, we use the following lemma from \cite{A07}.
\begin{lem}\label{weak}\textnormal{(Theorem 2.1, \cite{A07})}
    Let $p_0\in[1,2)$. Suppose $T$ is a bounded sublinear operator operator on $L^2(X)$. Let $F_r$ ($r>0$) be a family of linear operators acting on $L^2(X)$. Assume that for $j\geq 2$
    $$\left(\frac{1}{|2^{j+1}B|}\int_{A_j(B)}|T(I-F_{r(B)})f|^2\,d\mu\right)^{1/2}\leq g_{1,j}(B)\langle f\rangle_{p_0,B},$$
    and for $j\geq 1$
    $$\left(\frac{1}{|2^{j+1}B|}\int_{A_j(B)}|F_{r(B)}f|^2\,d\mu\right)^{1/2}\leq g_{2,j}(B)\langle f\rangle_{p_0,B}$$
    for all ball $B$ with $r(B)$ the radius of $B$ and all $f$ supported on $B$. If
    $$\Gamma_1:=\sup_B\sum_{j\geq2}\frac{|2^{j+1}B|}{|B|}g_{1,j}(B)<\infty\quad\text{and}\quad\Gamma_2:=\sup_B\sum_{j\geq2}\frac{|2^{j+1}B|}{|B|}g_{2,j}(B)<\infty,$$
    then $T$ is of weak type $(p_0,p_0)$. More precisely, $\|T\|_{p_0\rightarrow p_{0},\infty}\lesssim \Gamma_{1}+\|T\|_{2\rightarrow2}(1+\Gamma_2)$.
\end{lem}
This theorem was originally stated on $\mathbb{R}^n$, but the same proof applies to homogeneous spaces. We make one small change to the proof: we apply the Calder\'{o}n-Zygmund decomposition at height $\alpha/K$ instead of $\alpha$, where $K:=\Gamma_{1}+\|T\|_{2\rightarrow2}(1+\Gamma_2)$. This gives the weak bound in our statement, while the original proof only gives the weaker bound $\|T\|_{p_0\rightarrow p_0,\infty}\lesssim(1+K^2)^{1/p_0}$. We can now prove \textnormal{(iii)} of Proposition \ref{Property}.
\begin{proof}[Proof of Proposition \ref{Property} \textnormal{(iii)}]
We apply Lemma \ref{weak} with $F_{r(B)}:=\Pi_0+\widetilde{P}_{r(B)^m}^N$ for suitable $N\in\mathbb{N}$ to prove the result, we will choose $N$ later in the proof. Since \textnormal{(ii)} already gives the $L^{2}$ bound, it remains to verify the two average estimates. For simplicity, we write $r=r(B)$. We first prove the second estimate. 

Since \textnormal{(GGE$_{p_0,q_0,m}$)}$\Rightarrow$\textnormal{(GGE$_{p_0,2,m}$)}, it follows from \cite[Theorem 2.1]{BK05} and an argument similar to the proof of Lemma \ref{ave1} that 
$$\|\chi_{A_j(B)}(\widetilde{P}_{r^m}^N+\Pi_0)v_{r/R}^{\delta_{p_0}} \chi_{B}\|_{p_0\rightarrow2}\lesssim\textnormal{exp}\left(-c\left(\frac{d(A_{j}(B),B)R}{r}\right)^{m/(m-1)}\right).$$
Hence, for $j\geq 2$,
\begin{equation}\label{5.9}
\|\chi_{A_j(B)}(\widetilde{P}_{r^m}^N+\Pi_0) \chi_{B}\|_{p_0\rightarrow2}\lesssim R^{\nu\delta_{p_0}} e^{-c(2^j R)^{m/(m-1)}}|B|^{-\delta_{p_0}}.
\end{equation}
Thus, for $f$ supported in $B$, we have that
\begin{equation}\label{5.10}
|2^{j+1}B|^{-1/2}\|(\widetilde{P}_{r^m}^N+\Pi_0)f\|_{L^2(A_j(B))}\lesssim R^{\nu\delta_{p_0}}e^{-c(2^j R)^{m/(m-1)}}\frac{|B|^{1/2}}{|2^{j+1}B|^{1/2}}\langle f\rangle_{p_0,Q}.
\end{equation}
For $j=1$, the same estimate holds without the \textnormal{exp} factor. Summing over $j$ and using $R\geq1$, we obtain that
\begin{equation}\label{Gamma2}
    \Gamma_2\lesssim\sup_{B}\sum_{j\geq 1}R^{\nu\delta_{p_0}}\frac{|2^{j+1}B|^{1/2}}{|B|^{1/2}}e^{-c2^{jm/(m-1)}}\lesssim R^{\nu\delta}\sum_{j\geq 1}2^{j\nu/2}e^{-c2^{jm/(m-1)}}\lesssim R^{\nu\delta_{p_0}}.
\end{equation}

We next verify the first estimate. Fix $x\in A_{j}(B)$ and $f$ supported in $B$. We split the supremum defining $T^{\#}_Nf(x)$ into two parts, according to whether $r_{Q}\geq2^{j-1}r$ and $r_{Q}< 2^{j-1}r$. Denote these two parts by $T^{\#}_{N,l}f(x)$ and $T^{\#}_{N,s}f(x)$. Then $T^{\#}_Nf(x)\leq T^{\#}_{N,l}f(x)+ T^{\#}_{N,s}f(x)$. The difference in the estimates below mainly lies in the choice of the set containing $Q$ to apply \eqref{GGE5}. We start with the case $r_{Q}< 2^{j-1}r$. We prove that
\begin{equation}\label{(**)}
    \int_{0}^{\infty} \|\chi_{\widetilde{A}_j(B)}T \widetilde{Q}_t^N(I-\Pi_0-\widetilde{P}_{r^m}^N)f\|_{L^{q_0}(X;H)}\frac{dt}{t}\lesssim D|B|^{-\delta_{p_0}}|2^{j+2}B|^{-\delta_{q_0}}2^{-j(\sigma-\nu\delta_{q_0})}\|f\|_{L^{p_0}},\tag{$\ast\ast$}
\end{equation}
where $j\geq2$, $\widetilde{A}_j(B):=2^{j+2}B\,\backslash\,2^{j-1}B$ and $\widetilde{Q}_{t}^N:=Q^N_{t/R^m}$.
\begin{proof}[Proof of \eqref{(**)}]
Since $T$ is bounded and linear, by \eqref{5.2} we have that
\begin{equation}\label{5.12}
    T\widetilde{Q}_{t}^N(I-\Pi_0-\widetilde{P}_{r^m}^N)f=\int_{0}^{r^m}T\widetilde{Q}_{t}^N \widetilde{Q}_{s}^Nf\,\frac{ds}{s}\sim\int_{0}^{r^m}\frac{t^Ns^N}{(t+s)^{2N}}T\widetilde{Q}_{t+s}^{2N}f\,\frac{ds}{s}.
\end{equation}
    By Minkowski's inequality for vector valued function,
\begin{align}
    \textnormal{LHS of \eqref{(**)}}&\lesssim\int_{0}^{\infty}\int_{0}^{r^m} \frac{t^Ns^N}{(t+s)^{2N}}\|\chi_{\widetilde{A}_j(B)}T\widetilde{Q}_{t+s}^{2N}f\|_{L^{q_0}(X;H)}\frac{ds}{s}\frac{dt}{t}\notag\\
    &=\int_{0}^{\infty}\left(\int_{0}^{\min\{1,r^m/u\}}v^{N-1}(1-v)^{N-1}\,dv\right)\|\chi_{\widetilde{A}_j(B)}T\widetilde{Q}_{u}^{2N}f\|_{L^{q_0}(X;H)}u^{-1}\,du\notag\\
    &\lesssim\int_{0}^{\infty}\min\{1,r^{mN}/u^N\}\|\chi_{\widetilde{A}_j(B)}T\widetilde{Q}_{u}^{2N}(f\chi_B)\|_{L^{q_0}(X;H)}u^{-1}\,du.\label{5.13}
\end{align}
Here we used that change of variables $u=t+s$, $v=s/(t+s)$, together with the estimate $\int_{0}^a v^{N-1}(1-v)^{N-1}dv\leq\int_{0}^a v^{N-1}dv\sim a^N$. Moreover, by \eqref{GGE5} we have that
\begin{align}
    \|\chi_{\widetilde{A}_j(B)}T\widetilde{Q}_{u}^{2N}\chi_{B}\|_{p_0\rightarrow q_0(H)}&\lesssim D|B|^{-\delta_{p_0}}|2^{j+2}B|^{-\delta_{q_0}}\left(1+\frac{r}{u^{1/m}}\right)^{\nu\delta_{p_0}}\left(1+\frac{2^{j+2}r}{u^{1/m}}\right)^{\nu\delta_{q_0}}\left(1+\frac{2^jr}{u^{1/m}}\right)^{-\sigma} \notag\\
    &\sim D|B|^{-\delta_{p_0}}|2^{j+2}B|^{-\delta_{q_0}}\begin{cases}
        1\quad &\text{if}\ u^{1/m}\geq 2^j r;\\
        (2^jr u^{-1/m})^{\nu\delta_{q_0}-\sigma}\quad &\text{if}\ r\leq u^{1/m}<2^jr;\\
        2^{j(v\delta_{q_0}-\sigma)}(ru^{-1/m})^{\nu\delta_{p_0}+\nu\delta_{q_0}-\sigma}\quad &\text{if}\ 0<u^{1/m}<r.\notag
    \end{cases}
\end{align}
Combining these estimates and letting $w=ru^{-1/m}$, it follows that for $N>(\sigma-\nu\delta_{q_0})/m$,
\begin{align}
    \textnormal{LHS of \eqref{(**)}}&\lesssim D|B|^{-\delta_{p_0}}|2^{j+1}B|^{-\delta_{q_0}}\|f\|_{L^{p_0}}\Bigg(\int_{0}^{2^{-j}}w^{mN-1}dw\notag\\
    &\quad+\int_{2^{-j}}^{1}w^{mN-1}(2^jw)^{\nu\delta_{q_0}-\sigma}dw+\int_{1}^{\infty}2^{j(\nu\delta_{q_0}-\sigma)}w^{\nu(1/p_0-1/q_0)-\sigma-1}dw\Bigg)\notag\\
    &\lesssim D|B|^{-\delta_{p_0}}|2^{j+2}B|^{-\delta_{q_0}}2^{-j(\sigma-\nu\delta_{q_0})}\|f\|_{L^{p_0}}.\notag
\end{align}
This completes the proof of \eqref{(**)}.
\end{proof}
We now estimate $T^{\#}_{N,s}$. Recall that $x$ is a fixed point in $A_j(B)$. Let $Q\in \mathcal{D}$ be a dyadic cube containing $x$ with $r_{Q}<2^{j-1}r$. Thus, $\textnormal{diam}(Q)<r_{Q}<2^{j-1}Q$, so $Q\subset \widetilde{A}_j(B)$. By the definition of $P_t^N$,
\begin{equation}\label{5.14}
T\widetilde{P}_{r_{Q}^m}^N(I-\Pi_0-\widetilde{P}_{r^m}^N)f=\int_{r_{Q}^m}^{\infty}T\widetilde{Q}_{t}^N(I-\Pi_0-\widetilde{P}_{r^m}^N)f\,\frac
{dt}{t}.
\end{equation}
Consequently, for a.e. $y\in Q$,
\begin{equation}
    \|T\widetilde{P}_{r_{Q}^m}^N(I-\Pi_0-\widetilde{P}_{r^m})f(\cdot,y)\|_{H}\leq\left(\int_{0}^{\infty}\|T\widetilde{Q}_{t}^N(I-\Pi_0-\widetilde{P}_{r^m}^N)f(\cdot,y)\|_{H}\frac{dt}{t}\right)\chi_{\widetilde{A_j}}(y)=: F_j(y).\notag
\end{equation}
Taking averages over $Q$ and then the supremum, we obtain $T^{\#}_{N,s}(I-\Pi_0-\widetilde{P}_{r^m}^N)f(x)\leq M_{q_{0}}F_j(x)$. By Minkowski's inequality, \eqref{(**)} and Kolmogorov’s inequality (see e.g. \cite[Lemma 5.16]{D01}) for $M_{q_0}$, we have that
\begin{align}
    \|T^{\#}_{N,s}(I-\Pi_0-\widetilde{P}_{r^m}^N)f\|_{L^2(A_j(B))}&\lesssim |A_j(B)|^{1/2-1/q_0}\|F_j\|_{L^{q_0}}\notag\\
    &\leq |A_j(B)|^{\delta_{q_0}}\int_{0}^{\infty} \|\chi_{\widetilde{A}_j(B)}T \widetilde{Q}_t^N(I-\Pi_0-\widetilde{P}_{r^m}^N)f\|_{L^q(X;H)}\frac{dt}{t}\notag\\
    &\lesssim  D|2^{j+1}B|^{\delta_{q_0}}|B|^{-\delta_{p_0}}|2^{j+2}B|^{-\delta_{q_0}}2^{-j(\sigma-\nu\delta_{q_0})}\|f\|_{L^{p_0}(B)}\notag\\
    &\sim D|B|^{1/2}2^{-j(\sigma-\nu\delta_{q_0})}\langle f\rangle_{p_0,B}.\label{5.15}
\end{align}

We next estimate $T^{\#}_{N,l}$. Let $Q$ be a dyadic cube such that $x\in A_{j}(B)\cap Q$ and $r_Q\geq 2^{j-1}r$. By the definition of $\widetilde{P}_{t}^N$ and \eqref{5.12},
\begin{equation}\label{5.16}
T\widetilde{P}^N_{r_{a}^m}(I-\Pi_0-\widetilde{P}_{r^m})f\sim\int_{r_{Q}^m}^\infty\int_{0}^{r^m}\frac{t^Ns^N}{(t+s)^{2N}}T\widetilde{Q}_{t+s}^{2N}f\,\frac{ds}{s}\frac{dt}{t}.
\end{equation}
For any $y\in Q$, we have $d(x,y)\leq\textnormal{diam}(Q)<r_Q$, so $Q\subset B(x,r_Q)$. Moreover, since $x\in A_j(B)$ and $8r_Q\geq 2^{j+2}r$, we have $2^{j+1}B\subset B(x,8r_Q)$. Since $(t+s)^{1/m}\geq r_{Q}>r$, by \eqref{GGE5}, we estimate that
$$\|\chi_{Q}T\widetilde{Q}_{t+s}^{2N}\chi_{B}\|_{p_0\rightarrow q_0(H)}\leq\|\chi_{B(x,r_Q)}T\widetilde{Q}_{t+s}^{2N}\chi_{B}\|_{p_0\rightarrow q_0(H)}\lesssim|B|^{-\delta_{p_0}}|B(x,r_Q)|^{-\delta_{q_0}}.$$
Since $f$ is supported in $B$, Minkowski's inequality gives
\begin{align}
    \langle \|T\widetilde{P}_{r_{Q}^m}^N(I-\Pi_0-\widetilde{P}_{r^m}^N)f\|_{H}\rangle_{q_0,Q}&\lesssim|Q|^{-1/q_0}\int_{r_{Q}^m}^\infty\int_{0}^{r^m}\frac{t^Ns^N}{(t+s)^{2N}}\|\chi_{Q}T\widetilde{Q}_{t+s}^{2N}f\|_{L^{q_0}(X;H)}\,\frac{ds}{s}\frac{dt}{t}\notag\\
    &\lesssim D|B|^{-\delta_{p_0}}|B(x,r_{Q})|^{-1/2}\|f\|_{L^{p_0}(B)}\int_{r_{Q}^m}^\infty\int_{0}^{r^m}t^{-N-1}s^{N-1}\,dsdt\notag\\
    &\lesssim D|B|^{1/2}|2^{j+1}B|^{-1/2}r^{mN}r_{Q}^{-mN}\langle f\rangle_{p_0,B}\notag\\
    &\lesssim D|B|^{1/2}|2^{j+1}B|^{-1/2}2^{-jmN}\langle f\rangle_{p_0,B}.\label{5.17}
\end{align}
By taking supremum over cubes $Q$, we obtain that
$$T^{\#}_{N,l}(I-\Pi_0-\widetilde{P_{r^m}^N})f(x)\lesssim D|B|^{1/2}|2^{j+1}B|^{-1/2}2^{-jmN}\langle f\rangle_{p_0,B}$$
for $x\in A_j$. Hence,
\begin{equation}\label{5.18}
    \|T^{\#}_{N,l}(I-\Pi_0-\widetilde{P_{r^m}^N})f\|_{L^{2}(A_j(B))}\lesssim D|B|^{1/2}2^{-jmN}\langle f\rangle_{p_0,B}.
\end{equation}
Combining \eqref{5.18} with \eqref{5.15}, and choosing $N>(\sigma-\nu\delta_{q_0})/m$, we obtain that
\begin{equation}\label{Gamma1}
    \Gamma_1\lesssim\sup_BD\sum_{j\geq 2}  2^{-j(\sigma-\nu\delta_{q_0})}\frac{|2^{j+1}B|^{1/2}}{|B|^{1/2}}\lesssim D\sum_{j\geq 2}2^{-j(\sigma-(\nu/q^{\prime}))}\lesssim D,
\end{equation}
where the last inequality follows from our assumption in Theorem \ref{sdc} that $\sigma>\nu/q_0^{\prime}$. Finally, \eqref{Gamma1}, \eqref{Gamma2}, Proposition \ref{Property} \textnormal{(ii)} and Lemma \ref{weak} give the desired estimate for Proposition \ref{Property} \textnormal{(iii)}.
\end{proof}
\noindent\textbf{Step 2: sparse domination by recursion.}

We now prove Theorem \ref{sdc}. Recall the definition of $Q^{*}$ and the properties in \eqref{Q*}. As above, we work in a fixed dyadic grid $\mathcal{D}^{b_0}$ containing $Q_0$, and write $\mathcal{D}=\mathcal{D}^{b_0}$. We fix an integer $N>\max\{(\sigma-\nu\delta_{q_0})/m,N_0\}$, and write $T^{\#}:=T^{\#}_N$.
\begin{proof}[Proof of Theorem \ref{sdc}]
We start the recursion from $Q_0$. Let $f,g$ be bounded functions supported in $Q^*_0$. We write $K:=D+AR^{\nu(1/p_0-1/q_0+\delta_{s})}$. For large $C_0>0$ (which will soon be chosen), we define
$$E_{Q_0}:=\left\{x\in Q_0:\max\{M_{p_0,Q_0}^{\mathcal{D}}(f\chi_{Q^*}),T^{\#}_{Q_0}(f\chi_{Q^*})/K\}>C_0\langle f\rangle_{_{p_0,Q^*}}\right\},$$
where 
$$M_{p_0,Q_0}^{\mathcal{D}}f(x):=\sup_{\substack{x\in Q\in\mathcal{D}\\Q\subset Q_0}}\inf_{y\in Q}M_{p_0}f(y)\quad\text{and}\quad T^{\#}_{Q_0}f(x):=\sup_{\substack{x\in Q\in\mathcal D\\Q\subset  Q_0}}\left\langle\|T\widetilde{P}_{r_{Q}^m}^Nf\|_H\right\rangle_{q_0,Q}$$
are the $Q_0$-adapted version of $M_{p_0}$ and $T^{\#}$. By definition, $M_{p_0,Q_0}^{\mathcal{D}}f(x)\leq M_{p_0}f(x)$ and $T^{\#}_{Q_0}f(x)\leq T^{\#}f(x)$. Thus, the weak type $(p_0,p_0)$ of $M_{p_0}$ and Proposition \ref{Property} \textnormal{(iii)} give that
$$|E_{Q_0}|\lesssim C_{0}^{-p_0}\langle f\rangle_{p_0,Q_0^*}^{-p_0}\|f\|_{L^{p_0}(Q_0^*)}^{p_0}=C_0^{-p_0}|Q_0^*|\lesssim C_0^{-p_0}|Q_0|,$$
here the definition of $Q^*$ gives $|Q|\sim|B_Q|\sim|Q^*|$. Choosing a large $C_0$ we obtain $|E_{Q_0}|\leq|Q_{0}|/2$. Then we let $\{P_j\}_{j}$ be the maximal dyadic cubes in $E_{Q_0}$. By the definition of $M_{p_0,Q_0}^{\mathcal{D}}$ and $T^{\#}_{Q_0}$, every $x\in E_{Q_0}$ belongs a dyadic cube $Q_x\subset E_{Q_0}$. Hence, $E_{Q_0}=\bigcup_j P_j$, and the maximality implies that $\{P_j\}_{j}$ are pairwise disjoint, which indicates $\sum_{j}|P_j|=|E_{Q_0}|\leq |Q_0|/2$. 

\medskip
\noindent\textbf{Step 2.1: A recursive estimate.}

For every $Q\in\mathcal{D}$, define $T_{Q}f:=T(I-\widetilde{P}_{r_{Q}^m}^N)(f\chi_{Q^*})$. We show that there exist constant $C^{\prime}>0$, such that
\begin{equation}\label{Re}
    \int_{Q_0}\|T_{Q_0}f\|_{H}|g|\,d\mu\leq C^{\prime}K|Q_0|\langle f\rangle_{p_0,Q_0^*}\langle g\rangle_{q_{0}^{\prime},Q_0^*}+\sum_j\int_{P_j}\|T_Pf\|_{H}|g|\,d\mu.\tag{Re}
\end{equation}

By Proposition \ref{Property} \textnormal{(i)}, for a.e. $x\in Q_0\backslash E_{Q_0}$, 
$$\|T(f\chi_{Q_{0}^*})(\cdot,x)\|_{H}\leq T^{\#}_{Q_0}f(x)\lesssim K\langle f\rangle_{p_0,Q^*}.$$ 
Although Proposition \ref{Property} \textnormal{(i)} is stated for $T^{\#}$, its proof also applies to $T^{\#}_{Q_0}$, since all sufficiently small dyadic cubes containing $x$ are contained in $Q_0$. Thus, by H\"{o}lder's inequality for $g$, there exist a constant $C>0$, such that
\begin{equation}\label{5.20}
\int_{Q_0}\|Tf\|_{H}|g|\,d\mu\leq CK\langle f\rangle_{p_{0},Q_0^*}\langle g\rangle_{q_0^{\prime},Q_0^*}|Q_0|+\sum_j\int_{P_j}\|Tf\|_H|g|\,d\mu.
\end{equation}
Here we used $Q_0\subset Q_0^*$ and $|Q_0|\sim|Q_{0}^*|$, so $\langle g\rangle_{q_0^{\prime},Q_0}\lesssim\langle g\rangle_{q_0^{\prime},Q_0^*}$. Then, for every $P=P_j$, $Tf=T_{P}f+T\widetilde{P}_{r_{P}^m}^Nf+T(I-\widetilde{P}_{r_{P}^m}^N)(f\chi_{(P^*)^{c}})$. Thus, by H\"{o}lder's inequality,
\begin{equation}\label{5.21}
    \int_{P}\|Tf\|_{H}|g|\,d\mu\leq \int_{P}\|T_{P}f\|_{H}|g|\,d\mu+|P|\langle g\rangle_{q_{0}^{\prime},P}(\langle \|T\widetilde{P}_{r_{P}^m}^Nf\|\rangle_{q_0,P}+\langle \|T(I-\widetilde{P}_{r_{P}^m}^N)(f\chi_{(P^*)^c})\|_H\rangle_{q_0,P})
\end{equation}

The first term will be left for the recursion. We start with the third term. We write $A_k(B)=2^{k+1}B\,\backslash \,2^{k}B$ for $k\geq 1$ and $A_0(B)=2B$. By \eqref{Q*} we have $(P^*)^c\subset (2^l B_{P})^c=\bigcup_{k\geq l}A_{k}(B_P)$. It follows from \eqref{GGE5} that
\begin{align*}
\|\chi_{B_{P}}T\widetilde{Q}_{t}^N \chi_{A_{k}(B_P)}\|_{p_0\rightarrow q_{0}(H)}\lesssim D|B_{P}|^{-\delta_{q_0}}|2^{k+1}B_{P}|^{-\delta_{p_0}}(2^k r_{P}t^{-1/m})^{\nu\delta_{p_0}}(r_P t^{-1/m})^{\nu\delta_{q_0}}(2^kr_Pt^{-1/m})^{-\sigma}.
\end{align*}
Since $|2^{k+1}B_P|/|B_P|\lesssim 2^{\nu k}$, we have that
\begin{align}
    \langle\|T\widetilde{Q}_{t}^N(f\chi_{A_k(B_P)})\|_{H}\rangle_{q_0,B_P}&\lesssim D|B_P|^{-1/2}|2^{k+1}B_p|^{-\delta_{p_0}}2^{-k(\sigma-\nu\delta_{p_0})}(r_Pt^{-1/m})^{\nu(1/p_0-1/q_0)-\sigma}\|f\|_{L^{p_0}(A_k(B_P))}\notag\\
    &\lesssim D 2^{-k(\sigma-\nu/p_0)}(r_Pt^{-1/m})^{\nu(1/p_0-1/q_0)-\sigma}\langle f\rangle_{p_0,2^{k+1}B_P}.\label{5.22}
\end{align}
Note that $P\subset B_{P}$ and $|P|\sim|B_{P}|$ (see discussion before \eqref{Q*}). Since $T\Pi_0=0$, by \eqref{5.2} and Minkowski's inequality, we obtain that
\begin{align}
    \langle \|T(I-\widetilde{P}_{r_{P}^m}^N)(f\chi_{(P^*)^c})\|_{H}\rangle_{q_{0},P}&\lesssim\sum_{k\geq l} \langle \|T(I-\Pi_0-\widetilde{P}_{r_{P}^m}^N)(f\chi_{A_{k}(B_P)\cap (P^*)^c})\|_{H}\rangle_{q_{0},B_P}\notag\\
    &\lesssim\sum_{k\geq l}\int_0^{r_{P}^m}\langle\|T\widetilde{Q}_{t}^N(f\chi_{(P^*)^c}\chi_{A_k(B_P)})\|_{H}\rangle_{q_0,B_P}\frac{dt}{t}\notag\\
    &\lesssim D\sum_{k\geq l}2^{-k(\sigma-\nu/p_0)}\langle f\rangle_{p_0,2^{k+1}B_P}\int_{0}^{r_{P}^m}(r_Pt^{-1/m})^{\nu(1/p_0-1/q_0)-\sigma}\frac{dt}{t}\notag\\
    &\lesssim D\sum_{k\geq l}2^{-k(\sigma-\nu/p_0)}\inf_{y\in 2^{l+1}B_P} M_{p_0}f(y)\notag\\
    &\lesssim D\sum_{k\geq l}2^{-k(\sigma-\nu/p_0)}\inf_{y\in P^a} M_{p_0}f(y)\lesssim D\langle f\rangle_{p_0,Q_0^*}.\label{5.23}
\end{align}
Here we have assumed that $\sigma>\max\{\nu/p_0,\nu/q_0^{\prime}\}$. By the maximality of $P$, there exists $x_P\in P^a\,\backslash\, E_{Q_0}$. Since $P^a\subset Q_0$ (otherwise, we have $P=Q_0$, which contradicts $|P|\leq|Q_0|/2$), $\inf_{y\in P^a} M_{p_0}f(y)\leq M_{p_0,Q_0}^{\mathcal{D}}(x_P)\lesssim\langle f\rangle_{p_0,Q_0^*}$. Hence, by kolmogorov's inequality and the fact that $g$ is supported in $Q_0^*$,
\begin{align}
    \sum_j |P_j|\langle g\rangle_{q_0^\prime,P_j}\langle\|T(I-\widetilde{P}_{r_{P_j}^m}^N)(f\chi_{(P_j^*)^c})\|_H\rangle_{q_0,P_j}&\lesssim D\langle f\rangle_{p_0,Q_0^*}\sum_{j}|P_j|\inf_{y\in P_j} M_{q_0^{\prime}} g(y)\notag\\
    &\leq D\langle f\rangle_{p_0,Q_0^*}\int_{E_{Q_0}}M_{q_{0}^{\prime}}g(y)\,dy\notag\\
    &\lesssim D\langle f\rangle_{p_0,Q_0^*} |E_{Q_0}|^{1-1/q_0^{\prime}}\|g\|_{L^{q_0^{\prime}}}\notag\\
    &\lesssim D|Q_0|\langle f\rangle_{p_0,Q_0^*}\langle g\rangle_{q_{0}^{\prime},Q_0^*}.\label{5.24}
\end{align}

Then we estimate the second term in \eqref{5.21}. Since \textnormal{supp}$\,f\subset Q_0^*$, the maximality of $P$ gives $\langle\| T\widetilde{P}_{r_{P^a}^m}^N f\|_H\rangle_{q_0,P^a}\lesssim K\langle f\rangle_{p_0,Q_0^*}$. Moreover, by a similar argument as in \eqref{5.22}-\eqref{5.23}, we obtain that
\begin{align}
    \langle \|T(\widetilde{P}_{r_{P^a}^m}^N-\widetilde{P}_{r_{P}^m}^N)f\|_{H}\rangle_{q_0,P}&\lesssim\sum_{k\geq 0}\int_{r_{P}^m}^{r_{P^a}^m}\langle\|T\widetilde{Q}_{t}^N(f\chi_{A_k(B_P)})\|_H\rangle_{q_0,B_P}\frac{dt}{t}\notag\\
    &\lesssim D\inf_{y\in P^a} M_{p_0}f(y)\lesssim D\langle f\rangle_{p_0,Q_0^*}.\notag
\end{align}
Thus, by the same argument as in \eqref{5.24},
\begin{equation}
\sum_j|P_j|\langle g\rangle_{q_0^{\prime},P_j}\langle \|T\widetilde{P}_{r^m_{P_j}}^Nf\|\rangle_{q_0,P_j}\lesssim K|Q_0|\langle f\rangle_{p_0,Q_0^*}\langle g\rangle_{q_{0}^{\prime},Q_0^*}.\notag
\end{equation}
Combining this with \eqref{5.20}, \eqref{5.21} and \eqref{5.24}, we obtain that
\begin{equation}\label{5.25}
    \int_{Q_0}\|Tf\|_{H}|g|\,d\mu\leq C^{\prime}K|Q_0|\langle f\rangle_{p_0,Q_0^*}\langle g\rangle_{q_{0}^{\prime},Q_0^*}+\sum_j\int_{P_j}\|T_Pf\|_{H}|g|\,d\mu.
\end{equation}

Finally, we estimate $\|Tf-T_{Q_0}f\|_H=\|T\widetilde{P}_{r_{Q_0}^m}^Nf\|_{H}$. By the definition of $T^{\#}$, for every $x\in Q_0$, $T^{\#}f(x)\geq\langle\|T\widetilde{P}_{r_{Q_0}^m}^Nf\|_H\rangle_{q_0,Q_0}$. Thus, the weak type $(p_0,p_0)$ estimate in Proposition \ref{Property} \textnormal{(iii)} gives
$$|Q_0|\leq|\{x\in X:T^{\#}f(x)\geq\langle\|T\widetilde{P}_{r_{Q_0}^m}^Nf\|_H\rangle_{q_0,Q_0}\}|\lesssim\frac{K^{p_0}}{\langle\|T\widetilde{P}_{r_{Q_0}^m}^Nf\|_H\rangle_{q_0,Q_0}^{p_0}}\|f\|_{L^{p_0}}^{p_0}.$$
Hence, $\langle\|T\widetilde{P}_{r_{Q_0}^m}^Nf\|_H\rangle_{q_0,Q}\lesssim K|Q_0|^{-1/p_0}\|f\|_{L^{p_0}}\sim K \langle f\rangle_{p_0,Q_0^*}$, and
\begin{equation}\label{5.26}
\int_{Q_0}\|Tf-T_{Q_0}f\|_{H}|g|\,d\mu\lesssim K|Q_0|\langle f\rangle_{p_0,Q_0^*}\langle g\rangle_{q_0^{\prime},Q_0}\lesssim K|Q_0|\langle f\rangle_{p_0,Q_0^*}\langle g\rangle_{q_0^\prime,Q_{0}^*}.
\end{equation}
Combining this with \eqref{5.25}, we finish the proof of \eqref{Re}.

\medskip
\noindent\textbf{Step 2.2: Construction of the sparse family.}

We denote $\mathcal{S}_0:=\{Q_0\}$ and $\mathcal{S}_1:=\{P_j\}_j\subset Q_0$. We apply the argument above to each cube in $\mathcal{S}_1$ and then repeat this process. Let $\mathcal{S}_n=\{P_j^{(n)}\}_j$ be the family of cubes obtained at step $n$. By \eqref{Re} we have that

\begin{equation}\label{5.27}
    \int_{Q_0}\|T_{Q_0}f\|_{H}|g|\,d\mu\leq C^{\prime}K\sum_{\substack{P\in \mathcal{S}_k\\0\leq k\leq N-1}}|P|\langle f\rangle_{p_0,P^*}\langle g\rangle_{q_{0}^{\prime},P^*}+\sum_{P_j\in \mathcal{S}_N}\int_{P_j}\|T_{P_j}f\|_{H}|g|\,d\mu.
\end{equation}
Define $\mathcal{S}:=\bigcup_{i\in\mathbb{N}}\mathcal{S}_{i}$. For every $P\in \mathcal{S}_N$, we have that $\sum_{\substack{Q\in\mathcal{S}_{N+1}},Q\subset P}|Q|\leq({1}/{2})|P|$. Let $G_P:=P\,\setminus\bigcup_{\substack{Q\in\mathcal{S}_{N+1}},Q\subset P}Q$, then $\{G_P\}_{P\in\mathcal{S}}$ are pairwise disjoint and $|G_{P}|\geq (1/2)|P|$.
Thus, $\mathcal{S}$ is a $1/2$ sparse family, and $\sum_{P\in\mathcal{S}_N}|P|\leq |Q_0|/2^N$. Moreover, it follows from Proposition \ref{Property} \textnormal{(ii)} that
$$\int_{P}\|T_Pf\|_{H}|g|\,d\mu\leq\|T_Pf\|_{L^{2}(X;H)}\|g\|_{L^2(P)}\lesssim K\|f\|_{\infty}\|g\|_{\infty}|P^*|^{1/2}|P|^{1/2}\sim K|P|.$$
Hence, letting $N\rightarrow\infty$ in \eqref{5.27}, and combining the result with \eqref{5.26}, we obtain that
\begin{equation}
    \int_{Q_0}\|Tf\|_H|g|\,d\mu\leq C_1 K \sum_{P\in\mathcal{S}}|P|\langle f\rangle_{p_0,P^*}\langle g\rangle_{q_0^{\prime},P^*}.\notag
    \end{equation}
For every $P\in\mathcal{D}$, by the definition of $P^*$ in \eqref{Q*}, there exists a constant $c>0$ that $P^*\subset cB_P$. By Theorem \ref{Dyadic}, there exist $\overline{P}\in\bigcup_{b=1}^{K_{D}}\mathcal{D}^b$ that $P^*\subset \overline{P}$ and $|\overline{P}|\lesssim|P^*|$. Thus $|\overline{P}|\leq C|P|$ for some $C>1$. For each $P\in \mathcal{S}$, we fix a $\overline{P}$ for $P$. Define $\bar{\mathcal{S}}_b:=\{\overline{P}:P\in\mathcal{S},\overline{P}\in\mathcal{D}^b\}$, then for each $\overline{P}\in\bar{\mathcal{S}}_b$, we have $G_P\subset P\subset \overline{P}$ and $|G_P|\geq (1/2)|P|\geq(1/2C)|\overline{P}|$. Thus, $\bar{\mathcal{S}_b}$ is $1/(2C)$-sparse for every $b=1,...,K_D$. Moreover, we have that
\begin{align}
    \int_{Q_0}\|Tf\|_H|g|\,d\mu&\leq C_1K\sum_{b=1}^{K_D} \sum_{\overline{P}\in\bar{\mathcal{S}}_b}|\overline{P}|\langle f\rangle_{p_0,\overline{P}}\langle g\rangle_{q_0^{\prime},\overline{P}}.\label{5.28}
\end{align}
Finally, recall that here $K=D+AR^{\nu(1/p_0-1/q_0+\delta_s)}$ with $\delta_s=1/s-1/2$. For any $\varepsilon>0$, we choose $s<2$ sufficiently close to 2 that $\nu\delta_{s}<\varepsilon$. This completes the proof of Theorem \ref{sdc}.
\end{proof}

\subsection{Sparse bound for $q$-variation when $\varphi\in \Lambda^{\beta}$}
\noindent\textbf{Choice of representatives.}
 We first choose continuous representatives of $\varphi(tL)f$ for $f\in L^2(X)$. As noted in Section 4.1, even for every $t>0$ $\varphi(tL)f$ is well defined in $L^2(X)$, the $q$-variation may not be well defined pointwise. For $f\in L^2(X)$, it follows from \eqref{5.1} that there exist unique $f_0\in \ker L$ and $h\in \overline{\textnormal{Ran}\, L}$ such that $f=f_0+h$. Since $\varphi(0)=0$, we have that
 $$\varphi(tL)f_0=\int_{0}^{\infty}\varphi(t\lambda)\,dE_{L}(\{0\})f=0.$$ Thus, we choose the representative $\varphi(tL)f_0(x)\equiv0$ for all $t>0$ and $x\in X$.

 For $h\in \overline{\textnormal{Ran}\, L}$, let $g:=\varphi\circ\textnormal{exp}\in W^{\beta,2}$. Since $\beta>1/2$, we choose $g$ as a $C_0$-representative. Let $F(t):=g(t+\log L)h\in L^2(X)$. Then
 \begin{equation}\label{5.29}  \|F\|_{L^2(\mathbb{R};L^2(X))}^2=\int_{\mathbb{R}}\int_{(0,\infty)}|g(t+\log\lambda)|^2\,d\langle E_{L}(\lambda)h,h\rangle \,dt=\|g\|_{L^2}^2 \|h\|_{L^2(X)}^2.
 \end{equation}
Let $\widetilde{\mathcal{F}}$ be the Fourier transform on $L^2(\mathbb{R};L^2(X))$ and $\mathcal{F}$ be the Fourier transform on $L^2(\mathbb{R})$., we show that 
\begin{equation}\label{5.30}
    \widetilde{\mathcal{F}}F(\xi)=\mathcal{F}g(\xi)L^{i\xi}h\quad\text{in }L^2(\mathbb{R};L^2(X)).
\end{equation}
\begin{proof}[Proof of \eqref{5.30}]We argue by approximation. First we choose a sequence $g_n\in\mathcal{S}(\mathbb{R})\rightarrow g$ in $L^2(\mathbb{R})$. Define $h_n:=E_{L}([e^{-n},e^n])h$ and $F_n(t):=g_n(t+\log L)h_n$. With a similar calculation as in \eqref{5.29}, we have $$\|F_n(t)\|_{L^2(X)}\leq \|h\|_{L^2(X)}\sup_{|a|\leq n}|g_n(t+a)|\lesssim\|h\|_{L^2(X)}\sup_{|a|\leq n}(1+|t+a|)^{-2}\lesssim\|h\|_{L^2(X)}(1+|t|)^{-2},$$
the last inequality is due to $1+|t+a|\gtrsim n+1+|t+a|\geq 1+|t|$. Thus, $F_n\in L^1(\mathbb{R};L^2(X))\cap L^{2}(\mathbb{R};L^2(X))$, and
\begin{align*}
    \widetilde{\mathcal{F}}F_n(\xi)&=\int_{\mathbb{R}}e^{-it\xi}\int_{[e^{-n},e^n]}g_n(t+\log \lambda)\,dE_{L}(\lambda)h\,dt\\
    &=\int_{[e^{-n},e^n]}\int_{\mathbb{R}}e^{-it\xi}g_n(t+\log\lambda)\,dt\,dE_{L}(\lambda)h\\
    &=\mathcal{F}g_n(\xi)\int_{[e^{-n},e^n]}\lambda^{i\xi}\,dE_{L}(\lambda)h=\mathcal{F}g_n(\xi) L^{i\xi}h_n.
\end{align*}
Here the second identity follows from Fubini's theorem, since the integral is absolutely convergent. Then, it follows from \eqref{5.29} and $\|h_n\|_{L^2(X)}\leq\|h\|_{L^{2}(X)}$ that
$$\|\widetilde{\mathcal{F}}F_n-\widetilde{\mathcal{F}}F\|_{L^{2}(\mathbb{R};L^2(X))}\sim\|F_n-F\|_{L^2(\mathbb{R};L^{2}(X))}\leq\|g_n-g\|_{L^2}\|h_n\|_{L^2(X)}+\|g\|_{L^2}\|h_n-h\|_{L^2(X)}\rightarrow0.$$
Similarly, we have that
$$\|\mathcal{F}g_n(\cdot)L^{i(\cdot)}h_n-\mathcal{F}g(\cdot)L^{i(\cdot)}h\|_{L^{2}(\mathbb{R};L^2(X))}\lesssim\|g_n-g\|_{L^2}\|h_n\|_{L^2(X)}+\|g\|_{L^2}\|h_n-h\|_{L^2(X)}\rightarrow0.$$
By letting $n\infty$, we prove \eqref{5.30}.
\end{proof}
We fix a measurable representative of $F(t)$ as $F(t,x)$. By Fubini's theorem, we have $\int_{\mathbb{R}}|F(t,x)|^2\,dt<\infty$ for a.e. $x\in X$. Let $\widetilde{\mathcal{F}}F(\xi,x)$ be a measurable representative of $\widetilde{\mathcal{F}}F(\xi)$. Then for a.e. $x\in X$, we have $\widetilde{\mathcal{F}}F(\xi,x)=\mathcal{F}[F(\cdot,x)](\xi)$ in $L^2(\mathbb{R})$. We can first verify this identity for $F(t)=\sum_{k=1}^N \phi_k(t)h_k$ with $\phi_k\in\mathcal{S}(\mathbb{R})$ and $h_k\in L^2(X)$. The general case then follows by density. We refer the reader to \cite[Chapter 5]{HNVW16} for more background on the Fourier transform and Sobolev spaces of vector valued functions. Denote $\langle\xi\rangle:=(1+|\xi|^2)^{1/2}$, it follows from \eqref{5.30} that
\begin{equation}\label{5.31}
    \|F\|_{W^{\beta,2}(\mathbb{R};L^2(X))}^2=\int_{\mathbb{R}}\langle \xi\rangle^{2\beta}\|\widetilde{\mathcal{F}}F(\xi)\|_{L^2(X)}^2\,d\xi=\|h\|_{L^2(X)}^2\int_{\mathbb{R}}\langle \xi\rangle^{2\beta}|\mathcal{F}g(\xi)|^2\,d\xi=\|g\|_{W^{\beta,2}}^2\|h\|_{L^2(X)}^2<\infty.
\end{equation}
Thus, by Fubini's theorem we have that
\begin{equation}\label{5.32}
    \|F\|_{W^{\beta,2}(\mathbb{R};L^2(X))}^2=\int_{\mathbb{R}}\langle \xi\rangle^{2\beta}\|\widetilde{\mathcal{F}}F(\xi)\|_{L^2(X)}^2\,d\xi=\int_{X}\int_{\mathbb{R}}\langle\xi\rangle^{2\beta}|\mathcal{F}[F(\cdot,x)](\xi)|^2\, d\xi\,dx=\|F\|^2_{L^2(X;W^{\beta,2}(\mathbb{R}))},
\end{equation}
which implies that for a.e. $x\in X$, $F(\cdot,x)\in W^{\beta,2}(\mathbb{R})$. Since $\beta>1/2$, the Sobolev embedding gives a continuous representative $\widetilde{F}(\cdot,x)\in C_0(
\mathbb{R})$. For $t>0$, we choose $(\varphi(tL)h)(x):=\widetilde{F}(\log t,x)$. Finally, it remains to check that $\widetilde{F}(t,x)$ is a representative of $F(t)$ for every $t\in\mathbb{R}$. It follows from the choice of $\widetilde{F}$ that for a.e. $t\in\mathbb{R}$, $\widetilde{F}(t,\cdot)=F(t)$ in $L^2(X)$. We show that both $t\mapsto F(t)$ and $t\mapsto\widetilde{F}(t,\cdot)$ are continuous in $L^2(X)$, then they are equal for any $t\in\mathbb{R}$. Thus, $\widetilde{F}(t,x)$ is a representative of $F(t)$. For $F(t)$, if follows from the spectral theorem that
$$\|F(t+s)-F(t)\|_{L^2(X)}^2=\int_{(0,\infty)}|g(t+s+\log \lambda)-g(t+\lambda)|^2\,d\langle E_{L}(\lambda)h,h\rangle.$$
Recall that we take $g$ as a $C_0$-representative. As $s\rightarrow0$, we have $|g(t+s+\log\lambda)-g(t+\log\lambda)|^2\lesssim\|g\|_{\infty}^2$. Thus, by the dominated convergence theorem we have $F(t+s)\rightarrow F(t)$ in $L^2(X)$. For $\widetilde{F}$, since $\beta>1/2$, the Sobolev embedding gives that $\sup_{t\in\mathbb{R}}|\widetilde{F}(t,x)|\lesssim\|F(\cdot,x)\|_{W^{\beta,2}(\mathbb{R})}^2$. The right hand side is integrable over $X$ by \eqref{5.31} and $\eqref{5.32}$. Since $\widetilde{F}(\cdot,x)$ is continuous for a.e. $x\in X$, the dominated convergence theorem shows that $t\mapsto\widetilde{F}(t,\cdot)$ is continuous in $L^2(X)$. This completes the choice of representatives.

\medskip
We now come to the proof of Theorem \ref{thm2}. First, we state the following lemmas.
\begin{lem}\label{Fourierinverse}
    Let $L$ be a non-negative self-adjoint operator on a Hilbert space $\mathcal{H}$. Let $h\in W^{\alpha,2}(\mathbb{R})$ for some $\alpha>1/2$, and take its $C_0$-continuous representative. Then, for every $f\in \mathcal{H}$, we have
    $$h(L)f=\frac{1}{2\pi}\int_{\mathbb{R}}\mathcal{F}h(s)e^{isL}f\,ds,$$
    where the integral converges in $\mathcal{H}$ and $\mathcal{F}$ denotes the Fourier transform in $L^2(\mathbb{R})$. Especially, when $\mathcal{H}=L^2(X)$, there exist a representative $K(s,x)$ of $\mathcal{F}h(s)e^{isL}f$, such that 
    $$h(L)f(x)=\frac{1}{2\pi}\int_{\mathbb{R}}K(s,x)\,ds\quad\text{for a.e. }x\in X.$$ In the rest of the paper, we write $\mathcal{F}h(s)e^{isL}f(x)$ for $K(s,x)$.
  
\end{lem}
\begin{proof}[Proof of Lemma \ref{Fourierinverse}]
    By the H\"{o}lder's inequality, 
    $$\|\mathcal{F}h\|_{L^1}\leq \|\langle\cdot\rangle^{\alpha}\mathcal{F}h(\cdot)\|_{L^2}\|\langle\cdot\rangle^{-\alpha}\|_{L^2}\lesssim\|h\|_{W^{\alpha,2}}.$$
    Thus, $\widetilde{h}(\lambda):=(1/2\pi)\int_{\mathbb{R}}\mathcal{F}h(s) e^{is\lambda}\,ds$ is continuous in $\lambda$. Moreover, since $h\in W^{\alpha,2}(\mathbb{R})\subset L^{2}(
    \mathbb{R})$, we have $\mathcal{F}h\in L^{2}(\mathbb{R}
    )$. Since $\mathcal{F}h\in L^{1}(\mathbb{R})\cap L^2(\mathbb{R})$, Fourier inversion gives $h(\lambda)=\widetilde{h}(\lambda)$ for a.e. $\lambda\in\mathbb{R}$. Because both functions are continuous, they are equal for all $\lambda\in\mathbb{R}$. Then it follows from the spectral theorem that
    $$h(L)f=\int_{(0,\infty)}h(\lambda)\,dE_{L}(\lambda)f=\frac{1}{2\pi}\int_{(0,\infty)}\int_{\mathbb{R}}\mathcal{F}h(s)e^{is\lambda}\,ds\,dE_{L}(\lambda)f=\frac{1}{2\pi}\int_{\mathbb{R}}\mathcal{F}h(s)e^{isL}f\,ds.$$
    Here in the last identity we apply Fubini's theorem to change the order of integration. By the fact that $\mathcal{F}h\in L^1$, we can easily see that the integral converges absolutely.
    
    When $\mathcal{H}=L^2(X)$, we have $$\int_{\mathbb{R}}{\|\mathcal{F}h(s)e^{isL}f\|_{L^2(X)}}\,ds=\|f\|_{L^2(X)}\|\mathcal{Fh}\|_{L^{1}(\mathbb{R})}<\infty.$$
    Thus, it follows from \cite[Proposition 1.2.25]{HNVW16} that there exists a strongly measurable function $K(s,x)$, such that $K(s,\cdot)=\mathcal{F}h(s)e^{isL}f$ in $L^2(X)$ for a.e. $s$, and $$\left(\int_{\mathbb{R}}\mathcal{F}h(s)e^{isL}f
\,ds\right)(x)=\int_{\mathbb{R}}K(s,x)ds\quad\quad \text{for a.e. }x\in X.$$  Combining this with $h(L)f=(1/2\pi)\int_{\mathbb{R}}\mathcal{F}h(s)e^{isL}f\,ds$ in $L^2(X)$, we obtain the pointwise identity.
\end{proof}
\begin{lem}\label{Sobolevestimate}
    Let $h\in W^{\alpha,2}(\mathbb{R})$ for some $\alpha\geq 0$. For $l>0$, let $H_l(\lambda)=\lambda^N e^{-\lambda/2}h(\log(\lambda/l))$. Then for every integer $N\geq \alpha$, $H_{l}\in W^{\alpha,2}(\mathbb{R})$ and
    $$\int_{0}^{\infty}\|H_{l}\|_{W^{\alpha,2}(\mathbb{R})}^2\,\frac{dl}{l}\lesssim \|h\|_{W^{\alpha,2}(\mathbb{R})}^2,$$
    where the implicit constant only depends on $N$ and $\alpha$.
\end{lem}
\begin{proof}[Proof of Lemma \ref{Sobolevestimate}]
We first consider the case $\alpha=k\in\mathbb{N}$. Write $D=d/dx$. For $r\leq k$, we have
$$\frac{d^r}{d\lambda^r}h(\log(\lambda/l))=\lambda^{-r}(p_r(D)h)(\log(\lambda/b)),$$
where $p_r$ is a polynomial of degree $r$. Thus, for $j\leq k$, there exist constants $C_{a,b,r}>0$ such that
$$\frac{d^j}{d\lambda^j}H_l(\lambda)=\sum_{a+b+r=j}C_{a,b,r} \lambda^{N-a-r}e^{-\lambda/2}(p_r(D)h)(\log(\lambda/b)).$$
By the change of variable $x=\log l$, we have 
$$\int_{0}^{\infty}|(p_{r}(D)h)(\log\lambda-\log l)|^2\,\frac{dl}{l}=\|p_r(D)h\|_{L^2}^2\lesssim\|h\|_{W^{k,2}(\mathbb{R}
)}^2
.$$
Thus, it follows from $N\geq k\geq j\geq a+r$ that
\begin{equation}\label{5.33}
\int_{0}^{\infty}\left\|\frac{d^j}{d\lambda^j}H_l(\lambda)\right\|_{L^2}^2\,\frac{dl}{l}\lesssim\sum_{r=0}^j\|p_r(D)h\|_{L2}^2\sum_{a+b=j-r}\int_{0}^{\infty}\lambda^{2(N-a-r)}e^{-\lambda}\,d\lambda\lesssim\|h\|_{W^{k,2}(\mathbb{R})}^2.
\end{equation}
Summing over $0\leq j\leq k$, we obtain that
$$\int_{0}^{\infty}\|H_l\|_{W^{\alpha,2}(\mathbb{R})}^2\frac{dl}{l}\lesssim\|h\|_{W^{\alpha,2}(\mathbb{R})}.$$
Moreover, for fixed $l>0$ and $j\leq k$,
\begin{align}
\left\|\frac{d^j}{d\lambda^j}H_{l}\right\|_{L^2}^2&\lesssim\sum_{a+b+r=j}\int_{0}^{\infty}\lambda^{2(N-a-r)}e^{-\lambda}|(p_{r}(D)h)(\log\lambda-\log l)|^2 \,d\lambda\notag\\
&\lesssim\sum_{r=0}^j\int_{0}^{\infty}|(p_r(D)h)(\log\lambda-\log l)|^2\frac{d\lambda}{\lambda}\notag\\
&\lesssim\sum_{r=0}^j\|p_r(D)h\|_{L^2}^2\lesssim\|h\|_{W^{k,2}(\mathbb{R})}^2.\label{5.34}
\end{align}
Thus, $H_l\in W^{k,2}(\mathbb{R})$ and $\|H_l\|_{W^{k,2}(\mathbb{R})}\lesssim\|h\|_{W^{k,2}(\mathbb{R})}$.

For non-integer $\alpha$, denote $k=[\alpha]$ and choose $0\leq\theta\leq1$ such that $\alpha=(1-\theta) k+\theta(k+1)$. We apply a complex interpolation to the operator $h\rightarrow H_l$ between $W^{k,2}(\mathbb{R})\rightarrow W^{k,2}(\mathbb{R}
)$ and $W^{k+1,2}(\mathbb{R})\rightarrow W^{k+1,2}(\mathbb{R})$. By \cite[Section 6.4]{BL76} $[W^{k,2}(\mathbb{R}
),W^{k+1,2}(\mathbb{R})]_{\theta}=W^{\alpha,2}(\mathbb{R})$. Hence, we obtain $H_{l}\in W^{\alpha,2}(
\mathbb{R}
)$ and $\|H_l\|_{W^{\alpha,2}(\mathbb{R})}\lesssim\|h\|_{W^{\alpha,2}(\mathbb{R})}$. Moreover, by \cite[Theorem 5.1.2]{BL76} we also have 
$$[L^{2}(dl/l;W^{k,2}(
\mathbb{R}
)),L^{2}(dl/l;W^{k+1,2}(\mathbb{R}))]_{\theta}=L^2(dl/l;W^{\alpha,2}(\mathbb{R})).$$
Thus, by applying complex interpolation to $h\mapsto (H_l)_{l>0}$, we obtain the desired estimate.
\end{proof}
\begin{proof}[Proof of Theorem \ref{thm2}]
    Let $\chi\in C_{c}^{\infty}(\mathbb{R})$ be that $\chi\equiv1$ on $[-1,1]$, $\textnormal{supp}\,\chi\subset[-2,2]$, and $\chi$ is non-decreasing on $(-\infty,0]$ and non-increasing on $[0,\infty)$. Define $\psi_j:=\chi(2^{-j}(\cdot))-\chi(2^{-j+1}(\cdot))$ for $j\geq 1$ and $\psi_0=\chi$. Then $\sum_{j\geq0}\psi_{j}\equiv1$, and $\textnormal{supp}\,\psi_j\subset\{x\in \mathbb{R}: 2^{j-1}\leq|x|\leq 2^{j+1}\}$ for $j\geq 1$. For $\varphi\in \Lambda^{\beta}$ and $j\geq 0$, let $g:=\varphi\circ\textnormal{exp}\in W^{\beta,2}$ and $g_j=\mathcal{F}^{-1}[(\mathcal{F}g)\psi_j]$. Then $g_j\in W^{\gamma,2}$ for any $\gamma\geq0$ with $\|g_j\|_{W^{\gamma,2}}\sim2^{\gamma-\beta}\|g_j\|_{W^{\beta,2}}$, and
    \begin{equation}\label{5.35}
    \|g\|_{W^{\beta,2}(\mathbb{R})}\sim\left(\sum_{j\geq0}\|g_j\|_{W^{\beta,2}(\mathbb{R})}^2\right)^{1/2}\sim\left(\sum_{j\geq0}2^{2j(\beta-\gamma)}\|g_j\|^2_{W^{\gamma,2}(\mathbb{R})}\right)^{1/2}.
    \end{equation}
    Let $\varphi_j(\lambda)=g_j(\log\lambda)$ for $\lambda>0$. The Sobolev embedding gives that $\lim_{\lambda\rightarrow0}\varphi_j(\lambda)=0$, thus we define $\varphi_j(0)=0$. Fix $\eta>1/2$ ($\eta$ will be chosen at the end of the proof), we will apply Theorem \ref{sdc} with $H=\Lambda^{\eta}$, $T:f\mapsto\left(\varphi_j(sL)f\right)_{s>0}$ and $R=2^j$.

    \medskip
    \noindent\textbf{$\boldsymbol{T\Pi_0=0}$.} For $f\in L^2(X)$. By the spectral theorem, we have that for every $s>0$,
    \begin{equation}\label{5.36}
    \varphi_j(sL)\Pi_0f=\int_{0}^{\infty}\varphi_j(s\lambda)\,dE_{L}(\lambda)\,E_{L}(\{0\})f=\varphi_j(0)\Pi_0f=0.
    \end{equation}
    Thus, $\left(\varphi(sL)\right)_{s>0}\Pi_0=0$.

    \medskip
    \noindent\textbf{$\boldsymbol{TP_{t}^N=((P_t^N+\Pi_0)\otimes I_H)T}$.} By the spectral theorem, we have that for every $f\in L^2(X)$,
    \begin{equation}\label{5.37}
        \left(\varphi_j(sL)e^{-tL}f\right)_{s>0}=\left(e^{-tL}\varphi_j(sL)f\right)_{s>0}=(e^{-tL}\otimes I_{\Lambda^{\beta}})\left(\varphi_j(sL)f\right)_{s>0}.
    \end{equation}
    Thus, the desired result follows from Remark \textnormal{(i)} after Theorem \ref{sdc}.
    
    \medskip
    \noindent\textbf{$\boldsymbol{\|T\|_{L^2(X)\rightarrow L^{2}(X;H)}<\infty}$.} For $f\in L^2(X)$, denote $f_0=\Pi_0f\in\ker\,L$ and $h=f-f_0\in\overline{Ran\,L}=\ker L^{\perp}$. Then $\|f\|^2=\|f_0\|^2+\|h\|^2$. By \eqref{5.31}, \eqref{5.32} and \eqref{5.36}, we have that
    \begin{align}
        \|\varphi_j(sL)f\|_{L^{2}(X;\Lambda^{\eta})}&=\|g_j(s+\log L)\|_{L^2(X;W^{\eta,2}(\mathbb{R}))}=\|g_j(s+\log L)\|_{W^{\eta,2}(\mathbb{R};L^2(X)
        )}\notag\\
        &=\|g_j\|_{W^{\eta,2}(\mathbb{R})}\|h\|_{L^2(X)}\leq\|g_j\|_{W^{\eta,2}(\mathbb{R})}\|f\|_{L^2(X)}\label{5.38}.
    \end{align}

    \medskip
    \noindent\textbf{Proof of (\ref{OD}).} Fix $N\in\mathbb{N_+}$. Denote the Bessel potential $J_{\eta}g:=\mathcal{F}^{-1}(\mathcal{F}g(\cdot)\langle\cdot\rangle^{\eta})$. Since $(tL)^Ne^{-tL}f\in\textnormal{Ran}\, L$ for every $f\in L^2(X)$, it follows from \eqref{5.30} that
    \begin{align*}
    \langle\xi\rangle^\eta \mathcal{\widetilde{F}}\left[g_j(\cdot+\log L) Q_t^Nf\right](\xi)&=\langle\xi\rangle^\eta\mathcal{F}g_j(\xi)L^{i\xi}Q_t^Nf\\
    &=\mathcal{F}(J_\eta g_j)(\xi)L^{i\xi}Q_t^Nf\\
    &=\mathcal{\widetilde{F}}\left[(J_\eta g_j)(\cdot+\log L) Q_t^Nf\right](\xi)\quad\text{in } L^2(\mathbb{R};L^2(X)).
    \end{align*}
    Thus, the identity also holds in $L^2(\mathbb{R})$ for a.e. $x\in X$. Then by applying the identity $\widetilde{\mathcal{F}}F(\cdot,x)=\mathcal{F}[F(\cdot,x)]$ in $L^2$ for a.e. $x\in X$ (see argument before \eqref{5.31}), we obtain that
    $$\langle\xi\rangle^\eta \mathcal{F}\left[g_j(\cdot+\log L) Q_t^Nf(x)\right](\xi)=\mathcal{F}\left[(J_\eta g_j)(\cdot+\log L) Q_t^Nf(x)\right](\xi)\quad\text{in }L^2(\mathbb{R})\quad\text{for a.e. }x\in X.$$
    Let $\widetilde{H}_{r,j}(\lambda):=J_\eta g_j (\log \lambda-\log (t/r))\lambda^N e^{-\lambda/2}$ for $r>0$. By Lemma \ref{Sobolevestimate} we have $\widetilde{H}_{r,j}\in W^{\gamma,2}(\mathbb{R})$ for every $\gamma\geq 0$, and 
    \begin{equation}\label{5.39}
        \int_{0}^{\infty}\|\widetilde{H}_{r,j}\|_{W^{\gamma,2}(\mathbb{R})}^2\frac{dr}{r}=\int_{0}^{\infty}\|H_{(t/r)}\|_{W^{\gamma,2}(\mathbb{R})}^2\frac{dr}{r}\lesssim\|J_\eta g_j\|_{W^{\gamma,2}(\mathbb{R})}^2\sim2^{2\gamma j}\|g_j\|_{W^{\eta,2}(\mathbb{R})}^2.
    \end{equation}
    Consequently, for a.e. $x\in X$, it holds that
    \begin{align}
        \|\varphi_j(sL)Q_t^Nf(x)\|_{\Lambda^{\eta}}&\sim\|g_{j}(s+\log L)(tL)^Ne^{-tL}f(x)\|_{W^{\eta,2}(\mathbb{R})}\notag\\
        &=\left\|\mathcal{F}\left[(J_\eta g_j)(\cdot+\log L) (tL)^Ne^{-tL}f(x)\right]\right\|_{L^2(\mathbb{R})}\notag\\
        &\sim\left(\int_{\mathbb{R}}|J_\eta g_j(\xi+\log L)(tL)^Ne^{-tL}f(x)|^2\,d\xi\right)^{1/2}\notag\\
        &=\left(\int_{0}^{\infty}|J_\eta g_j (\log r+\log L)(tL)^Ne^{-tL}f(x)|^2\frac{dr}{r}\right)^{1/2}\notag\\
        &=\left(\int_{0}^{\infty}|\widetilde{H}_{r,j}(tL)e^{-tL/2}f(x)|^2\frac{dr}{r}\right)^{1/2}\label{5.40}
    \end{align}
By Lemma \ref{ext} \textnormal{(ii)}, we have that for every $M>0$,
\begin{align*}
\|\chi_{F}v_{t^{1/m}}^{\delta_{q_0}}e^{-t((1/2)-iu)L}v_{t^{1/m}}^{\delta_{p_0}}\chi_E\|_{p_0\rightarrow q_0}&\lesssim\textnormal{exp}\left(-c\left(\frac{d(E,F)}{t^{1/m}(1+|u|)}\right)^{m/(m-1)}\right)\\
&\lesssim\left(1+\frac{d(E,F)}{t^{1/m}}\right)^{-M}(1+|u|)^M.
\end{align*}
Here the implicit constant depend on $M$. Thus, by applying Minkowski's inequality and Lemma \ref{Fourierinverse}, we obtain that
\begin{align}
    \|\chi_{F}v_{t^{1/m}}^{\delta_{q_0}}\widetilde{H}_{r,j}(tL)e^{-tL/2}v_{t^{1/m}}^{\delta_{p_0}}(f\chi_{E})\|_{L^{q_0}}&\sim
    \left\|\int_{\mathbb{R}}\chi_{F}v_{t^{1/m}}^{\delta_{q_0}}(\mathcal{F}\widetilde{H}_{r,j})(u)e^{-t((1/2)-iu)L}\,v_{t^{1/m}}^{\delta_{p_0}}(f\chi_{E})\,du\right\|_{L^{q_0}}\notag\\
    &\leq\int_{\mathbb{R}}|(\mathcal{F}\widetilde{H}_{r,j})(u)|\|\chi_{F}v_{t^{1/m}}^{\delta_{q_0}}e^{-t((1/2)-iu)L}v_{t^{1/m}}^{\delta_{p_0}}(f\chi_E)\|_{L^{q_0}}\,du\notag\\
    &\lesssim\left(1+\frac{d(E,F)}{t^{1/m}}\right)^{-M}\|f\chi_{E}\|_{L^{p_0}}\int_{\mathbb{R}}|(\mathcal{F}\widetilde{H}_{r,j})(u)|(1+|u|)^M\,du.\label{5.41}
\end{align}
Fix $\kappa>1$, by H\'{o}lder's inequality,
\begin{align}
    \int_{\mathbb{R}}|(\mathcal{F}\widetilde{H}_{r,j})(u)|(1+|u|)^M\,du&\lesssim\left(\int_{\mathbb{R}}|(\mathcal{F}\widetilde{H}_{r,j})(u)|^2\langle u\rangle^{2M+\kappa}\,du\right)^{1/2}\left(\int_{\mathbb{R}}\langle u\rangle^{-\kappa}\,du\right)^{1/2}\notag\\
    &\lesssim\|\widetilde{H}_{r,j}\|_{W^{M+(\kappa/2),2}(\mathbb{R})}.\notag
\end{align}
Hence, it follows from \eqref{5.39}, \eqref{5.40} and Minkowski's inequality that
\begin{align}
   \|\chi_Fv_{t^{1/m}}^{\delta_{q_0}}\varphi_j(sL)Q_{t}^Nv_{t^{1/m}}^{\delta_{q_0}}(f\chi_{E})\|_{L^{q_0}(X;\Lambda^{\eta})}&\lesssim \|\chi_{F}v_{t^{1/m}}^{\delta_{q_0}}\widetilde{H}_{r,j}(tL)e^{-tL/2}v_{t^{1/m}}^{\delta_{p_0}}(f\chi_{E})\|_{L^{q_0}(X;L^2(dr/r))}\notag\\
    &\leq\|\chi_{F}v_{t^{1/m}}^{\delta_{q_0}}\widetilde{H}_{r,j}(tL)e^{-tL/2}v_{t^{1/m}}^{\delta_{p_0}}(f\chi_{E})\|_{L^{2}((dr/r);L^{q_0}(X))}\notag\\
    &\lesssim\left(1+\frac{d(E,F)}{t^{1/m}}\right)^{-M}\|f\chi_{E}\|_{L^{p_0}}\|\widetilde{H}_{r,j}\|_{L^2(dr/r;W^{M+(\kappa/2),2}(\mathbb{R}))}\notag\\
    &\lesssim2^{(M+(\kappa/2))}\left(1+\frac{d(E,F)}{t^{1/m}}\right)^{-M}\|f\chi_{E}\|_{L^{p_0}}\|g_j\|_{W^{\eta,2}(\mathbb{R})}.\label{5.42}
\end{align}
On the other hand, we have
    $$\chi_{F}v_{t^{1/m}}^{\delta_{q_0}}\varphi(sL)Q_{t}^Nv_{t^{1/m}}^{\delta_{p_0}}\chi_F=c_{N}^{-1}(\chi_Fv_{t^{1/m}}^{\delta_{q_0}}e^{-tL/4})\circ(\varphi(sL)(tL)^Ne^{-tL/2})\circ(e^{-tL/4}v_{t^{1/m}}\chi_{E}).$$
    Thus, by \eqref{5.31}, \eqref{5.32} with $\beta$ replaced by $\eta$, and \cite[Proposition 3.1 \textnormal{(ii)}]{B07}, it follows that
    \begin{align}
        &\quad\ \|\chi_{F}v_{t^{1/m}}^{\delta_{q_0}}\varphi_j(sL)Q_{t}^Nv_{t^{1/m}}^{\delta_{p_0}}\chi_F\|_{p_0\rightarrow q_0(\Lambda^{\eta})}\notag\\
        &\leq\|\chi_Fv_{t^{1/m}}^{\delta_{q_0}}e^{-tL/4}\| _{2(\Lambda^{\eta})\rightarrow q_0(\Lambda^{\eta})}\|\varphi_j(sL)(tL)^Ne^{-tL/2}\|_{2\rightarrow2(\Lambda^\eta)}\|e^{-tL/4}v_{t^{1/m}}\chi_{E}\|_{p_0\rightarrow2}\notag\\
        &\lesssim\|\chi_Fv_{t^{1/m}}^{\delta_{q_0}}e^{-tL/4}\|_{2\rightarrow q_0}\|g_j(s+\log L)(tL)^N e^{-tL/2}\|_{2\rightarrow2(W^{\eta,2}(\mathbb{R}))}\notag\\
        &\lesssim\|g_j\|_{W^{\eta,2}(\mathbb{R})}\|(tL)^Ne^{-tL}\|_{2\rightarrow2}\lesssim\|g_j\|_{W^{\eta,2}(\mathbb{R})}.\label{5.43}
    \end{align}
    For the second inequality, we used that $(tL)^Ne^{-tL}f=L[t^N L^{N-1}e^{-tL}f]\in\textnormal{Ran}\ L$ for every $f\in L^2(X)$. The last inequality follows from that $\sup_{\lambda\geq0} (t\lambda)^Ne^{-t\lambda}<\infty$ uniformly for all $t>0$. 

    We now fix a $\sigma>\nu\max\{1/p_0,1/q_0^{\prime}\}$. For $M>\sigma$, \eqref{5.42} and \eqref{5.43} give that
    \begin{align}
        \|\chi_{F}v_{t^{1/m}}^{\delta_{q_0}}\varphi_j(sL)Q_{t}^Nv_{t^{1/m}}^{\delta_{p_0}}\chi_F\|_{p_0\rightarrow q_0(\Lambda^{\eta})}&\lesssim\|g_j\|_{W^{\eta,2}(\mathbb{R})}\min\left\{1,2^{j(M+(\kappa/2))}\left(1+\frac{d(E,F)}{t^{1/m}}\right)^{-M}\right\}\notag\\
        &\leq\|g_j\|_{W^{\eta,2}(\mathbb{R})}\min\left\{1,2^{(\kappa\sigma)j/(2M)}\left(\frac{1}{2^j}+\frac{d(E,F)}{2^jt^{1/m}}\right)^{-\sigma}\right\}\notag\\
        &\lesssim\|g_j\|_{W^{\eta,2}(\mathbb{R})}2^{(\kappa\sigma)j/(2M)}\left(1+\frac{d(E,F)}{2^jt^{1/m}}\right)^{-\sigma}\label{5.44}
    \end{align}
    For the first inequality, we used $\min\{1,Y\}\leq\min\{1,Y^{\sigma/M}\}$. The second inequality follows by considering separately the cases $t^{1/m}\geq d(E,F)2^{-j}$ and $t^{1/m}<d(E,F)2^{-j}$. Next for any $\varepsilon>0$, choose $M$ sufficiently large such that $(\kappa\sigma)/(2M)<\epsilon$. Then by the doubling estimate $v_{2^jt^{1/m}}(x)\lesssim 2^{j\nu}v_{t^{1/m}}(x)$, we obtain that
    $$\|\chi_{F}v_{2^jt^{1/m}}^{\delta_{q_0}}\varphi_j(sL)Q_{t}^Nv_{2^jt^{1/m}}^{\delta_{p_0}}\chi_E\|_{p_0\rightarrow q_0(\Lambda^{\eta})}\lesssim 2^{j(\nu(1/p_0-1/q_0)+\varepsilon)}\|g_j\|_{W^{\eta,2}(\mathbb{R})}\left(1+\frac{d(E,F)}{2^jt^{1/m}}\right)^{-\sigma}.$$
    Combining this with \eqref{5.36}, \eqref{5.37} and \eqref{5.38}, we now apply Theorem \ref{sdc} to $T=(\varphi_j(sL))_{s>0}$ with $R=2^j$, $A\leq \|g_j\|_{W^{\eta,2}(\mathbb{R})}$ and $D=2^{j(\nu(1/p_0-1/q_0)+\varepsilon)}\|g_{j}\|_{W^{\eta,2}(\mathbb{R})}$. For bounded functions $f,g$ with compact support, by Lemma \ref{Dyadic} we choose a dyadic cube $Q_0\in\bigcup_{b=1}^{K_D} \mathcal{D}^b$ such that $\textnormal{supp}\ f\, \cup\, \textnormal{supp}\ g\subset Q_0\subset Q_0^*$. Thus, It follows from Lemma \ref{lambdaemb} that for every $j\geq0$, there exist sparse families $\mathcal{S}_{b,j}\subset \mathcal{D}^b$, $b=1,...,K_{D}$, such that
    \begin{equation}\label{5.45}
        |\langle V^{q}(\varphi_j(tL):t>0)f,g\rangle|\lesssim\langle\|\varphi_j(sL)f\|_{\Lambda^{\eta}},|g|\rangle\lesssim2^{j(\nu(1/p_0-1/q_0)+\varepsilon)}\|g_j\|_{W^{\eta,2}(\mathbb{R})}\sum_{b=1}^{K_D}\Lambda_{\mathcal{S}_{b,j};p_0,q_0}(f,g).
    \end{equation}
    Then by Lemma \ref{universal}, we can replace $\mathcal{S}_{b,j}$ here by sparse families $\mathcal{S}_b$ independent of $j$. Finally, for $\beta>\nu(1/p_0-1/q_0)+1/2$, we choose $\eta>1/2$ and $\varepsilon>0$ such that $\beta>\nu(1/p_0-1/q_0)+\eta+\epsilon$. It follows from \eqref{5.35} and the Cauchy-Schwarz inequality that
    \begin{align}
         |\langle V^{q}(\varphi(tL):t>0)f,g\rangle|&\leq\sum_{j\geq0}\, |\langle V^{q}(\varphi_j(tL):t>0)f,|g|\rangle|\notag\\
         &\lesssim\sum_{j\geq0}2^{j(\nu(1/p_0-1/q_0)+\eta+\epsilon-\beta)}\|g_{j}\|_{W^{\beta,2}(\mathbb{R})}\sum_{b=1}^{K_{D}}\Lambda_{\mathcal{S}_{b};p_0,q_0}(f,g)\notag\\
         &\leq\left(\sum_{j\geq0}2^{2j(\nu(1/p_0-1/q_0)+\eta+\epsilon-\beta)}\right)^{1/2}\left(\sum_{j\geq0}\|g_j\|_{W^{\beta,2}(\mathbb{R})}^2\right)^{1/2}\sum_{b=1}^{K_{D}}\Lambda_{\mathcal{S}_{b};p_0,q_0}(f,g).\notag\\
         &\lesssim\|g\|_{W^{\beta,2}(\mathbb{R})}\sum_{b=1}^{K_{D}}\Lambda_{\mathcal{S}_{b};p_0,q_0}(f,g).\notag
    \end{align}
    This completes the proof of Theorem \ref{thm2}.
\end{proof}
\begin{rem}
     Our initial proof applied the sparse domination criterion Lemma \ref{spDo} by Bernicot, Frey and Petermichl under a stronger condition $\beta>\nu+\nu/p_0+2$. Their proof in \textnormal{\cite{BFP16}} passes from off-diagonal estimates between two balls to estimates between a ball and an annulus. In our setting, this covering argument causes an additional loss of decay. In Lemma \ref{ext}, we therefore extend the complex-time generalized Gaussian estimates to arbitrary measurable sets $E,F$. These give off-diagonal estimates for $\varphi (sL) Q_{t}^N$ between such sets and weaken the condition on $\beta$ to $\beta>\sigma\max\{1/p_0,1/q_0^{\prime}\}+1$. Finally, we use a dyadic Fourier decomposition. For each dyadic piece, we have a new estimate $\eqref{5.43}$. Combining this with the usual estimate $\eqref{5.42}$ allows us to reduce the additional power of $(1+d(E,F)/2^jt^{1/m})$ to arbitrarily small number. This is the key point in our argument. For fixed $p$, the term $\nu(1/p_0-1/q_0)$ gives $L^p$-boundedness under the condition $\beta>\nu|1/p-1/2|$, which agrees with the sharp regularity condition for some examples (see Section 6 for more details). The additional $1/2$ comes from the requirement $\eta>1/2$ in the embedding $\Lambda^\eta\hookrightarrow V^q$. 
\end{rem}

\section{Sharpness of the regularity condition}
In this section, we show that the regularity condition in Theorem \ref{thm2} can not be lowered in general. We find a non-negative self-adjoint operator $L$ on $L^2(\mathbb{R})$. The semigroup will satisfies \textnormal{(DGGE$_{1,\infty,2}$)}, and thus \textnormal{(DGGE$_{p_0,q_0,2}$)} for every $1\leq p_0<2<q_0\leq\infty$. We show that for every $0\leq\alpha<1/p_0-1/q_0+1/2$, there exist a $\varphi\in\Lambda^{\alpha}$ such that $\varphi(L)\notin S(p_0,q_0)$. This conclude the sharpness for $q$-variational sparse bound. The counterexample in this section was generated by ChatGPT during discussions with the author. We checked the correctness and explains the argument below.

\noindent\textit{\textbf{The idea of the construction.} Following the method in \textnormal{\cite[Section 3]{BC20}}, We take $f=\chi_{E}$ and $g=\chi_{F}$, where $E$ and $F$ are intervals with $d(E,F)\sim R$. Then for every $Q$ in sparse family, the condition $Q\in E,F\neq \varnothing$ gives a uniform upper bound for $\Lambda_{\mathcal{S};p_0,q_0}(f,g)$ by $R^{1/q_0-1/p_0}$. For the operator, We construct a local operator with a simple eigenvalue at $1$. The integrals of its normalized eigenfunction $\Phi$ over $E$ and $F$ are bounded below, and a spectral gap allow us to write $\varphi(L)f=\langle \Phi f,f\rangle \Phi$}. Finally, we translate such local operators into disjoint intervals to build up a global operator that satisfies Gaussian dominations. The uniform boundedness principle gives a single multiplier for which the sparse bound fails.

\medskip

For $R\in2\pi\mathbb{N}$ with $R\geq100$ (we choose a sufficiently large constant), let $a_R=\pi-\arctan(R^{-1/2})$, $b_R=R+2a_{R}$. We consider the following form on $L^2(0,b_{R})$,
\begin{equation}\label{A1}
    \mathfrak{q}_{R}(u,v)=\int_{0}^{b_{R}} u^{\prime}(x)\overline{v^{\prime}(x)}\,dx+\sqrt{R}(u(a_{R})\overline{v(a_{R})}+u(a_{R}+R)\overline{v(a_{R}+R)})
\end{equation}
with $D(\mathfrak{q}_R)=W^{1,2}_0(0,b_{R})$. Point values are taken from the continuous representatives of $W^{1,2}$ functions. It follows by a simple calculation that $\mathfrak{q}_{R}$ is densely defined, closed, symmetric and non-negative. Thus, it defines a non-negative self-adjoint operator $L_{R}$ on $L^2(0,b_{R})$.

We next define the global operator. Let $R_{j}=2\pi(j+20)>100$, $s_1=0$ and $s_{j+1}=s_{j}+b_{R_j}$ for $j\geq1$. We write $J_{0}=(-\infty,0)$, $J_{j}=(s_j,s_{j+1})$, $c_{j,1}=s_{j}+a_{R_j}$ and $c_{j,2}=c_{j,1}+R_j$, which correspond to $a_R$ and $a_R+R$ in \eqref{A1}. Define a form on $L^{2}(\mathbb{R})$ that
\begin{equation}\label{A2}
\mathfrak{a}(u,v)=\int_{\mathbb{R}}u^{\prime}(x)\overline{v^{\prime}(x)}\,dx+\sum_{j\geq 1}\sqrt{R_j}(u(c_{j,1})\overline{v}(c_{j,1})+u(c_{j,2})\overline{v(c_{j,2})})
\end{equation}
with domain
$$D(\mathfrak{a})=\{u\in W^{1,2}(\mathbb{R}):u(s_j)=0\ (j\geq1),\ \sum_{j\geq1}\sqrt{R_j}(|u(c_{j,1})|^2+|u(c_{j,2})|^2)<\infty\}.$$
Thus, $\mathfrak{a}$ is densely defined, symmetric and non-negative on $L^2(\mathbb{R})$. We briefly check closedness. The form norm is $\|u\|_{\mathfrak{a}}^2=\|u\|^{2}_{W^{1,2}(\mathbb{R})}+\sum_j\sqrt{R_j}(|u(c_{j,1})|^2+|u(c_{j,2})|^2)$. If $(u_k)_k$ is a Cauchy sequence in this norm, then $u_k\rightarrow u$ in $W^{1,2}(\mathbb{R})$ and $(R_j^{1/4}u_{k}(c_{j,l}))_{j,l}\rightarrow z$ in $\ell^2$. By the Sobolev embedding $W^{1,2}(\mathbb{R})\hookrightarrow C_{0}(\mathbb{R})$, we have $u(s_j)=0$ and $z_{j,l}=R_{j}^{1/4}u(c_{j,l})$. Hence $u\in D(\mathfrak{a})$ and $u_k\rightarrow u$ in the form norm. This implies that $\mathfrak{a}$ is closed. Thus, $\mathfrak{a}$ defines a non-negative self-adjoint operator $L$ on $L^2(\mathbb{R})$. Formally, we may write
\begin{equation}\label{L}
L=-\partial_x^2+\sum_{j\geq1}\sqrt{R_j}(\delta_{c_{j,1}}+\delta_{c_{j,2}}),
\end{equation}
where $\delta_c$ denote the Dirac distribution that $\langle\delta_{c}u,v\rangle=u(c)v(c)$. We then check the semigroup assumption. For $u\in D(\mathfrak{a})$, we have $w=(\textnormal{Re}\,u)^{+}\in D(\mathfrak{a})$ and $\mathfrak{a}(w,w)\leq \mathfrak{a}(u,u)$. It follows from \cite[Theorem 2.7]{O05} that semigroup $\{e^{-tL}\}_{t>0}$ is positive. Let $\mathfrak{b}(u,v)=\int_{\mathbb{R}}u^{\prime}(x)\overline{v^{\prime}(x)}\,dx$ with $D(\mathfrak{b})=W^{1,2}(\mathbb{R})$ be the form associated with $-\partial_x^2$. Then for $u\in D(\mathfrak{a})$, $v\in W^{1,2}(\mathbb{R})$ with $0\leq v\leq u$, we have $v(s_j)=0$ and $\sum_{j,l}\sqrt{R_j}|v(c_{j,l})|^2<\infty$. Thus $v\in D(\mathfrak{a})$. This is the ideal condition in \cite[Proposition 2.23]{O05}. Moreover, for non-negative $u,v\in D(\mathfrak{a})$, we have 
$$\mathfrak{a}(u,v)-\mathfrak{b}(u,v)=\sum_{j\geq1}\sqrt{R_j}(u(c_{j,1})v(c_{j,1})+u(c_{j,2})v(c_{j,2}))\geq 0.$$
 The domination criterion \cite[Theorem 2.24]{O05} therefore gives the domination $|e^{-tL}f(x)|\leq e^{t\Delta_{\mathbb{R}}}|f|(x)$ a.e. $x\in\mathbb{R}$. The free heat semigroup is positive and contractive on $L^2(\mathbb{R})$. It also satisfies \textnormal{(GGE$_{1,\infty,2}$)}. Thus, we obtain \textnormal{(DGGE$_{1,\infty,2}$)} for $L$, which also implies \textnormal{(DGGE$_{p_0,q_0,2}$)} for every $1\leq p_0<2<q_0\leq \infty$.

 We collect some properties in the following proposition. Its proof will be given at the end of this section. For $j\geq1$, let $U_j: L^2(0,b_{R_j})\rightarrow L^2(J_j)$ be the translation $(U_jv)(x)=v(x-s_j)$, and let $L_j=U_j L_{R_j} U_{j}^{-1}$ be the translation of $L_R$ into interval $J_j$. Let $L_0$ be the Laplacian $-\partial_x^2$ on $J_0=(-\infty,0)$, which is associated with the form $\mathfrak{q}_0(u,v)=\int_{-\infty}^0u^{\prime}(x)\overline{v^{\prime}(x)}\,dx$ with $D(\mathfrak{q_0})=W^{1,2}_0(J_0)$.
\begin{prop}\label{propertyL}
Let $R\in2\pi\mathbb N$, $R\geq100$, the following propositions holds.\\

\noindent\textnormal{(i)} For $j\geq0$, let $\tau_j: L^2(J_j)\rightarrow L^2(\mathbb{R})$ be the extension by zero. Then for every bounded Borel function $\varphi:[0,\infty)\rightarrow \mathbb{C}$ and $f\in L^2(\mathbb{R})$, we have
    $$\varphi(L)f=\sum_{j\geq0}\tau_j \varphi(L_j)(f\,|_{J_j}).$$
\noindent \textnormal{(ii)} Define
\begin{equation}\label{eigenfunction}
u_R(x)=
\begin{cases}
\sin x,&0\leq x\leq a_R,\\
\sin a_R\cos(x-a_R),&a_R\leq x\leq a_R+R,\\
\sin(b_R-x),&a_R+R\leq x\leq b_R.
\end{cases}\notag
\end{equation}
Then $u_R\in D(L_R)$, $L_Ru_R=u_R$, and
$\ker(L_R-I)=\operatorname{span}\{u_R\}$, $\|u_R\|_2^2\sim1$. Let $\Phi_R=u_R/\|u_R\|_2$, $E_R=(\pi/4,3\pi/4)$ and $F_R=b_R-E_R$. Then
\begin{equation}\label{sh:mass}
\int_{E_R}\Phi_R=\int_{F_R}\Phi_R=\frac{\sqrt2}{\|u_R\|_2}.\notag
\end{equation}
\noindent\textnormal{(iii)} There is a constant $c_0>0$ such that
\begin{equation}
\textnormal{dist}\left(1,\sigma(L_R)\setminus\{1\}\right)\geq\frac{c_0}{R}.\notag
\end{equation}
\end{prop}
Here \textnormal{(i)} shows that $L$ is the orthogonal sum of $L_j$ on the intervals $J_j$, where $L_j$ is a translated copy of $L_{R_j}$ for $j\geq 1$. The condition $u(s_j)=0$ ensure that these intervals do not interact. \textnormal{(ii)} and \textnormal{(iii)} provide the eigenfunction estimates and the spectral gap for each $L_{j}$. We first assume Proposition \ref{propertyL} and prove the existence of the counterexample.
\begin{thm}\label{counterexample}
    Let $L$ be the operator above. Let $1\leq p_0<2<q_0\leq\infty$. For every $0\leq\beta<1/p_0-1/q_0+1/2$, there exists $\varphi\in\Lambda^{\beta}$ such that $\varphi(L)\notin S(p_0,q_0)$.
\end{thm}

\begin{proof}
     Since $\Lambda^{\alpha}\subset \Lambda^{\beta}$ for $0\leq\beta\leq\alpha$, we may assume $\beta>1/2$ without loss of generality. For $j\geq 1$, set $E_j=s_j+E_{R_j}$, $F_{j}=s_j+F_{R_j}$ be the translation of $E_{R}$ and $F_{R}$ in Proposition \ref{propertyL}. Let $f_{j}=\chi_{E_j}$ and $g_{j}=\chi_{F_j}$. For each sparse family $\mathcal{S}\subset \mathcal{D}$ of dyadic cubes (interval), we consider $\Lambda_{\mathcal{S};p_0,q_0}(f_j,g_j)=\sum_{Q\in\mathcal{S}}\langle f_j\rangle_{p_0,Q}\langle g_j \rangle_{q_0^{\prime},Q}|Q|$. Only cube $Q$ meeting both $E_j$ and $F_j$ contribute to the sum. Since $R_j>100$ and $0\leq a_j\leq \pi$, we have $d(E_j,F_j)=R_j+2a_j-(3\pi/2)\sim R_j$. Thus every contributing cube $Q$ satisfies $\ell(Q)\gtrsim R_j$. For every $k\geq0$, number of elements in $\{Q\in\mathcal{D}:Q\cap E_{j}\neq\varnothing,\ Q\cap F_j\neq\varnothing,\ \ell(Q)\sim 2^k R_j\}$ is uniformly bounded. Hence, 
    $$\sum_{\substack{Q\in\mathcal{S}\\\ell(Q)\sim 2^k R_j}}\langle f_j\rangle_{p_0,Q}\langle g_j\rangle_{q_0^{\prime},Q}|Q|\lesssim|E_j|^{1/p_0}|F_j|^{1/q_0^{\prime}}(2^{k}R_j)^{1-1/p_0-1/q_0^{\prime}}\sim 2^{-k(1/p_0-1/q_0)}R_j^{1/q_0-1/p_0}.$$
    Summing over $k$ gives that
    \begin{equation}\label{A4}
        \Lambda_{\mathcal{S};p_0,q_0}(f_j,g_j)\lesssim R_j^{1/q_0-1/p_0}.
    \end{equation}
    Next we choose a $\chi\in C_c^{\infty}(-\varepsilon,\varepsilon)$ with $0<\varepsilon<\min\{c_0/4,1\}$ and $\chi(0)=1$. For $R\in 2\pi\mathbb{N}$ with $R\geq100$, define $\varphi_{R}(\lambda)=\chi(R\log\lambda)$ for $\lambda>0$ and $\varphi_{R}(0)=0$. Then
    \begin{equation}\label{A5}
    \|\varphi_R\|_{\Lambda^{\beta}}=\|\chi(R\cdot)\|_{W^{\beta,2}(
    \mathbb{R}
    )}=\left(\int_{\mathbb{R}}|\langle\xi\rangle^{\beta}\mathcal{F}\chi(\xi/R)R^{-1}|^2\,d\xi\right)^{1/2}\leq R^{\beta-(1/2)}\|\chi\|_{W^{\beta,2}(\mathbb{R})}.
    \end{equation}
    For every $\lambda\in\textnormal{supp}\,\varphi_{R}$, $|\log\lambda|<\varepsilon/R$, so $|\lambda-1|\leq e^{|\log\lambda|}|\log\lambda|<2\varepsilon/R<c_0/R$. Thus $\textnormal{supp}\,\varphi_R\subset [1/2,2]$, and it follows from Proposition \ref{propertyL} \textnormal{(ii)}, \textnormal{(iii)} that $\sigma(L_{R})\cap\textnormal{supp}\,\varphi_R=\{1\}$. Since $\varphi_R(1)=1$, the spectral theorem gives
    \begin{equation}
        \varphi_{R}(L_R)=\int_{\sigma(L_R)}\varphi_{R}(\lambda)\,dE_{L_{R}}(\lambda)=E_{L_{R}}(\{1\}),\notag
    \end{equation}
    which is the orthogonal projection onto $\ker(L_R-I)$. Hence, for every $f\in L^2(0,b_R)$, $f-\varphi_{R}(L_R)f\in\ker(L_R-I)^{\perp}$. Recall that $\Phi_R$ is the normalized eigenfunction. We have
    \begin{align*}
        \langle (\varphi_{R}(L_R)f-\langle f,\Phi_R\rangle \Phi_{R}),\Phi_R\rangle&=\langle \varphi_{R}(L_R)f-f,\Phi_{R}\rangle+\langle f,\Phi_R\rangle-\langle f,\Phi_R\rangle\|\Phi_R\|^2_{L^2(0,b_R)}=0.
    \end{align*}
   By Proposition \ref{propertyL} \textnormal{(ii)}, $\ker(L_R-I)=\textnormal{span}\{\Phi_R\}$. Therefore, $\varphi_{R}(L_R)f-\langle f,\Phi_R\rangle \Phi_{R}\in\ker(L_R-I)\cap \ker(L_R-I)^{\perp}$. Thus, 
    \begin{equation}\label{A6}
    \varphi_{R}(L_R)f=\langle f,\Phi_R\rangle \Phi_{R}.
    \end{equation}
    for every $f\in L^2(0,b_R)$. Set $\Phi_j=U_j \Phi_{R_j}$. Since $L_j=U_j L_{R_j} U_{j}^{-1}$, unitary invariance of the spectral calculus gives $\varphi_{R_j}(L_j)=U_j \varphi _{R_j}(L_{R_j})U_j^{-1}$ on $L^2(J_j)$. Thus, Proposition \ref{propertyL} \textnormal{(i)}, \textnormal{(ii)} and \eqref{A6} yield that
    \begin{equation}\label{A7}
    \langle\varphi_{R_j}(L)f_j,g_j\rangle=\langle\varphi_{R_j}(L_{j})f_j,g_j\rangle=\left(\int_{E_j}\Phi_j(x)\,dx\right)\left(\int_{F_j}\Phi_j(x)\,dx\right)=\frac{2}{\|u_{R_j}\|_{2}^2}\sim 1
    \end{equation}

    Now let $X_{\beta}:=\{m\in\Lambda^\beta:\textnormal{supp}\,m\subset[1/2,2]\}$ with $\|m\|_{X^{\beta}}=\|m\|_{\Lambda^{\beta}}$. The correspondence $m\rightarrow m\circ\textnormal{exp}$ gives the completeness. Thus $X_{\beta}$ is a Banach space. Since $\beta>1/2$ and $m\circ\textnormal{exp}$ support on $[-\log 2,\log2]$, Sobolev embedding gives
    $$\|m\|_{\infty}=\|m\circ\textnormal{exp}\|_{\infty}\lesssim\|m\circ\textnormal{exp}\|_{W^{\beta,2}(\mathbb{R})}=\|m\|_{\Lambda^\beta}.$$
    Here we use the continuous representative of each $m$ and set $m(0)=0$. For $j\geq1$, define 
    $$T_j: X_\beta\rightarrow \mathbb{C},\quad T_j(m)=R_j^{1/p_0-1/q_0}\langle m(L)f_j,g_j\rangle.$$
    Cauchy's inequality and the spectral theorem give that
    \begin{equation}\label{A8}
    |T_j(m)|\leq R_j^{1/p_0-1/q_0}\|m\|_{\infty}\|f_j\|_{2}\|g_j\|_{2}\lesssim R_j^{1/p_0-1/q_0}\|m\|_{\Lambda^\beta}.
    \end{equation}
    This implies that each $T_j$ is a bounded linear operator. Suppose that $m(L)\in S(p_0,q_0)$ for every $m\in X_\beta$. Then, it follows from \eqref{A4} that $\sup_{j\geq1}|T_j(m)|<\infty$ for every $m\in X_\beta$. The uniform boundedness principle implies that $\sup_{j\geq1}\|T_j\|<\infty$. However, by \eqref{A5} and \eqref{A7} we have $\|T_j\|\geq|T_j(\varphi_{R_j})|/ \|\varphi_{R_j}\|_{\Lambda^{\beta}}\gtrsim R_j^{1/p_0-1/q_0+1/2-\beta}.$
    Since $1/p_0-1/q_0+1/2-\beta>0$ and $R_j=2\pi(j+20)\rightarrow\infty$, we have $\|T_j\|\rightarrow\infty$, which leads to a contradiction. 
    \end{proof}

    Finally, we prove Proposition \ref{propertyL}.
    \begin{proof}[Proof of Proposition \ref{propertyL}] \textbf{Proof of \textnormal{(i)}.} Let $\mathfrak{q}_j$ be the form associated with $L_j$. For $j\geq1$, $\mathfrak{q}_j$ is a translation of \eqref{A1}. It follows from definition that for $u\in D(L_j)$, $\tau_j u\in D(\mathfrak{a})$. Moreover, for every $v\in D(\mathfrak{a})$, we have that
    $$\mathfrak{a}(\tau_j u,v)=\mathfrak{q}_j(u,v\,|_{J_j})=\langle L_ju,v\,|_{J_j}\rangle=\langle \tau_j L_j u,v\rangle,$$
    which implies $\tau_ju\in D(L)$ and $L\tau_ju=\tau_j L_j u$. Then \cite[Problem 4.6]{T14} gives $\tau_j \varphi(L_j)\subset \varphi(L)\tau_j$. Since $D(\varphi(L))=L^2(\mathbb{R})$ for bounded $\varphi$, we obtain $\tau_j \varphi(L_j)= \varphi(L)\tau_j$. Thus, 
    $$\varphi(L)f=\sum_{j\geq0}\varphi(L)\tau_j(f\,|_{J_j})=\sum_{j\geq0}\tau_j\varphi_j(L_j)(f\,|_{J_j}).$$ 
    \noindent \textbf{Proof of $\textnormal{(ii)}$.} We first check the desired properties for $u_R$. Since $R\in2\pi\mathbb{N}$, it follows from a direct calculation that $u_R$ is continuous and vanishes at both endpoints. It also satisfies $-u_R^{\prime\prime}=u_{R}$ on each subinterval. Thus $u_R\in D(\mathfrak{q}_R)=W^{1,2}_0(0,b_R)$. It follows from integration by parts that for every $v\in W^{1,2}_0(0,b_R)$,
    \begin{align}
        \int_{0}^{b_R}u_R^{\prime}(x)\overline{v^{\prime}(x)}\,dx&=-\int_{0}^{b_R}u_{R}^{\prime\prime}(x)\overline{v(x)}\,dx+u_{R}^{\prime}\overline{v}\,|_{0}^{a_{R}^-}+u_{R}^{\prime}\overline{v}\,|_{a_R^+}^{(a_R+R)^-}+u_{R}^{\prime}\overline{v}\,|_{(a_R+R)^+}^{b_R}\notag\\
        &=\int_{0}^{b_R} u_{R}(x)\overline{v(x)}\,dx+(u_{R}^{\prime}(a_R^-)-u_R^{\prime}(a_R^+))\overline{v(a_R)}+(u_R^{\prime}((a_R+R)^-)-u_R^{\prime}((a_R+R)^+)\overline{v(a_R+R)}\notag\\
        &=\int_{0}^{b_R} u_{R}(x)\overline{v(x)}\,dx-\sqrt{R}(u_R(a_R)\overline{v(a_R)}+u_R(a_R+R)\overline{v(a_R+R)})\label{A9}.
    \end{align}
    Thus, $\mathfrak{q}_R(u_R,v)=\langle u_R,v\rangle$, which implies $u_R\in D(L_R)$ and $L_R u_{R}=u_R$. Recall that $a_R=\pi-\arctan(R^{-1/2})$, so $\sin a_R=1/\sqrt{R+1}$ and $\cos a_R=-\sqrt{R}/\sqrt{R+1}$. Direct integration yields that
    $$\|u_R\|_{2}^2=a_R+\frac{\sqrt{R}}{R+1}+\frac{R}{2(R+1)}\sim 1$$
    for $R>100$. Moreover, $\int_{E_R}\Phi_R(x)\,dx=\|u_R\|_2^{-1}\int_{\pi/4}^{3\pi/4}\sin x\,dx=\|u_R\|_2^{-1}\sqrt{2}$ and $\int_{F_R}\Phi_{R}(x)\,dx=\int_{E_R}\Phi_{R}(b_R-x)\,dx=\int_{E_R}\Phi_{R}(x)\,dx$. It remains to verify $\textnormal{dim}(\ker(L_R-I))=1$.

    Let $w\in D(L_R)$ satisfy $L_Rw=w$. Then
    $\mathfrak{q}_R(w,v)=\langle w,v\rangle$ for every
    $v\in W^{1,2}_0(0,b_R)$. Taking $v\in C_c^\infty(I)$,
    where $I$ is any of the three subintervals, we have $(w^{\prime})^{\prime}=-w\in L^{2}(I)$ in the sense of distributions on $I$. Thus $w\in W^{2,2}(I)$ on each subinterval. Integration by parts as in \eqref{A9} gives
    \begin{equation}
    0=
    \bigl(w^\prime(a_R^-)-w^\prime(a_R^+)+\sqrt{R}w(a_R)\bigr)
    \overline{v(a_R)}+
    \bigl(w^\prime((a_R+R)^-)-w^\prime((a_R+R)^+)
    +\sqrt{R}w(a_R+R)\bigr)
    \overline{v(a_R+R)}.\notag
\end{equation}
Since $v$ can be chosen arbitrarily, we have
\begin{equation}\label{A10}
w^\prime(a_R^-)-w^\prime(a_R^+)+\sqrt{R}w(a_R)=0,\qquad w^\prime((a_R+R)^-)-w^\prime((a_R+R)^+)+\sqrt{R}w(a_R+R)=0.
\end{equation}
On $(0,a_R)$, the equation $-w^{\prime\prime}=w$ and the condition $w(0)=0$ imply that $w(x)=c\sin x$. Let $z=w-cu_{R}$, we have $z\equiv0$ on $(0,a_R)$. The continuity gives that $z(a_R^+)=0$ and $z^{\prime}(a_R^+)=z^\prime(a_R^-)+\sqrt{R}z(a_R)=0$. Thus, we have $z\equiv0$ on $(a_R,a_R+R)$. Repeating this argument at $a_R+R$ gives $z\equiv0$ on $(a_R+R,b_R)$. Hence $w=c u_{R}$ on $(0,b_R)$, which completes the proof of \textnormal{(ii)}.

\medskip

\noindent\textbf{Proof of \textnormal{(iii)}.} We first check that all points in spectrum of $L_R$ are eigenvalues. We refer to \cite{LSW10} for the argument in this part. Since $L_R$ is non-negative self-adjoint, the form norm of $L_R$ is defined as \begin{equation}\label{A11}
\|u\|_{Q(L_R)}=\left(\|u\|_{L^2(0,b_R)}^2+\|(L_R+I)^{1/2}u\|_{L^2(0,b_R)}^2\right)^{1/2}=\left(\mathfrak{q}_R(u,u)+2\|u\|_{L^2(0,b_R)}^2\right)^{1/2}
\end{equation}
with form domain $Q(L_R)=D(\mathfrak{q}_R)=W^{1,2}_0(0,b_R)$. Moreover, 
$$\|u\|_{W^{1,2}(0,b_R)}\leq\left(\|u\|_{L^2(0,b_R)}^2+\mathfrak{q}_R(u,u)\right)^{1/2}\sim\|u\|_{Q(L_R)}$$
for every $u\in Q(L_R)$. Thus the identity map $Q(L_R)\hookrightarrow W^{1,2}_0(0,b_R)$ is bounded. By the compact Sobolev embedding, the inclusion $W^{1,2}_0(0,b_R)\hookrightarrow L^2(0,b_R)$ is compact. Hence the composition $\textnormal{Id}:Q(L_R)\rightarrow L^2(0,b_R)$ is also compact. By \cite[Corollary 1.5]{LSW10}, $L_R$ has empty essential spectrum. Thus, the spectrum consists of isolated eigenvalues.

We now prove the spectral gap estimate in \textnormal{(iii)}. Since $L_R$ is non-negative, $\sigma(L_R)\subset [0,\infty)$. Let $\delta=10^{-2}$ (any sufficiently small constant would also do), and let $k^2$ be an eigenvalue of $L_R$ with $k>0$ and $|k-1|<\delta/R$. We will show that $k=1$. First choose a nonzero eigenfunction $w$. That is $w\in D(L_R)$ and $L_R w=k^2w$. As in the proof of \textnormal{(ii)}, we have $w\in W^{2,2}(I)$ and $-w^{\prime\prime}=k^2 w$ on each subinterval $I=(0,a_R)$, $(a_R,a_R+R)$ and $(a_R+R,b_R)$). Moreover, $w$ is continuous on $[0,b_R]$ and vanishes at both endpoints. Integration by parts as in \eqref{A9} gives the same estimate \eqref{A10}. Thus there exists $C,D\in\mathbb{C}$ that 
$$w(x)=
    \begin{cases}
        C\sin(kx), & 0<x<a_R,\\
        D\sin(k(b_R-x)), & a_R+R<x<b_R.
    \end{cases}$$
Here $C\neq 0$, else the same argument as in \textnormal{(ii)} would give $w=0$ on $(0,b_R)$. The endpoint condition \eqref{A10} gives that
\begin{equation}\label{A12}
w^{\prime}(a_R^+)=Ck\cos(ka_R)+\sqrt{R}C\sin (ka_R)= A(k)w(a_R),\ w^{\prime}((a_R+R)^-)=-A(k)w(a_R+R),
\end{equation}
where $A(k)=k\cot(ka_R)+\sqrt{R}$. Since $w(a_R)\neq0$, we may normalize $w$ so that $w(a_R)=1$. Solving $-w^{\prime\prime}=k^2 w$ on the middle interval $(a_R,a_R+R)$ with $w(a_R)=1$ and $w^{\prime}(a_R^+)=A(k)w(a_R)$ gives that
\begin{equation}
    w(x)=\cos(k(x-a_R))+\frac{A(k)}{k}\sin (k(x-a_R)),\qquad a_R\leq x\leq R+a_R.\notag
\end{equation}
Thus, at the endpoint $a_R+R$ we have $w(a_R+R)=\cos(kR)+k^{-1}A(k)\sin(kR)$ and $w^{\prime}((a_R+R)^{-})=-k\sin(kR)+A(k)\cos(kR)$. By \eqref{A12} we obtain
\begin{equation}\label{A13}
    (A(k)^2-k^2)\sin(kR)+2kA(k)\cos(kR)=0.
\end{equation}
We next derive some estimates for $A(k)$. Since $|ka_R-a_R|<\pi\delta/R$ and $\pi-a=\arctan (R^{-1/2})\geq R^{-1/2}/2$, we have $\pi/2<ka_R<\pi$ and 
$$\sin(ka_R)\geq\sin a_R-a_R|k-1|\geq\frac{1}{\sqrt{R+1}}-\frac{\pi\delta}{R}\geq\frac{1}{2\sqrt{R+1}}.$$ 
Thus, $A^{\prime}(k)=\cot(ka_R)-ka_R\sin^{-2}(ka_R)<0$. Since $A(1)=\cot a_{R}+R^{1/2}=0$, it follows that $A(k)<0$ for $k>1$ and $A(k)>0$ for $k<1$. Moreover, for every $k$ in the interval $|k-1|<\delta/R$ we have $k>1-\delta R^{-1}>0.99$, which imlies
$$|A^{\prime}(k)|\leq 2\sqrt{R+1}+\frac{\pi}{(2\sqrt{R+1})^{-2}}=2\sqrt{R+1}+4\pi(R+1)<40R.$$
Hence $|A(k)|=|A(k)-A(1)|\leq 40R|k-1|<0.4<k$, so $A(k)^2-k^2<0$. We now return to \eqref{A13}. Since $R\in 2\pi\mathbb{N}$ and $|(k-1)R|<\delta=10^{-2}$, we have $\cos(kR)=\cos((k-1)R)>0$. If $k>1$, then $\sin(kR)=\sin((k-1)R)>0$ and $A(k)<0$. Since we have verified $A(k)^2-k^2<0$, both terms on the left-hand side of \eqref{A13} are negative, which leads to a contradiction. If $k<1$, then $\sin(kR)<0$ and $A(k)>0$. In this case, both terms are positive, which is also a contradiction. Hence $k=1$.

Finally, let $\lambda\in\sigma(L_R)$ satisfy $|\lambda-1|<\delta/R$. Then $\lambda>0$ and $|\sqrt{\lambda}-1|\leq|\lambda-1|<\delta/R$. The preceding argument gives $\lambda=1$, which proves \textnormal{(iii)} with $c_0=\delta$.
\end{proof}

\section{Weighted consequences and applications}

In this section, we first establish weighted estimates for fixed $p$ under a weaker regularity condition on the multiplier. Then, we prove Corollary \ref{application} for the applications to Bochner-Riesz means and regularized Schr\"{o}dinger groups. Finally, we give examples of operators $L$ satisfying our assumptions.
\subsection{Weighted estimates for spectral multipliers}
As we noted in the remark after Theorem \ref{thm2}, for fixed $p\in(p_0,q_0)$, the regularity condition on $\beta$ can be improved by choosing suitable $p_1,q_1$. We now give the corresponding weighted estimates.
\begin{prop}\label{weightedes}
    Let $L$ be a non-negative self-adjoint operator on $L^2(X)$. Suppose that the semigroup $\{e^{-tL}\}_{t>0}$ satisfies \textnormal{(DGGE$_{p_0,q_0,m}$)} for some $m\geq2$ and $1\leq p_0<2<q_0\leq\infty$. Fix $p_0<p<q_0$, $q>2$ and $\beta>\nu|1/p-1/2|+1/2$. Let $\varphi(\lambda)=\varphi(0)e^{-\lambda}+\varphi_0(\lambda)$ for $\lambda>0$, where $\varphi_0\in\Lambda^{\beta}$. Then there exists $\varepsilon_0>0$, such that for every $w\in A_{\max\{1,p/2\}+\varepsilon_0}\cap RH_{1/((1-p/2)_++\varepsilon_0)}$, we have
    $$\|V^q(\varphi(tL):t>0)f\|_{L^{p}(w)}\leq C(|\varphi(0)|+\|\varphi_0\|_{\Lambda^\beta})\left([w]_{A_{\max\{1,p/2\}+\varepsilon_0}}[w]_{RH_{1/((1-p/2)_++\varepsilon_0)}}\right)^{\gamma_{p,\varepsilon_0}}\|f\|_{L^{p}(w)},$$
    where $\gamma_{p,\varepsilon_0}:=(\max\{1,p-1\}+\varepsilon_0)/(p\varepsilon_0)$.
    \end{prop}

    \begin{proof}[Proof of Proposition \ref{weightedes}]
        Choose $\varepsilon_0>0$ sufficiently small and define $p_1,q_1$ by
        \begin{equation}\label{6.1}
        \frac{1}{p_1}=\max\left\{\frac{1}{p},\frac{1}{2}\right\}+\frac{\varepsilon_0}{p},\quad\frac{1}{q_1}=\min\left\{\frac{1}{p},\frac{1}{2}\right\}-\frac{\varepsilon_0}{p}.
        \end{equation}
        Since $p_0<p,2<q_0$, we can choose $\varepsilon_0$ so that
        $$p_0\leq p_1<\min\{p,2\}\leq\max\{p,2\}<q_1\leq q_0$$ 
        and $2\varepsilon_0/p<[(\beta-(1/2))/\nu]-|1/p-1/2|$. Then, $1/p_1-1/q_1=|1/p-1/2|+2\varepsilon_0/p$, so $\beta>\nu(1/p_1-1/q_1)+1/2$. By H\"{o}lder's inequality, \textnormal{(DGGE$_{p_0,q_0,m}$)} implies \textnormal{(DGGE$_{p_1,q_1,m}$)}. Thus, by Theorem \ref{thm2} and Lemma \ref{wei-lem}, we obtain 
        \begin{equation}\label{6.2}
        \|V^q(\varphi(tL):t>0)f\|_{L^{p}(w)}\leq C(|\varphi(0)|+\|\varphi_0\|_{\Lambda^\beta})\left([w]_{A_{p/p_1}}[w]_{RH_{(q_1/p)^{\prime}}}\right)^{\widetilde{\gamma}}\|f\|_{L^{p}(w)}
        \end{equation}
        for every $w\in A_{p/p_1}\cap RH_{(q_1/p)^{\prime}}$, with $\widetilde{\gamma}=\max\{1/(p-p_1),(q_1-1)/(q_1-p)\}$. By the definition of $p_1,q_1$, we have $p/p_1=\max\{1,p/2\}+\varepsilon_0$ and $(q_1/p)^{\prime}=1/((1-p/2)_++\varepsilon_0)$. For $p\leq2$,
        $$\widetilde{\gamma}=\max\left\{\frac{1+\varepsilon_0}{p\varepsilon_0},\frac{\frac{1}{2}+\frac{\varepsilon_0}{p}}{1-\frac{p}{2}+\varepsilon_0}\right\}=\frac{1+\varepsilon_0}{p\varepsilon_0}=\gamma_{p,\varepsilon_0}.$$
        For $p>2$, 
        $$\widetilde{\gamma}=\max\left\{\frac{p+2\varepsilon_0}{p^2-2(1-\varepsilon_0)p},\frac{p-1+\varepsilon_0}{p\varepsilon_0}\right\}=\frac{p-1+\varepsilon_0}{p\varepsilon_0}=\gamma_{p,\epsilon_0}.$$
       This completes the proof.
    \end{proof}
    
    \begin{rem}
       The weighted estimates in Corollary \ref{wei-cor} and Proposition \ref{weightedes}
       also have vector-valued versions. Let $q>2$ and $\varphi_k(\lambda)=\varphi_k(0)e^{-\lambda}+\phi_k(\lambda)$ for $k\in\mathbb{N}$. Assume that
       \begin{itemize}
        \item In the setting of Theorem \ref{thm1}, there exist $\alpha,\beta>0$ that $\phi_k\in \Psi_{\alpha,\beta}(\Sigma_{\mu}
        )$, and $\sup_{k}\|\phi_k\|_{\Psi_{\alpha,\beta}(\Sigma_\mu)}+|\varphi_k(0)|<\infty$.
        \item In the setting of Theorem \ref{thm2}, $\sup_k \|\phi_k\|_{\Lambda^{\beta}}+|\varphi_k(0)|<\infty$. 
        \end{itemize}
        Under these assumptions, the scalar weighted estimates are uniform in $k$. Then it follows from the extrapolation theorem of Auscher and Martell \textnormal{\cite[Theorem 4.9]{AM07}} \textnormal{(see also Section 5 of \cite{AM07} for the extension to spaces of homogeneous type)} that
    \begin{equation}\label{vector}
    \left\|
\left(\sum_k[V^q(\varphi_k(tL):t>0)f_k]^r\right)^{1/r}\right\|_{L^p(w)}
\lesssim\left\|\left(\sum_k|f_k|^r\right)^{1/r}\right\|_{L^p(w)}
\end{equation}
for every $p,r\in(p_0,q_0)$ and $w\in A_{p/p_0}\cap RH_{(q_0/p)^{\prime}}$. The same argument applies to the fixed-$p$ estimates in Proposition \ref{weightedes}. For fixed $p_0<p<q_0$, under the uniform condition $\sup_k \|\phi_k\|_{\Lambda^{\beta}}+|\varphi_k(0)|<\infty$ \textnormal{(}$\beta>\nu|1/p-1/2|+1/2$\textnormal{)}, the estimate \eqref{vector} holds for $p_1<r<q_1$ and $w\in A_{\max\{1,p/2\}+\varepsilon_0}\cap RH_{1/((1-p/2)_++\varepsilon_0)}$, where $p_1,q_1$ are defined in \eqref{6.1}.
    \end{rem}
We next prove the Corollary \ref{application}. Recall that the corresponding multipliers for Bocher-Riesz means and regularized Schr\"{o}dinger groups are
$$\widetilde{B}^\delta(\lambda)=(1-\lambda)_+^\delta,
\qquad
\widetilde{S}^s(\lambda)=(1+\lambda)^{-s}e^{i\lambda}.$$
Since both of them take the value $1$ at zero, we consider $\widetilde{B}^\delta-e^{-\lambda}$ and
$\widetilde{S}^s-e^{-\lambda}$. We verify that both multipliers belong to the Sobolev class in Theorem \ref{thm2}.
\begin{lem}\label{Application}
    Let $\beta\geq0$. Then $B^{\delta}(\lambda):=(1-\lambda)_+^{\delta}-e^{-\lambda}\in\Lambda^{\beta}$ for $\delta>\beta-(1/2)$, and $S^s(\lambda
    ):=(1+\lambda)^{-s}e^{i\lambda}-e^{-\lambda}\in\Lambda^{\beta}$ for $s>\beta$.
\end{lem}
\begin{proof}
    We start with the Bochner-Riesz multiplier. First, we consider the case $\beta=n\in\mathbb{N}$. By the change of variables $\lambda=e^{x}$ and direct differentiation, we have
    \begin{align}
        \|f\|_{\Lambda^{n}}^2=\|f\circ\textnormal{exp}\|_{W^{n,2}(\mathbb{R})}^2&\lesssim\sum_{r=0}^{n}\int_{0}^{\infty}|[(\lambda\partial_{\lambda})^rf](\lambda)|^2\,\frac{d\lambda}{\lambda}\notag\\
        &\lesssim\sum_{r=0}^n\int_{0}^{\infty}|\lambda^{r}(\partial_{\lambda}^rf)(\lambda)|^2\,\frac{d\lambda}{\lambda}.\label{6.4}
    \end{align}
    For $w\in\mathbb{C}$, we fix $\textnormal{Re}\, w>n-(1/2)$ and consider the function $B^w(\lambda):=(1-\lambda)_+^w-e^{-\lambda}$. For $r=0$, Taylor's formula gives
    $$B^w(\lambda)=(1-w\lambda+\textnormal{O}(\lambda^2))-(1-\lambda+\textnormal{O}(\lambda^2))=(1-w)\lambda+\textnormal{O}(\lambda^2) \quad \text{as}\quad \lambda\rightarrow0.$$
   Thus, it follows that $\int_0^\infty |B^w(\lambda)|^2\lambda^{-1}\,d\lambda\lesssim(1+|\textnormal{Im}\ w|)^2$, where the implicit constant may depend on $\textnormal{Re}\, w$. For $1\leq r\leq n$
    $$|\lambda^r\partial_{\lambda}^r B^w(\lambda)|\lesssim\lambda^re^{-\lambda}+\left(\prod_{l=0}^{r-1}|w-l|\right)\lambda^r(1-\lambda)^{w-r}\quad\text{for }0\leq\lambda\leq1.
    $$
    Hence, we obtain that for $\textnormal{Re}\, w>n-(1/2)$, $\int_{0}^{\infty}|\lambda^r(\partial_{\lambda}^r\,B^w )(\lambda)|^2\lambda^{-1}\,d\lambda\lesssim(1+|\textnormal{Im}\, w|)^{2n}$. Combining these estimates with \eqref{6.4}, we have $\|B^w\|_{\Lambda^{n}}\lesssim(1+|\textnormal{Im}\, w|)^{n+1}$. The implicit constants are uniform when $\textnormal{Re}\,w$ ranges over $[a,b]$ for $n-1/2< a\leq b<\infty$.
    
    For non-integer $\beta$, we prove the statement by complex interpolation. Suppose $\beta=k+\theta$, where $k\in\mathbb{N}$ and $0<\theta<1$. Then for $\delta>\beta-(1/2)$, define 
    $$F_{B}:\{z\in\mathbb{C}:0\leq\textnormal{Re}\ z\leq1\}\rightarrow \Lambda^{k},\qquad F_B(z):= e^{(z-\theta)^2}B^{\delta-\theta+z}.$$
    Since $\Lambda^{k+1}\subset \Lambda^k$, we have $\Lambda^{k}=\Lambda^{k}+\Lambda^{k+1}$. Moreover, the estimate $\textnormal{Re}\ (\delta-\theta+z)>k-(1/2)$ implies that $F_B(z)\in \Lambda^k$. The function $F_B$ is analytic in $\{z\in \mathbb{C}: 0<\textnormal{Re}\,z<1\}$ and continuous on $\{z\in \mathbb{C}: 0\leq\textnormal{Re}\,z\leq1\}$. The continuity follows from 
    $\|B^{w+h}-B^w\|_{\Lambda^k}\rightarrow0$ as $h\rightarrow0$. Indeed, fix $w$ with $\textnormal{Re}\,w>k-1/2$ and choose $0<a<\textnormal{Re}\,w-k+1/2$. For $|h|<a$, we have
    \begin{equation}\label{6.5}
    |\partial_{\lambda}^r (B^{w+h}-B^w)(\lambda)|=\left|h\int_{0}^1\partial_{w}\partial_{\lambda}^r B^{w+\theta h}(\lambda)\,d\theta\right|\leq|h|\sup_{0\leq\theta\leq1}|\partial_{w}\partial_{\lambda}^r B^{w+\theta h}(\lambda)|.
    \end{equation}
    By the calculation above, we have
    $$|\partial_{w}\partial_{\lambda}^r B^{w+\theta h}(\lambda)|\leq| p_{r}^{\prime}(w+\theta h)(1-\lambda)^{w+\theta h-r}|+|p_r(w+\theta h)(1-\lambda)^{w+\theta h-r}\log(1-\lambda)|,$$
    where $p_r(w)=\prod_{l=0}^{r-1}(w-l)$ for $r\geq1$ and $p_0(w)=1$. Both $|p_r(w+\theta h)|$ and $|p^{\prime}_r(w+\theta h)|$ are uniformly bounded for $0\leq\theta\leq 1$ and $|h|<a$, and $\textnormal{Re}\,(w+\theta h)\geq \textnormal{Re}\,w-a$. Thus, 
    $$|\lambda^r\partial_{\lambda}^r (B^{w+h}-B^{w})(\lambda)|\leq C|h|\lambda^r[(1-\lambda)^{\textnormal{Re}\,w-a-r}+|\log(1-\lambda)|(1-\lambda)^{\textnormal{Re}\,w-a-r}].$$
    Since $\textnormal{Re}\,w>a+k-(1/2)\geq a+r-(1/2)$, combining this with \eqref{6.4} we obtain $\|B^{w+h}-B^h\|_{\Lambda^k}\lesssim |h|\rightarrow0$. Together with the continuity of $e^{(z-\theta)^2}$, this gives the continuity of $F_B$. Analyticity follows from a similar argument by checking that
    $$\lim_{h\rightarrow0}\left\|\frac{B^{w+h}-B^w}{h}-\partial_w B^{w}\right\|_{\Lambda^k}=0\qquad\text{as }t\rightarrow0.$$
    We omit the calculation here. The integer estimate implies 
    \begin{equation}\label{6.6}
        \|F_B(ia)\|_{\Lambda^{k}}\lesssim e^{\theta^2-a^2}(1+|a|)^{k+1}\leq C_1\qquad\text{and}\qquad\|F_B(1+ib)\|_{\Lambda^{k+1}}\lesssim e^{(1-\theta)^2-b^2}(1+|b|)^{k+2}\leq C_2.
    \end{equation}
    Since the integer estimate above is uniform for $\textnormal{Re}\,w\in[\delta-\theta,\delta-\theta+1]$, we also have $$\|F_B(u+it)\|_{\Lambda^{k}}\lesssim e^{u^2-t^2}(1+|t|)^{k+1}\leq C_3$$
    uniformly for $0<u<1$ and $t\in\mathbb{R}$. Therefore, $B^{\delta}=F_B(\theta)\in[\Lambda^{k},\Lambda^{k+1}]_{\theta}$. By \cite[Section 6.4]{BL76}, we have $[W^{k,2}(\mathbb{R}),W^{k+1,2}(\mathbb{R})]_{\theta}=W^{\beta,2}(\mathbb{R})$, and the map $f\rightarrow f\circ\textnormal{exp}$ identifies $\Lambda^s$ with $W^{s,2}(\mathbb{R})$ for $s=k,k+1$. Thus, we also have $[\Lambda^{k},\Lambda^{k+1}]_{\theta}=\Lambda^{\beta}$ (see \cite[Theorem 6.4.2]{BL76}). Hence, we obtain $B^{\delta}\in \Lambda^{\beta}$.

    The same argument applies to the regularized Schr\"{o}dinger multiplier. For the case $\beta\in\mathbb{N}$, let $w\in\mathbb{C}$ and $S^W(\lambda):=(1+\lambda)^{-w}e^{i\lambda}-e^{-\lambda}$. When $\lambda\rightarrow0$, we have
    $$S^w(\lambda)=(1-w\lambda+\textnormal{O}(\lambda^2))(1+i\lambda+\textnormal{O}(\lambda^2))-(1-\lambda+\textnormal{O}(\lambda^2))=(1+i-w)\lambda+\textnormal{O}(\lambda^2).$$
    For $1\leq r\leq n$, 
    $$\partial_{\lambda}^r S^w(\lambda)=e^{i\lambda}\sum_{j=0}^r i^{r-j}(-1)^j\left(\prod_{l=0}^{j-1}(w+l)\right)(1+\lambda)^{-w-j}-(-1)^re^{-\lambda}.$$
    Thus, we have that
    \begin{equation}\label{6.7}
    |\lambda^r \partial_{\lambda}^{r}S^w(\lambda)|\lesssim\lambda^r e^{-\lambda}+(1+|w|)^r\lambda^r(1+\lambda)^{-\textnormal{Re}\,w}.
    \end{equation}
    Since $\int_{0}^{\infty}\lambda^{2r-1}(1+\lambda)^{-2\textnormal{Re}\,w}\,d\lambda<\infty$ when $\textnormal{Re}\ w>r$, it follows from \eqref{6.4} that $\|S^w\|_{\Lambda^n}\lesssim (1+|\textnormal{Im}\,w|)^{n+1}$ for $\textnormal{Re}\,w>n$. The implicit constant is uniform when $\textnormal{Re}\,w\in[a,b]$ for $n<a\leq b<\infty$. For non-integer $\beta=k+\theta$, we apply the same complex interpolation argument with 
    $F_{S}(z):=e^{(z-\theta)^2}S^{s-\theta+z}$, $0\leq\textnormal{Re}\,z\leq1.$ This gives $S^s=F_{S}(\theta)\in\Lambda^\beta$ for $s>\beta$.
\end{proof}

Corollary \ref{application} then follow from Proposition \ref{weightedes}.

\begin{rem}
    Chen, Duong, Li and Yan \textnormal{\cite{CDLY20}} proved the sharp fixed-time estimate
    $\|(I+L)^{-s}e^{itL}\|_{L^{p}\rightarrow L^p}\lesssim(1+|t|)^s$ for $s\geq s_0=\nu|1/2-1/p|$. Our condition $s>s_0+(1/2)$ for the variational estimate of regularized Sch\"{o}dinger group $\{(I+tL)^{-s}e^{itL}\}_{t>0}$ may not be optimal for fixed $p_0<p<q_0$. However, the variational estimates can require more regularity than fixed-time cases.
    
    Consider the harmonic oscillator $L=-\Delta+|x|^2$ on $\mathbb{R}^n$. This is a non-negative self-adjoint operator on $L^2(X)$, and the semigroup satisfies $|e^{-tL}f|\leq e^{t\Delta}|f|$ \textnormal{(}see Subsection 6.2 below\textnormal{)}, so it satisfies \textnormal{(DGGE$_{1,\infty,2}$)}. Let $h(x)=e^{-|x|^2/2}$, then $Lh=nh$, and the spectral theorem gives
    $$(I+tL)^{-s}e^{itL}h=(1+nt)^{-s}e^{int}h.$$
    For $t_k=k\pi/n$, we have $(I+t_kL)^{-s}e^{it_kL}h(x)=(-1)^k(1+k\pi)^{-s}h(x)$. Thus, 
    $$V^q((I+tL)^{-s}e^{itL}:t>0)h(x)\geq|h(x)|\left(\sum_{k=1}^\infty|(1+(k+1)\pi)^{-s}+(1+k\pi)^{-s}|^q\right)^{1/q}.$$
    Therefore, the $q$-variational boundedness requires $\sum_{k=1}^\infty k^{-sq}<\infty$. Thus, $s>1/q$ is necessary in this example. In particular, when $p=2$, $1/2$ in our condition can not be lowered if we allow the variational bound holds for every $q>2$.   
\end{rem}
\subsection{Examples of $L$}
We finally give several examples of operators that their semigroups satisfy the assumptions of our main theorems. Recall that if $\{e^{-tL}\}_{t>0}$ is positive contraction semigroup on $L^2(X)$ satisfying \textnormal{(GGE$_{p_0,q_0,m}$)}, then \textnormal{(DGGE$_{p_0,q_0,m}$)} holds with $R_t=e^{-tL}$.

\paragraph{(1) Schr\"{o}dinger operators with non-negative potentials.} 
This is a standard example with Gaussian bounds. For $0\leq V\in L^1_{loc}(\mathbb{R}^n)$, consider the Schr\"{o}dinger operator
$$L=-\Delta+V$$
given by the associated form $\mathfrak{a}(u,v):=\int_{\mathbb{R}^n}\nabla u\cdot\overline{\nabla v}\,dx+\int_{\mathbb{R}^n}Vu\overline{v}\,dx$ with $D(\mathfrak{a})=\{u\in W^{1,2}(\mathbb{R}^n):V^{1/2}u\in L^2(\mathbb{R}^N)\}.$ It is non-negative and self-adjoint. By the Feynman-Kac formula, we have $0\leq e^{-tL}f\leq e^{t\Delta}f$ for $f\geq0$. Thus, $\{e^{-tL}\}_{t>0}$ is a positive contraction on $L^2(\mathbb{R}^n)$ satisfying \textnormal{(GGE$_{1,\infty,2}$)}, which proves \textnormal{(DGGE$_{1,\infty,2}$)}. This includes the harmonic oscillator $-\Delta+|x|^2$ in the Remark above.

\paragraph{(2) Schr\"{o}dinger operators with negative potentials}
For negative potentials, generalized Gaussian estimates may hold only for a restricted range. Let $n\geq 3$ and let $V\geq0$ be strongly subcritical, that is for some $\varepsilon>0$,
$$\int_{\mathbb{R}^n}V|u|^2\,dx\leq \frac{1}{1+\varepsilon}\int_{\mathbb{R}^n}|\nabla u|^2\,dx$$
for all $u\in W^{1,2}(\mathbb{R}^n)$. We consider the Schr\"{o}dinger operator 
$$L=-\Delta-V$$
defined by the form $\mathfrak{a}(u,v)=\int_{\mathbb{R}^n}\nabla u\cdot\overline{\nabla v}\,dx
-\int_{\mathbb{R}^n}Vu\overline v\,dx$ with $D(\mathfrak{a})=W^{1,2}(\mathbb{R}^n)$. It follows from \cite{A11} that this form gives a non-negative self-adjoint operator, and $-L$ generates an contraction semigroup on $L^2(X)$. Moreover, for every real valued $u\in D(\mathfrak{a})$, we have $|u|\in D(\mathfrak{a})$ and 
$$\mathfrak{a}(|u|,|u|)=\int_{\mathbb{R}^n}|\nabla|u||^2\,dx-\int_{\mathbb{R}^n}V|u|^2\,dx
 \leq\mathfrak {a}(u,u).$$
Thus, the semigroup $\{e^{-tL}\}_{t>0}$ is positive (see \cite[Proposition 2.5, Theorem 2.7]{O05}). It remains to check the generalized Gaussian estimate. Let $p_0=2n/[(n-2)(1-\sqrt{1-1/(1+\varepsilon)})]>2$. By \cite{A11}, the semigroup satisfies $L^p-L^2$ off-diagonal estimates $$\|\chi_{F}e^{-tL}\chi_{E}\|_{p\rightarrow2}\leq C t^{-(n/2)(1/p-1/2)}\,\textnormal{exp}\left(-c{d(E,F)^2}{t^{-1}}\right)$$
for closed sets $E,F$ and $p_0^{\prime}<p\leq 2$. By self-adjointness, we also have
$$\|\chi_{F}e^{-tL}\chi_{E}\|_{2\rightarrow q}\leq C t^{-(n/2)(1/2-1/q)}\,\textnormal{exp}\left(-c{d(E,F)^2}{t^{-1}}\right)$$
for $2\leq q<p_0$. Now we fix $p_0^{\prime}<p<2$ and let $B_1,B_2$  be two balls of radius $t^{1/2}$. If $d(B_1,B_2)=0$, the estimates above give that
$$\|\chi_{B_2}e^{-tL}\chi_{B_1}\|_{p\rightarrow p^{\prime}}\leq\|\chi_{B_2}e^{-tL/2}\|_{2\rightarrow p^{\prime}}\|e^{-tL/2}\chi_{B_1}\|_{p\rightarrow 2}\lesssim t^{-(n/2)(1/p-1/p^{\prime})}=|B_1|^{-1/p}|B_2|^{1/p^{\prime}}.$$
Otherwise if $d(B_1,B_2)>0$, let $G:=\{x: d(x,B_1)<d(B_1,B_2)/2\}$. Then $d(B_2,G)=d(B_1,G^c)= d(B_1,B_2)/2$. Since 
$$\chi_{B_2}e^{-tL}\chi_{B_1}=\chi_{B_2}e^{-tL/2}\chi_{G}e^{-tL/2}\chi_{B_1}+\chi_{B_2}e^{-tL/2}\chi_{G^c}e^{-tL/2}\chi_{B_1},$$ 
the off-diagonal estimates yield
\begin{align*}
\|\chi_{B_2}e^{-tL}\chi_{B_1}\|_{p\rightarrow p^{\prime}}&\leq\|\chi_{B_2}e^{-tL/2}\chi_{G}\|_{2\rightarrow p^{\prime}}\|\chi_{G}e^{-tL/2}\chi_{B_1}\|_{p\rightarrow2}+\|\chi_{B_2}e^{-tL/2}\chi_{G^c}\|_{2\rightarrow p^{\prime}}\|\chi_{G^c}e^{-tL/2}\chi_{B_1}\|_{p\rightarrow2}\\
&\lesssim t^{-(n/2)(1/p-1/p^{\prime})}\textnormal{exp}(-c^{\prime}d(E,F)^2 t^{-1})=|B_1|^{-1/p}|B_2|^{1/p^{\prime}}\textnormal{exp}(-c^{\prime}d(E,F)^2 t^{-1}).
\end{align*}
Thus, the semigroup satisfies \textnormal{(GGE$_{p,p^{\prime},2}$)}. Since the semigroup is positive and contractive on $L^2(\mathbb {R}^n)$, it also satisfies \textnormal{(DGGE$_{p,p^{\prime},2}$)}.

\paragraph{(3) Elliptic operators with real coefficients.} This class includes non-self-adjoint operators. Let $A$ be an $n\times n$ matrix of real and $L^{\infty}$ valued coefficients defined on $\mathbb{R}^n$. We assume the ellipticity condition: there exists $0<\lambda\leq\Lambda<\infty$ such that
$$\textnormal{Re}\,(A(x)\xi\cdot\overline{\xi})\geq\lambda|\xi|^2\quad\text{and}\quad |A(x)\xi\cdot\overline{\eta}|
 \leq\Lambda|\xi||\eta|$$
for all $\xi,\eta\in\mathbb{C}^n$ and a.e. $x\in\mathbb{R}^n$. We consider the the operator 
$$L=-\textnormal{div}(A\,\nabla)$$
defined by the form $\mathfrak{a}(u,v)=\int_{\mathbb{R}^n}A\nabla u\cdot\overline{\nabla v}\,dx$ with $D(\mathfrak{a})=W^{1,2}(\mathbb{R}^n)$. It follows from \cite{AM06} that $L$ is injective and $\omega$-accretive for some $0\leq\omega<\pi/2$. They also noted that $-L$ generates is a contraction semigroup on $L^2(\mathbb{R}^n)$. For real coefficients, the heat kernel is non-negative and satisfies the pointwise Gaussian upper bound. Thus, the semigroup $\{e^{-tL}\}_{t>0}$ satisfies \textnormal{(DGGE$_{p_0,q_0,m}$)}.

\paragraph{(4) Magnetic Schr\"{o}dinger operators.} The domination condition is nontrivial in this example. Let $a=(a_1,...,a_n)$ be a real vector with $a_j\in L^2_{loc}(\mathbb{R}^n)$, and let $0\leq V\in L^{1}_{loc}(\mathbb{R}^n)$. We consider the magnetic Schr\"{o}dinger operator
$$L=-(\nabla-ia)^2+V,$$
which is defined by the form $\mathfrak{a}(u,v)=\int_{\mathbb {R}^n}D_a u\cdot\overline{D_a v}\,dx
+\int_{\mathbb {R}^n}Vu\overline v\,dx$, where $D_a=\nabla-ia$ and $D(\mathfrak{a})=\{u\in L^2(\mathbb{R}^n):D_a u\in L^2(\mathbb{R}^n;\mathbb{C}^n), \ V^{1/2}u\in L^2(\mathbb{R}^n)\}$. As shown in \cite{DOY06}, this form is densely defined, closed, symmetric and non-negative. Hence $L$ is a non-negative self-adjoint operator on $L^2(\mathbb{R}^n)$. By the diamagnetic inequality \cite[Section 2]{DOY06}, we have
$|e^{-tL}f|\leq e^{t\Delta}|f|$ for $t>0$. Now take $R_t=e^{t\Delta}$. The heat semigroup is positive and contractive on $L^2(\mathbb{R}^n)$ and satisfies \textnormal{(GGE$_{1,\infty,2}$)}. Thus, $\{e^{-tL}\}_{t>0}$ satisfies \textnormal{(DGGE$_{1,\infty,2}$)}.

\paragraph{(5) Laplace-Beltrami operators} We next consider an example beyond the Euclidean setting. We follow the setting and notation of Grigor'yan and
Saloff-Coste \cite[Section 2]{GSC05}. Let $M$ be a complete connected Riemannian manifold without boundary, equipped with the geodesic distance $d$ and a doubling Riemannian measure $\mu$. Suppose $M$ satisfies the $L^2$ Poincar\'{e} inequality. We consider the Laplace-Beltrami operator $$L=-\Delta_M$$ defined by the form $\mathfrak {a}(u,v)
 :=\int_M g(\nabla u,\overline{\nabla v})\,d\mu.$ The form is densely defined, closed, symmetric and non-negative. Thus, $L$ is a non-negative and self-adjoint operator on $L^2(M)$. The spectral theorem gives the contraction. By \cite[Theorem 2.7]{GSC05}, the heat kernel is positive and satisfies the pointwise Gaussian estimate on both sides. This implies that the semigroup $\{e^{t\Delta_{M}}\}_{t>0}$ satisfies \textnormal{(DGGE$_{1,\infty,2}$)}.

\addcontentsline{toc}{section} {\protect\numberline{}Acknowledgements and AI Declaration}
\section*{Acknowledgements}
\indent The author is grateful to his supervisor, Professor Y. Tsutsui for guidance and helpful discussions throughout this work.

\section*{AI Declaration}
\indent The author developed the main methods for Theorem \ref{thm1} and \ref{thm2} and wrote the initial drafts of the proof. ChatGPT was used to refine some of the arguments and results. We also used ChatGPT to improve the English, search for related literature, locate existing results that are used as some of Lemmas here, and assist with calculations. We list some of the important parts as below.
\begin{itemize}
    \item The counterexample and its proof in Section 5 were generated entirely by ChatGPT through discussion.
    \item The proof of Lemma \ref{Application} and the remark following it were generated by ChatGPT. The examples and explanations in Section 6.2 were also developed with assistance from AI.
    \item ChatGPT pointed out the need to choose suitable representatives and suggested the choices used at the beginning of Section 3.1, 3.2 and 4.2.
    \item ChatGPT provided the stronger estimate in
    Lemma \ref{ext}\textnormal{(ii)}. The author first proved a weaker version that was sufficient for the argument in this paper.
    \item ChatGPT suggested the vector-valued formulation of Theorem \ref{sdc}. It also helped remove the domination
    assumption for the Sobolev-type multipliers in Theorem \ref{thm2}.
\end{itemize}
The final manuscript was written and revised by the author with AI assistance mentioned above. We have checked all mathematical arguments and references, and take full responsibility for the content of this paper.


\begin{thebibliography}{99}
\begin{sloppypar}

\bibitem{ADM96}{D. Albrecht, X. T. Duong and A. McIntosh,
``\textit{Operator theory and harmonic analysis}''
in \textit{Instructional Workshop on Analysis and Geometry, Part III (Canberra, 1995)}, Proc. Centre Math. Appl. Austral. Nat. Univ., vol. 34, Australian National University, Canberra, 1996, pp. 77--136.}

\bibitem{AM06}{P. Auscher and J. M. Martell,
``\textit{Weighted norm inequalities, off-diagonal estimates and elliptic operators. Part III: Harmonic analysis of elliptic operators}''
J. Funct. Anal. \textbf{241} (2006), no. 2, 703--746.}

\bibitem{AM07}{P. Auscher and J. M. Martell,
``\textit{Weighted norm inequalities, off-diagonal estimates and elliptic operators. Part I: General operator theory and weights}''
Adv. Math. \textbf{212} (2007), no. 1, 225--276.}

\bibitem{A07}{P. Auscher,
``\textit{On necessary and sufficient conditions for $L^p$-estimates of Riesz transforms associated to elliptic operators on $\mathbb{R}^n$ and related estimates}''
Mem. Amer. Math. Soc. \textbf{186} (2007), no. 871, xviii+75 pp.}

\bibitem{A11}{J. Assaad,
``\textit{Riesz transforms associated to Schr\"odinger operators with negative potentials}''
Publ. Mat. \textbf{55} (2011), no. 1, 123--150.}

\bibitem{BL76}{J. Bergh and J. L\"ofstr\"om,
``\textit{Interpolation spaces. An introduction}''
Grundlehren der mathematischen Wissenschaften, vol. 223, Springer-Verlag, Berlin-New York, 1976.}

\bibitem{B89}{J. Bourgain,
``\textit{Pointwise ergodic theorems for arithmetic sets}''
Inst. Hautes \'{E}tudes Sci. Publ. Math. \textbf{69} (1989), 5--41.}

\bibitem{B03}{S. Blunck,
``\textit{A H\"ormander-type spectral multiplier theorem for operators without heat kernel}''
Ann. Sc. Norm. Super. Pisa Cl. Sci. (5) \textbf{2} (2003), no. 3, 449--459.}

\bibitem{BK03}{S. Blunck and P. C. Kunstmann,
``\textit{Calder\'on-Zygmund theory for non-integral operators and the $H^\infty$ functional calculus}''
Rev. Mat. Iberoam. \textbf{19} (2003), no. 3, 919--942.}

\bibitem{BK05}{S. Blunck and P. C. Kunstmann,
``\textit{Generalized Gaussian estimates and the Legendre transform}''
J. Operator Theory \textbf{53} (2005), no. 2, 351--365.}

\bibitem{B07}{S. Blunck,
``\textit{Generalized Gaussian estimates and Riesz means of Schr\"odinger groups}''
J. Aust. Math. Soc. \textbf{82} (2007), no. 2, 149--162.}

\bibitem{BFP16}{F. Bernicot, D. Frey and S. Petermichl,
``\textit{Sharp weighted norm estimates beyond Calder\'on-Zygmund theory}''
Anal. PDE \textbf{9} (2016), no. 5, 1079--1113.}

\bibitem{BBL17}{C. Benea, F. Bernicot and T. Luque,
``\textit{Sparse bilinear forms for Bochner-Riesz multipliers and applications}''
Trans. Lond. Math. Soc. \textbf{4} (2017), no. 1, 110--128.}

\bibitem{BMSW18}{J. Bourgain, M. Mirek, E. M. Stein and B. Wr\'obel,
``\textit{On dimension-free variational inequalities for averaging operators in $\mathbb R^d$}''
Geom. Funct. Anal. \textbf{28} (2018), no. 1, 58--99.}

\bibitem{BC20}{D. Beltran and L. Cladek,
``\textit{Sparse bounds for pseudodifferential operators}''
J. Anal. Math. \textbf{140} (2020), no. 1, 89--116.}

\bibitem{BRS24}{D. Beltran, J. Roos and A. Seeger,
``\textit{Multi-scale sparse domination}''
Mem. Amer. Math. Soc. \textbf{298} (2024), no. 1491, v+104 pp.}

\bibitem{CW71}{R. R. Coifman and G. Weiss,
``\textit{Analyse harmonique non-commutative sur certains espaces homog\`enes. \'Etude de certaines int\'egrales singuli\`eres}''
Lecture Notes in Mathematics, vol. 242, Springer-Verlag, Berlin-New York, 1971.}

\bibitem{C90}{M. Christ,
``\textit{A $T(b)$ theorem with remarks on analytic capacity and the Cauchy integral}''
Colloq. Math. \textbf{60/61} (1990), no. 2, 601--628.}

\bibitem{COSY16}{P. Chen, E. M. Ouhabaz, A. Sikora and L. Yan,
``\textit{Restriction estimates, sharp spectral multipliers and endpoint estimates for Bochner-Riesz means}''
J. Anal. Math. \textbf{129} (2016), no. 1, 219--283.}

\bibitem{CDLY20}{P. Chen, X. T. Duong, J. Li and L. Yan,
``\textit{Sharp endpoint $L^p$ estimates for Schr\"odinger groups}''
Math. Ann. \textbf{378} (2020), no. 1--2, 667--702.}

\bibitem{CLSY20}{P. Chen, S. Lee, A. Sikora and L. Yan,
``\textit{Bounds on the maximal Bochner-Riesz means for elliptic operators}''
Trans. Amer. Math. Soc. \textbf{373} (2020), no. 6, 3793--3828.}

\bibitem{CLWY26}{P. Chen, X. Lin, L. Wu and L. Yan,
``\textit{On maximal functions generated by H\"ormander-type spectral multipliers}''
J. Funct. Anal. \textbf{290} (2026), no. 2, Paper No. 111238.}

\bibitem{D97}{E. B. Davies,
``\textit{Limits on $L^p$ regularity of self-adjoint elliptic operators}''
J. Differential Equations \textbf{135} (1997), no. 1, 83--102.}

\bibitem{D01}{J. Duoandikoetxea,
``\textit{Fourier analysis}''
Graduate Studies in Mathematics, vol. 29, American Mathematical Society, Providence, RI, 2001.}

\bibitem{DOY06}{X. T. Duong, E. M. Ouhabaz and L. Yan,
``\textit{Endpoint estimates for Riesz transforms of magnetic Schr\"odinger operators}''
Ark. Mat. \textbf{44} (2006), no. 2, 261--275.}

\bibitem{DK23}{L. Deleaval and C. Kriegler,
``\textit{Maximal H\"ormander functional calculus on $L^p$ spaces and UMD lattices}''
Int. Math. Res. Not. IMRN (2023), no. 6, 4643--4694.}

\bibitem{DK24}{L. Deleaval and C. Kriegler,
``\textit{$q$-variational H\"ormander functional calculus and Schr\"odinger and wave maximal estimates}''
Math. Z. \textbf{307} (2024), no. 2, Paper No. 25, 39 pp.}

\bibitem{FN19}{D. Frey and Z. Nieraeth,
``\textit{Weak and strong type $A_1$-$A_\infty$ estimates for sparsely dominated operators}''
J. Geom. Anal. \textbf{29} (2019), no. 1, 247--282.}

\bibitem{GSC05}{A. Grigor'yan and L. Saloff-Coste,
``\textit{Stability results for Harnack inequalities}''
Ann. Inst. Fourier (Grenoble) \textbf{55} (2005), no. 3, 825--890.}

\bibitem{H06}{M. Haase,
``\textit{The functional calculus for sectorial operators}''
Operator Theory: Advances and Applications, vol. 169, Birkh\"auser Verlag, Basel, 2006.}

\bibitem{HK12}{T. Hyt\"{o}nen and A. Kairema,
``\textit{Systems of dyadic cubes in a doubling metric space}''
Colloq. Math. \textbf{126} (2012), no. 1, 1--33.}

\bibitem{H12}{T. P. Hyt\"onen,
``\textit{The sharp weighted bound for general Calder\'on-Zygmund operators}''
Ann. of Math. (2) \textbf{175} (2012), no. 3, 1473--1506.}

\bibitem{HNVW16}{T. Hyt\"onen, J. van Neerven, M. Veraar and L. Weis,
``\textit{Analysis in Banach spaces. Vol. I. Martingales and Littlewood-Paley theory}''
Ergebnisse der Mathematik und ihrer Grenzgebiete (3), vol. 63, Springer, Cham, 2016.}

\bibitem{JSW08}{R. L. Jones, A. Seeger and J. Wright,
``\textit{Strong variational and jump inequalities in harmonic analysis}''
Trans. Amer. Math. Soc. \textbf{360} (2008), no. 12, 6711--6742.}

\bibitem{L76}{D. L\'epingle,
``\textit{La variation d'ordre $p$ des semi-martingales}''
Z. Wahrscheinlichkeitstheorie und Verw. Gebiete \textbf{36} (1976), no. 4, 295--316.}

\bibitem{LSV02}{V. Liskevich, Z. Sobol and H. Vogt,
``\textit{On the $L^p$-theory of $C_0$-semigroups associated with second-order elliptic operators. II}''
J. Funct. Anal. \textbf{193} (2002), no. 1, 55--76.}

\bibitem{L04}{S. Lee,
``\textit{Improved bounds for Bochner-Riesz and maximal Bochner-Riesz operators}''
Duke Math. J. \textbf{122} (2004), no. 1, 205--232.}

\bibitem{LSW10}{D. Lenz, P. Stollmann and D. Wingert,
``\textit{Compactness of Schr\"odinger semigroups}''
Math. Nachr. \textbf{283} (2010), no. 1, 94--103.}

\bibitem{LMX12}{C. Le Merdy and Q. Xu,
``\textit{Strong $q$-variation inequalities for analytic semigroups}''
Ann. Inst. Fourier (Grenoble) \textbf{62} (2012), no. 6, 2069--2097.}

\bibitem{LRS12}{S. Lee, K. M. Rogers and A. Seeger,
``\textit{Improved bounds for Stein's square functions}''
Proc. Lond. Math. Soc. (3) \textbf{104} (2012), no. 6, 1198--1234.}

\bibitem{L13}{A. K. Lerner,
``\textit{On an estimate of Calder\'on-Zygmund operators by dyadic positive operators}''
J. Anal. Math. \textbf{121} (2013), 141--161.}

\bibitem{LM13}{A. K. Lerner and K. Moen,
``\textit{Mixed $A_p-A_\infty$ estimates with one supremum}''
Studia Math. \textbf{219} (2013), no. 3, 247--267.}

\bibitem{LRS14}{S. Lee, K. M. Rogers and A. Seeger,
``\textit{Square functions and maximal operators associated with radial Fourier multipliers}''
in \textit{Advances in analysis: the legacy of Elias M. Stein}, Princeton Math. Ser., vol. 50, Princeton University Press, Princeton, NJ, 2014, pp. 273--302.}

\bibitem{LM16}{M. T. Lacey and D. Mena Arias,
``\textit{The sparse $T1$ theorem}''
Houston J. Math. \textbf{43} (2017), no. 1, 111--127.}

\bibitem{L17}{M. T. Lacey,
``\textit{Sparse bounds for spherical maximal functions}''
J. Anal. Math. \textbf{139} (2019), no. 2, 613--635.}

\bibitem{M80}{A. Miyachi,
``\textit{On some Fourier multipliers for $H^p(\mathbb{R}^n)$}''
J. Fac. Sci. Univ. Tokyo Sect. IA Math. \textbf{27} (1980), no. 1, 157--179.}

\bibitem{M81}{A. Miyachi,
``\textit{On some singular Fourier multipliers}''
J. Fac. Sci. Univ. Tokyo Sect. IA Math. \textbf{28} (1981), no. 2, 267--315.}

\bibitem{M86}{A. McIntosh,
``\textit{Operators which have an $H^{\infty}$ functional calculus}''
in \textit{Miniconference on operator theory and partial differential equations (North Ryde, 1986)}, Proc. Centre Math. Anal. Austral. Nat. Univ., vol. 14, Australian National University, Canberra, 1986, pp. 210--231.}

\bibitem{M88}{J. Marschall,
``\textit{Pseudodifferential operators with coefficients in Sobolev spaces}''
Trans. Amer. Math. Soc. \textbf{307} (1988), no. 1, 335--361.}

\bibitem{O05}{E. M. Ouhabaz,
``\textit{Analysis of heat equations on domains}''
London Mathematical Society Monographs Series, vol. 31, Princeton University Press, Princeton, NJ, 2005.}

\bibitem{S70}{E. M. Stein,
``\textit{Topics in harmonic analysis related to the Littlewood-Paley theory}''
Annals of Mathematics Studies, vol. 63, Princeton University Press, Princeton, NJ, 1970.}

\bibitem{S76}{E. M. Stein,
``\textit{Maximal functions. I. Spherical means}''
Proc. Nat. Acad. Sci. U.S.A. \textbf{73} (1976), no. 7, 2174--2175.}

\bibitem{SV94}{G. Schreieck and J. Voigt,
``\textit{Stability of the $L^p$-spectrum of generalized Schr\"odinger operators with form small negative part of the potential}''
in \textit{Functional analysis (Essen, 1991)}, Lecture Notes in Pure and Appl. Math., vol. 150, Dekker, New York, 1994, pp. 95--105.}

\bibitem{SYY14}{A. Sikora, L. Yan and X. Yao,
``\textit{Sharp spectral multipliers for operators satisfying generalized Gaussian estimates}''
J. Funct. Anal. \textbf{266} (2014), no. 1, 368--409.}

\bibitem{T02}{T. Tao,
``\textit{On the maximal Bochner-Riesz conjecture in the plane for $p<2$}''
Trans. Amer. Math. Soc. \textbf{354} (2002), no. 5, 1947--1959.}

\bibitem{T14}{G. Teschl,
``\textit{Mathematical methods in quantum mechanics. With applications to Schr\"odinger operators}''
2nd ed., Graduate Studies in Mathematics, vol. 157, American Mathematical Society, Providence, RI, 2014.}

\end{sloppypar}
\end{thebibliography}
\end{document}